\documentclass[reqno]{amsart}
\usepackage{amssymb, latexsym}
\usepackage{hyperref}

\allowdisplaybreaks

\let\Re=\undefined\DeclareMathOperator*{\Re}{Re}

\theoremstyle{plain}
\newtheorem{theorem}{Theorem}
\newtheorem{proposition}[theorem]{Proposition}
\newtheorem{lemma}[theorem]{Lemma}

\theoremstyle{definition}

\newtheorem{remark}[theorem]{Remark}

\newcounter{smalllist}

\numberwithin{equation}{section} \numberwithin{theorem}{section}
\begin{document}

\title[Global wellposedness of the 3-D full water wave problem]{Global wellposedness of the 3-D full water wave problem}
\author{Sijue Wu
}
\address{University of Michigan, Ann Arbor, MI}

\thanks {Financial support in part by NSF grant DMS-0800194}

\begin{abstract}
We consider the problem of global in time existence and uniqueness of solutions of the 3-D infinite depth full water wave problem. We show that the nature of the nonlinearity of the water wave equation is essentially of cubic and higher orders. For any initial interface that is sufficiently small in its steepness and velocity, we show that there exists a unique smooth solution of the full water wave problem for all time, and the solution decays at the rate $1/t$. 
\end{abstract}

\maketitle

\baselineskip18pt

\section{Introduction}
In this paper we continue our study of the global in time behaviors of the full water wave problem. 

The mathematical problem of $n$-dimensional water wave concerns the
motion of the 
interface separating an inviscid, incompressible, irrotational fluid,
under the influence of gravity, 
from a region of zero density (i.e. air) in 
$n$-dimensional space. It is assumed that the fluid region is below the
air region. Assume that
the density  of the fluid is $1$, the gravitational field is
$-{\bold k}$, where ${\bold k}$ is the unit vector pointing in the  upward
vertical direction, and at  
 time $t\ge 0$, the free interface is $\Sigma(t)$, and the fluid
occupies  region
$\Omega(t)$. When surface tension is
zero, the motion of the fluid is  described by 
\begin{equation}\label{euler}
\begin{cases}   \ \bold v_t + \bold v\cdot \nabla \bold v = -\bold k-\nabla P
\qquad  \text{on } \Omega(t),\ t\ge 0,
\\
\ \text{div}\,\bold v=0 , \qquad \text{curl}\,\bold v=0, \qquad  \text{on }
\Omega(t),\ t\ge 
0,
\\  
\ P=0, \qquad\text{on }
\Sigma(t) \\ 
\ (1, \bold v) \text{ 
is tangent to   
the free surface } (t, \Sigma(t)),
\end{cases}
\end{equation}
where $ \bold v$ is the fluid velocity, $P$ is the fluid
pressure. 
It is well-known that when surface tension is neglected, the water wave motion can be subject to the Taylor instability \cite{bi, ta, bhl}.  Assume that the free interface $\Sigma(t)$ is described by $
\xi=\xi(\alpha, t)$, 
where $\alpha\in R^{n-1}$ is the Lagrangian 
coordinate, i.e.
$\xi_t(\alpha,t)=\bold v(z(\alpha, t),t)$
is the fluid velocity on the interface, $\xi_{tt}(\alpha, t)=(\bold v_t+\bold v\cdot\nabla \bold v)(z(\alpha,t), t)$ is the
acceleration. Let 
 $\text{\bf n}$ be the unit normal pointing out of  
$\Omega(t)$. The Taylor sign condition relating to
Taylor instability is 
\begin{equation}\label{taylor}
-\frac{\partial P}{\partial \text{\bf n}}=(\xi_{tt}+{\bold k})\cdot \text{\bf n}\ge c_0>0,
\end{equation}
point-wisely on the interface for some positive constant $c_0$. 
In previous works \cite{wu1, wu2}, we showed that  the Taylor sign condition \eqref{taylor} 
always holds for the
$n$-dimensional infinite depth water wave problem \eqref{euler}, $n\ge 2$, as long as the interface is non-self-intersecting; and the initial value problem of
the  water
wave system \eqref{euler} is uniquely 
solvable {\bf locally} in   
time in 
Sobolev spaces for arbitrary given data.  
Earlier work includes Nalimov
\cite{na},   
 and Yosihara \cite{yo} on local existence and uniqueness for small data in 2D. 
We  mention the following recent work on local wellposedness \cite{am, cl, cs, ig1, la, li, ot, sz, zz}.  However the global in time behavior of the solutions remained open until last year.

In \cite{wu3}, we showed that for the 2D full water wave problem \eqref{euler} ($n=2$), the quantities $\Theta=(I-\frak H)y$, $(I-\frak H)\psi$,   under an appropriate coordinate change $k=k(\alpha,t)$, satisfy  equations of the type
\begin{equation}\label{cubic}
\partial_t^2\Theta-i\partial_\alpha\Theta=G
\end{equation}
with $G$ consisting of nonlinear terms of only cubic and higher orders. Here $\frak H$ is the Hilbert transform related to the water region $\Omega(t)$,
$y$ is the height function for the interface $\Sigma(t): (x(\alpha, t),y(\alpha, t))$, and $\psi$ is the trace on $\Sigma(t)$ of the velocity potential. Using this favorable structure, and the $L^\infty$ time decay rate for the 2D water wave  $1/{t^{1/2}}$, we showed that the full water wave equation \eqref{euler} in two space dimensions has a unique smooth solution for a time period $[0, e^{c/\epsilon}]$ for initial data  $\epsilon\Phi$, where $\Phi$ is arbitrary, $c$ depends only on $\Phi$,  and $\epsilon$ is sufficiently small.

Briefly, the structural advantage of \eqref{cubic} can be explained as the following. We know the water wave 
equation \eqref{euler} is equivalent to an equation on the interface of the form
\begin{equation}\label{linearized}
\partial_t^2u+|D|u=\text{nonlinear terms}
\end{equation}
where the nonlinear terms contain quadratic nonlinearity.
For given smooth data, the free equation $\partial_t^2u+|D|u=0$ has a unique solution globally in time, with $L^\infty$ norm decays at the  rate $1/t^{\frac{n-1}2}$. However the nonlinear interaction can cause blow-up at finite time. The weaker the nonlinear interaction, the longer the solution stays smooth. For small data, 
quadratic interactions are in general stronger than the cubic and higher order interactions. In \eqref{cubic} there is no quadratic terms, using it we are able to prove a longer time existence of classical solutions for small initial data.

Naturally, we would like to know if the 3D water wave equation also posses such special structures. 
We find that indeed this is the case. A natural setting for 3D to utilize the ideas of 2D is the Clifford analysis. However deriving such equations \eqref{cubic} in 3D in the Clifford Algebra framework is not straightforward due to  the non-availability of the Riemann mapping, the non-commutativity of the Clifford numbers, and the fact that the multiplication of two Clifford analytic functions is not necessarily analytic. 
Nevertheless we have overcome these difficulties. 

Let  $\Sigma(t): \xi=(x(\alpha,\beta, t),y(\alpha, \beta, t), z(\alpha, \beta, t))$ be the interface in Lagrangian coordinates $(\alpha, \beta)\in R^2$,
and let $\frak H$ be the Hilbert transform associated to the water region $\Omega(t)$, 
$N=\xi_\alpha\times \xi_\beta$ be the outward normal. In this paper, we show that  the quantity $\theta=(I-\frak H)z$ satisfies such equation
\begin{equation}\label{cubic'}
\partial_t^2\theta- \frak a N\times \nabla \theta =G
\end{equation}
where $G$ is a nonlinearity of cubic and higher orders in nature. 
We also find a coordinate change $k$  that transforms \eqref{cubic'} into an equation consisting of a linear part plus only
cubic and higher order nonlinear terms.\footnote{We will explain more precisely the meaning of these statements in subsection 1.2.} For $\psi$ the trace of the velocity potential, $(I-\frak H)\psi$ also satisfies a similar type equation. However we will not derive it  since we do not need it in this paper.

Given that in 3D the $L^\infty$ time decay rate is a faster $1/t$, it is not surprising that for small data, the water wave equation \eqref{euler} ($n=3$) is solvable globally in time. In fact we obtain better results than in 2D in terms of the initial data set. We show that if the steepness of the initial interface  and the  velocity along the initial interface (and finitely many of their derivatives) are sufficiently  small, then the solution of the 3D full water wave equation \eqref{euler} remains smooth for all time and decays at a $L^\infty$ rate of $1/t$. No  assumptions are made to the height of the initial interface and the velocity field in the fluid domain.  
In particular, this means that the amplitude of the initial interface can be arbitrary large, the initial kinetic energy $\|\bold v\|_{L^2(\Omega(0))}^2$ can be infinite.
This certainly makes sense physically. We note that the almost global wellposedness result we obtained for 2D water wave \cite{wu3} requires the initial amplitude of the interface and the initial kinetic energy $\|\bold v\|_{L^2(\Omega(0))}^2$
being small. One may view 2D water wave as a special case of 3D where the wave is constant in one direction. In 2D there is one less direction for the wave to disperse and the $L^\infty$ time decay rate is a slower $1/{t^{1/2}}$. Technically our proof of the almost global wellposedness result in 2D \cite{wu3} used to the full extend the decay rate and required the smallness in the amplitude and kinetic energy since we needed to control the derivatives in the full range. One may think the assumption on the smallness in amplitude and kinetic energy is to compensate the lack of decay in one direction. However this is merely a  technical reason. In 3D assuming the wave  tends to zero at spatial infinity, we have a
faster $L^\infty$ time decay rate $1/t$. This allows us a less elaborate proof and a global wellposedness result with less assumptions on the initial data.

Recently, Germain, Masmoudi, Shatah \cite{gms} studied the global existence of the 3D water wave through analyzing the space-time resonance, for initial interfaces that are small in their amplitude and half derivative of the trace of the velocity potential on the interface, as well as  finitely many of their derivatives (this implies smallness in velocity and steepness as well, since velocity and steepness are derivatives of the velocity potential and the height.) In particular, they assume among other things that the half derivative of the trace of the initial velocity potential  $|D|^{1/2}\psi_0\in L^2((1+|x|^2)\,dx)\cap L^1(dx)\cap H^N(dx)$ for some large $N$.
Here $H^N$ is the $L^2$ Sobolev space with $N$ derivatives. We know $|D|^{1/2}\psi_0\in  L^1(dx)$ implies $\psi_0\in L^{4/3}(dx)$ (see \cite{sgw} p.119, Theorem 1). This together with the assumption that $|D|^{1/2}\psi_0\in H^N(dx)$  implies that the trace of the initial velocity potential $\psi_0$ decays at  infinity. This is equivalent to assuming the line integral of the initial velocity field along any curve on the interface from infinity to infinity is zero. This is a moment condition which in general doesn't hold. It puts the data set studied by \cite{gms} into a lower dimensional subset of the data set we consider in this paper. 
Moreover
we know $|D|^{1/2}\psi_0\in L^2(dx)$ is equivalent to the finiteness of the kinetic energy, $\|\bold v\|^2_{L^2(\Omega(0))}<\infty$.  $|D|^{1/2}\psi_0\in L^2((1+|x|^2)\,dx)$ would require more than the initial kinetic energy being finite. In fact one may estimate the decay rate of $\psi_0$ at infinity as follows.  
We know (\cite{sgw} p.117)
\begin{equation*}
\begin{aligned}
c\psi_0(x)&=\int\frac1{|x-y|^{2-1/2}}|D|^{1/2}\psi_0(y)\,dy\\&
=(\int_{|y|\le\frac 12|x|}+\int_{|y|\ge 2|x|}+\int_{\frac 12|x|\le |y|\le 2|x|})\frac1{|x-y|^{2-1/2}}|D|^{1/2}\psi_0(y)\,dy
\end{aligned}
\end{equation*}
where
\begin{equation*}
\begin{aligned}
|(\int_{|y|\le\frac 12|x|}+&\int_{|y|\ge 2|x|})\frac1{|x-y|^{2-1/2}}|D|^{1/2}\psi_0(y)\,dy|\\&\lesssim \frac1{|x|^{3/2}}(\||D|^{1/2}\psi_0\|_{L^1(dx)}+\||D|^{1/2}\psi_0\|_{L^2(|x|^2dx)})
\end{aligned}
\end{equation*}
and
\begin{equation*}
\begin{aligned}
\int|&\int_{\frac 12|x|\le |y|\le 2|x|}\frac1{|x-y|^{2-1/2}}|D|^{1/2}\psi_0(y)\,dy|^2|x|\,dx\\&\le\int(\int_{\frac 12|x|\le |y|\le 2|x|}\frac1{|x-y|^{3/2}}\,dy)(\int_{\frac 12|x|\le |y|\le 2|x|}\frac1{|x-y|^{3/2}}||D|^{1/2}\psi_0(y)|^2\,dy)|x|\,dx\\&\lesssim
\int(\int_{\frac12|y|\le |x|\le 2|y|}\frac{|x|^{3/2}}{|x-y|^{3/2}}\,dx) ||D|^{1/2}\psi_0(y)|^2\,dy\lesssim \int |y|^2||D|^{1/2}\psi_0(y)|^2\,dy
\end{aligned}
\end{equation*}
So if  $\psi_0$ satisfies $|D|^{1/2}\psi_0\in L^1(dx)\cap L^2(|x|^2dx)$ as assumed in \cite{gms},  it is necessary then that $\psi_0(x)$ decays at a rate no slower than $1/{|x|^{3/2}}$ as $|x|\to\infty$.

\subsection{Notations and Clifford analysis}

We study the 3D water wave problem  in the setting of the Clifford Algebra $\mathcal C(V_2)$, i.e. the algebra of quaternions.  We refer  to \cite{gm} for an in depth 
discussion of Clifford analysis.

Let $\{1, e_1, e_2, e_3\}$ be the basis of $\mathcal C(V_2)$ satisfying
\begin{equation}\label{product}
e^2_i=-1,  \quad e_ie_j=-e_je_i, \quad i, j=1,2,3, \ i\ne j,\qquad e_3=e_1e_2.
\end{equation}
An element $\sigma\in \mathcal C(V_2)$ has a unique representation
$\sigma=\sigma_0+\sum_{i=1}^3 \sigma_i e_i$, with $\sigma_i\in \mathbb R$ for $0\le i\le  3$. We call $\sigma_0$ the real part of $\sigma$ and denote it by $\Re \sigma$ and $\sum_{i=1}^3 \sigma_i e_i$ the vector part of $\sigma$. We call $\sigma_i$ the $e_i$ component of $\sigma$. We denote $\overline \sigma=e_3\sigma e_3$, $|\sigma|^2=\sum_{i=0}^3 \sigma_i^2$. If not specified, we always assume in such an expression $\sigma=\sigma_0+\sum_{i=1}^3 \sigma_i e_i$   that $\sigma_i\in \mathbb R$, for $0\le i\le 3$. We define $\sigma\cdot\xi=\sum_{j=0}^3\sigma_j\xi_j$.
We call $\sigma\in \mathcal C(V_2)$ a vector if $\Re \sigma=0$.
We identify a point or vector $\xi=(x,y,z)\in \mathbb R^3$ with its $\mathcal C(V_2)$ counterpart $\xi=xe_1+ye_2+ze_3$.  
 For vectors $\xi ,\ \eta \in \mathcal C(V_2)$, we know
\begin{equation}\label{vector1}
\xi\eta=-\xi\cdot\eta+\xi\times\eta,
\end{equation}
where $\xi\cdot\eta$ is the dot product, $\xi\times\eta$ the cross product. For vectors $\xi$, $\zeta$, $\eta$, $\xi(\zeta\times \eta)$ is obtained by first finding the cross product $\zeta\times \eta$, then regard it as a Clifford vector and calculating its multiplication with 
$\xi$ by the rule \eqref{product}. We write $\mathcal D=\partial_xe_1+\partial_y e_2+\partial_z e_3$, $\nabla=(\partial_x,\partial_y,\partial_z)$. At times we also use the notation $\xi=(\xi_1, \xi_2, \xi_3)$ to indicate a point in $\mathbb R^3$. In this case $\nabla=(\partial_{\xi_1},\partial_{\xi_2},\partial_{\xi_3})$, $\mathcal D=\partial_{\xi_1}e_1+\partial_{\xi_2}e_2+\partial_{\xi_3}e_3$.

Let $\Omega$ be an unbounded\footnote{Similar definitions and results exist for bounded domains, see \cite{gm}. For the purpose of this paper, we discuss only for  unbounded domain $\Omega$.}
$C^2$ domain in $\mathbb R^3$, $\Sigma=\partial\Omega$ be its boundary and $\Omega^c$ be its complement. A $\mathcal C(V_2)$ valued function $F$ is Clifford analytic in $\Omega$ if $\mathcal DF=0$ in $\Omega$. Let
\begin{equation}
\Gamma(\xi)=-\frac1{\omega_3}\frac1{ |\xi|},
\qquad K(\xi)=-2\mathcal D \Gamma(\xi)
=-\frac 2{\omega_3} \frac{ \xi}{|\xi|^{3}},
\quad\qquad \text{for }\xi=\sum_1^3\xi_ie_i,
\end{equation}
where $\omega_3$ is the surface area of the unit sphere in $\mathbb R^3$.
Let $\xi=\xi(\alpha,\beta)$, $(\alpha,\beta)\in \mathbb R^2$ be a parameterization of $\Sigma$ with $N=\xi_\alpha\times\xi_\beta$ pointing out of $\Omega$.  The Hilbert transform associated to the parameterization $\xi=\xi(\alpha,\beta)$, $(\alpha,\beta)\in \mathbb R^2$  
is defined by
\begin{equation}
 \frak H_{\Sigma} f(\alpha,\beta )=p.v. \iint_{\mathbb R^2} K(\xi(\alpha',\beta') -\xi(\alpha,  \beta))\,(\xi'_{\alpha'}\times\xi'_{\beta'})\,f( \alpha',\beta')\,d\alpha'd\beta'.
\end{equation}
We know a $\mathcal C(V_2)$ valued function $F$  that decays at infinity is Clifford analytic in $\Omega$ if and only if its trace on $\Sigma$: $f(\alpha,\beta)=F(\xi(\alpha,\beta))$ satisfies
\begin{equation}\label{hilbert}
f=\frak H_{\Sigma} f.
\end{equation}
We know
$\frak H_\Sigma^2=I$ in $L^2$. We use the convention $\frak H_\Sigma1=0$. 
We abbreviate 
\begin{align*} \frak H_{\Sigma}f(\alpha,\beta)&
=\iint K(\xi(\alpha',\beta') -\xi(\alpha,  \beta))\,(\xi'_{\alpha'}\times\xi'_{\beta'})\,f( \alpha',\beta'
)\,d\alpha'd\beta'
\\&= \iint K(\xi' -\xi)\,(\xi'_{\alpha'}\times\xi'_{\beta'})\,f'\,d\alpha'd\beta'= \iint K\,N'\,f'\,d\alpha'd\beta'.
\end{align*}

Let $f=f(\alpha,\beta)$ be defined for $(\alpha,\beta)\in \mathbb R^2$. We say $f^\hbar$ is the harmonic extension of $f$ to $\Omega$ if $\Delta f^{\hbar}=0$ in $\Omega$ and $f^\hbar(\xi(\alpha,\beta))=f(\alpha,\beta)$. We denote by $\mathcal D_\xi f$ the trace of $\mathcal D f^{\hbar}$ on $\Sigma$, i.e.
\begin{equation}\label{hext}
\mathcal D_\xi f(\alpha,\beta)=\mathcal D f^\hbar (\xi(\alpha,\beta)).
\end{equation}
Similarly $\nabla_\xi f(\alpha,\beta)=\nabla f^\hbar(\xi(\alpha,\beta))$, $\partial_xf(\alpha,\beta)=\partial_x f^{\hbar}(\xi(\alpha,\beta))$ etc. In the context of water wave where
$\Omega(t)$ is the fluid domain, we denote by 
 $\nabla_\xi^+ f$ (respectively $\nabla_\xi^-f$) the trace of $\nabla f^\hbar$ on $\Sigma(t)$, where $f^\hbar$ is the harmonic extension of $f$ to $\Omega(t)$ (respectively $\Omega(t)^c$).
We have
\begin{lemma}\label{lemma 1.1} 1.
Let $f=f(\alpha,\beta)$, $(\alpha,\beta)\in \mathbb R^2$ be a real valued smooth function decays fast at infinity. We have
\begin{equation}\label{basicid}
\iint K(\xi(\alpha',\beta')-\xi(\alpha,\beta)) \cdot (N'\times \nabla' f)(\alpha',\beta')\,d\alpha'd\beta'=0.
\end{equation}
2. For any function $f=\sum_1^3 f_i e_i$ satisfying $f=\frak H_{\Sigma}f$ or $f=-\frak H_{\Sigma}f$, we have
\begin{equation}\label{analytic1}
\xi_\beta\cdot \partial_\alpha f-\xi_\alpha\cdot \partial_\beta f=0.
\end{equation}
\end{lemma}
\begin{proof} 
Let $f^{\hbar}$ be the harmonic extension of $f$ to the domain $\Omega$. We know $\mathcal D f^\hbar$ is Clifford analytic in $\Omega$. Therefore the trace of $\mathcal D f^\hbar$ on $\Sigma$ satisfies
\begin{equation}\label{id}
\mathcal D_\xi f= \frak H_{\Sigma} \mathcal D_\xi f. 
\end{equation}
 Taking the real part of \eqref{id} gives us \eqref{basicid}.

For \eqref{analytic1}, we only prove for the case $f=\frak H_{\Sigma} f$. The proof for the case $f=-\frak H_{\Sigma}f$ is similar,
since $f=-\frak H_{\Sigma}f$ is equivalent to the harmonic extension of $f$ to $\Omega^c$ being analytic. 

We have from  $f=\frak H_{\Sigma} f$ that $\mathcal D_\xi f=0$. Therefore
$$\xi_\beta\cdot \partial_\alpha f-\xi_\alpha\cdot \partial_\beta f=\sum_{i,j=1}^3\partial_\beta\xi_{i}\partial_\alpha\xi_{j}\partial_{\xi_j}f_i-\sum_{i,j=1}^3\partial_\alpha\xi_{j}\partial_\beta\xi_{i}\partial_{\xi_i}f_j=0.$$
\end{proof}

Assume that for each $t\in [0,T]$, $\Omega(t)$ is a $C^2$ domain with boundary $\Sigma(t)$. Let $\Sigma(t): \xi=\xi(\alpha,\beta,t)$, $(\alpha,\beta)\in\mathbb R^2$; $\xi\in C^2(\mathbb R^2\times [0, T])$, $N=\xi_\alpha\times\xi_\beta$.  We know $N\times\nabla=\xi_\beta\partial_\alpha-\xi_\alpha\partial_\beta$. Denote $[A,B]=AB-BA$.  We have 
\begin{lemma}\label{lemma 1.2}
 Let $f\in C^1(\mathbb R^2\times [0,T])$ be a $\mathcal C(V_2)$ valued function vanishing at spatial infinity, and $\frak a$ be real valued. Then
 \begin{gather}
[\partial_t, \frak H_{\Sigma(t)}]f=\iint K(\xi'-\xi)\,(\xi_{t}-\xi'_t)\times
(\xi'_{\beta'}\partial_{\alpha'}-\xi'_{\alpha'}\partial_{\beta'})f'\,d\alpha'd\beta'.\label{commuteth}\\
 [\partial_\alpha, \frak H_{\Sigma(t)}]f=\iint
K(\xi'-\xi)\,(\xi_{\alpha}-\xi'_{\alpha'})\times 
(\xi'_{\beta'}\partial_{\alpha'}-\xi'_{\alpha'}\partial_{\beta'})f'\,d\alpha'd\beta'\label{commuteah}\\
[\partial_\beta, \frak H_{\Sigma(t)}]f=\iint
K(\xi'-\xi)\,(\xi_{\beta}-\xi'_{\beta'})\times 
(\xi'_{\beta'}\partial_{\alpha'}-\xi'_{\alpha'}\partial_{\beta'})f'\,d\alpha'd\beta'\label{commutebh}\\
[\frak a N\times \nabla, \frak H_{\Sigma(t)}]f=\iint
K(\xi'-\xi)\,(\frak aN-\frak a'N')\times 
(\xi'_{\beta'}\partial_{\alpha'}-\xi'_{\alpha'}\partial_{\beta'})f'\,d\alpha'd\beta'\label{commutenh}
\end{gather}
\begin{align}
[\partial_t^2, \frak H_{\Sigma(t)}]f=\iint K(\xi'-\xi)\,(\xi_{tt}-\xi'_{tt})\times
(\xi'_{\beta'}\partial_{\alpha'}-\xi'_{\alpha'}\partial_{\beta'})f'\,d\alpha'd\beta'\tag*{}\\
+\iint
K(\xi'-\xi)\,(  \xi_{t}-\xi'_t  )\times 
(\xi'_{t\beta'}\partial_{\alpha'}-\xi'_{t\alpha'}\partial_{\beta'})f'\,d\alpha'd\beta'\label{commutetth}\\
+ \iint \partial_t K(\xi'-\xi)\,(\xi_{t}-\xi'_t)\times
(\xi'_{\beta'}\partial_{\alpha'}-\xi'_{\alpha'}\partial_{\beta'})f'\,d\alpha'd\beta'\tag*{}\\
+ 2\iint K(\xi'-\xi)\,(\xi_{t}-\xi'_t)\times
(\xi'_{\beta'}\partial_{\alpha'}-\xi'_{\alpha'}\partial_{\beta'})f'_t\,d\alpha'd\beta'\tag*{}
\end{align}

\end{lemma}

\begin{proof}
Applying  Lemma 3.1 of \cite{wu2} component-wisely to $f$ gives 
\eqref{commuteth}, \eqref{commuteah}, \eqref{commutebh}.  \eqref{commutetth} is a direct consequence of \eqref{commuteth} and the fact $[\partial^2_t, \frak H_{\Sigma(t)}]=\partial_t[\partial_t, \frak H_{\Sigma(t)}]+[\partial_t, \frak H_{\Sigma(t)}]\partial_t$. We now prove \eqref{commutenh}  for $f$ real valued. 
Notice that $\frak a N\times \nabla \frak H_{\Sigma(t)} f=\frak a\xi_\beta\partial_\alpha\frak H_{\Sigma(t)}f-\frak a\xi_\alpha\partial_\beta\frak H_{\Sigma(t)}f$. From \eqref{commuteah}, we have
\begin{align}\frak a\xi_\beta&\partial_\alpha\frak H_{\Sigma(t)}f=\frak a\xi_\beta[\partial_\alpha, \frak H_{\Sigma(t)}]f+\frak a\xi_\beta\frak H_{\Sigma(t)}\partial_\alpha f\tag*{}\\&=\iint \frak a\xi_\beta
K(\xi'-\xi)\,(\xi_{\alpha}-\xi'_{\alpha'})\times 
(\xi'_{\beta'}\partial_{\alpha'}-\xi'_{\alpha'}\partial_{\beta'})f'\,d\alpha'd\beta'+\frak a\xi_\beta\iint K N'\partial_{\alpha'}f'\,d\alpha'\,d\beta'\tag*{}\\&=
\iint \frak a\xi_\beta
K(\xi'-\xi)\,\xi_{\alpha}\times 
(\xi'_{\beta'}\partial_{\alpha'}-\xi'_{\alpha'}\partial_{\beta'})f'\,d\alpha'd\beta'\label{11}
\end{align}
Similarly
\begin{equation}\label{12}
\frak a\xi_\alpha\partial_\beta\frak H_{\Sigma(t)}f=\iint \frak a\xi_\alpha
K(\xi'-\xi)\,\xi_{\beta}\times 
(\xi'_{\beta'}\partial_{\alpha'}-\xi'_{\alpha'}\partial_{\beta'})f'\,d\alpha'd\beta'
\end{equation}
Now for any vectors $K$, $\eta$,
\begin{equation}\label{13}
\begin{aligned}
&\xi_\beta
K\,\xi_{\alpha}\times\eta-\xi_\alpha
K\,\xi_{\beta}\times \eta\\
&= -K\xi_\beta\,\xi_{\alpha}\times\eta+ K\xi_\alpha\,\xi_{\beta}\times \eta-
2\xi_\beta\cdot
K\,\xi_{\alpha}\times\eta+2\xi_\alpha\cdot
K\,\xi_{\beta}\times \eta\\&
=-K\xi_\beta\times(\xi_{\alpha}\times\eta)+ K\xi_\alpha\times(\xi_{\beta}\times \eta)\\&
+K\xi_\beta\cdot(\xi_{\alpha}\times\eta)- K\xi_\alpha\cdot(\xi_{\beta}\times \eta)-
2(K\times (\xi_\alpha\times\xi_\beta))\times\eta\\&=K (-\xi_\alpha \xi_\beta\cdot\eta+\xi_\beta\xi_\alpha\cdot \eta)-
2K(\xi_\alpha\times\xi_\beta)\cdot\eta+2K (\xi_\alpha\times\xi_\beta)\cdot\eta-2K\cdot\eta (\xi_\alpha\times\xi_\beta)\\&
=K (\xi_\alpha\times\xi_\beta)\times\eta-2K\cdot\eta (\xi_\alpha\times\xi_\beta)
\end{aligned}
\end{equation}
In the above calculation, we  used repeatedly the  identities \eqref{vector1} and $a\times( b\times c)=b\,a\cdot c-c \,a\cdot b$.
Combining \eqref{11}, \eqref{12} and applying  \eqref{basicid} and \eqref{13} with $\eta=(\xi'_{\beta'}\partial_{\alpha'}-\xi'_{\alpha'}\partial_{\beta'})f'=N'\times\nabla'f'$, we get
\begin{equation*}
\frak a N\times \nabla \frak H_{\Sigma(t)} f= \iint K (\xi'-\xi)\frak a (\xi_\alpha\times\xi_\beta)\times (N'\times\nabla'f')\,d\alpha'\,d\beta'.
\end{equation*}
Notice that
$$\frak H_{\Sigma(t)} (\frak a N\times \nabla f)
=\iint K (\xi'-\xi)\frak a' N'\times (N'\times\nabla'f')\,d\alpha'\,d\beta'.$$
 \eqref{commutenh} therefore holds for real valued $f$. \eqref{commutenh} for $\mathcal C(V_2)$ valued $f$ directly follows.

\end{proof}

\subsection{The main equations and  main results}

We now discuss the 3D water wave. Let $\Sigma(t): \xi(\alpha,\beta,t)=x(\alpha,\beta,t)e_1+y(\alpha,\beta,t)e_2+z(\alpha,\beta,t) e_3 $, $(\alpha,\beta)\in \mathbb R^2$ be the parameterization of the interface at time $t$ in Lagrangian coordinates $(\alpha,\beta)$ with $N=\xi_\alpha\times\xi_\beta=(N_1, N_2, N_3)$ pointing out of the fluid domain $\Omega(t)$. Let $\frak H=\frak H_{\Sigma(t)}$, and 
$$\frak a=-\frac1{|N|}\frac{\partial P}{\partial \text{\bf n}}.$$
We know from \cite{wu2} that $\frak a>0$ and equation \eqref{euler} is equivalent to the following nonlinear system defined on the interface $\Sigma(t)$: 
\begin{align}
\xi_{tt}+e_3=\frak a N\label{ww1}\\
\xi_t=\frak H\xi_t\label{ww2}
\end{align}
Motivated by   \cite{wu3}, we would like to know whether in 3-D, the quantity $\pi=(I-\frak H)ze_3$  under an appropriate coordinate change satisfies an equation with nonlinearities containing no quadratic terms. We first derive the equation for $\pi$ in Lagrangian coordinates.  
\begin{proposition}\label{prop:cubic1} We have
\begin{equation}\label{cubic1}
\begin{aligned}
(\partial_t^2-\frak a N\times \nabla)\pi=\iint K(\xi'-\xi)\,(\xi_{t}-\xi'_t)\times
(\xi'_{\beta'}\partial_{\alpha'}-\xi'_{\alpha'}\partial_{\beta'})\overline{\xi_t'}\,d\alpha'd\beta'\\
-\iint K(\xi'-\xi)\,(\xi_{t}-\xi'_t)\times
(\xi'_{t\beta'}\partial_{\alpha'}-\xi'_{t\alpha'}\partial_{\beta'})z'\,d\alpha'd\beta'e_3\\
-\iint \partial_t K(\xi'-\xi)\,(\xi_{t}-\xi'_t)\times
(\xi'_{\beta'}\partial_{\alpha'}-\xi'_{\alpha'}\partial_{\beta'})z'\,d\alpha'd\beta'e_3
\end{aligned}
\end{equation}
\end{proposition}

\begin{proof}
Notice 
from \eqref{ww1}
\begin{equation}\label{112}
(\partial_t^2-\frak a N\times \nabla)z e_3=z_{tt}e_3 +\frak aN_1e_1+\frak aN_2e_2=\xi_{tt}
\end{equation}
and from \eqref{ww2}  that
\begin{equation}\label{111}
(I-\frak H) \xi_{tt}=[\partial_t, \frak H]\xi_t
\end{equation}
\eqref{cubic1} is an easy consequence of  \eqref{commuteth}, \eqref{commutenh} and \eqref{commutetth} and \eqref{ww1}, \eqref{112}, \eqref{111}:
\begin{equation*}
\begin{aligned}
(\partial_t^2-\frak a N\times \nabla)\pi&=(I-\frak H)(\partial_t^2-\frak a N\times \nabla)ze_3-[\partial_t^2-\frak a N\times \nabla, \frak H]ze_3\\
&=[\partial_t, \frak H]\xi_t-[\partial_t^2-\frak a N\times \nabla, \frak H]ze_3\\
&=\iint K(\xi'-\xi)\,(\xi_{t}-\xi'_t)\times
(\xi'_{\beta'}\partial_{\alpha'}-\xi'_{\alpha'}\partial_{\beta'})\overline{\xi_t'}\,d\alpha'd\beta'\\
&-\iint K(\xi'-\xi)\,(\xi_{t}-\xi'_t)\times
(\xi'_{t\beta'}\partial_{\alpha'}-\xi'_{t\alpha'}\partial_{\beta'})z'\,d\alpha'd\beta'e_3\\
&-\iint \partial_t K(\xi'-\xi)\,(\xi_{t}-\xi'_t)\times
(\xi'_{\beta'}\partial_{\alpha'}-\xi'_{\alpha'}\partial_{\beta'})z'\,d\alpha'd\beta'e_3
\end{aligned}
\end{equation*}

\end{proof}

We see that the second and third terms in the right hand side of \eqref{cubic1} are consisting of terms of cubic and higher orders, while the first term contains quadratic terms. Unlike the 2D case, multiplications of Clifford analytic functions are not necessarily analytic, so we cannot reduce the first term at the right hand side of equation \eqref{cubic1}  into a cubic form. However we notice that $\overline \xi_t=x_t e_1+y_t e_2-z_t e_3$ is almost analytic in the air region $\Omega(t)^c$, and this implies that the first term is  almost  analytic in the fluid domain $\Omega(t)$, or in other words, is almost of the type $(I+\frak H)Q$ in nature, here $Q$ means quadratic.\footnote{See Proposition~\ref{propon}, also see the derivation of the energy estimates in section~\ref{energy}.} Notice that the left hand side of \eqref{cubic1} is almost analytic in the air region, or of the type $(I-\frak H)$. The orthogonality of the projections 
$(I-\frak H)$ and $(I+\frak H)$ allows us to reduce the first term to cubic in energy estimates.

Notice that the left hand side of \eqref{cubic1} still contains quadratic terms and \eqref{cubic1}  is invariant under a change of coordinates. We now want to see  if in 3D, there is a coordinate change $k$, such that under which the left hand side of \eqref{cubic1} becomes a linear part plus only cubic and higher order terms. In 2D, such a coordinate change exists (see (2.18) in \cite{wu3}). However it is defined by the Riemann mapping. Although there is no Riemann mapping in 3D, we realize that the Riemann mapping used in 2-D is just a holomorphic function in the fluid region with its imaginary part equal to zero on $\Sigma(t)$. This motivates us to define
\begin{equation}\label{k}
k=k(\alpha,\beta,t)=\xi(\alpha,\beta,t)-(I+\frak H)z (\alpha,\beta, t)e_3+\frak K z (\alpha,\beta,t) e_3
\end{equation}
Here $\frak K=\Re \frak H$: 
\begin{equation}\label{doublelayer1}
\frak K f(\alpha,\beta, t)=-\iint K(\xi(\alpha',\beta',t)-\xi(\alpha,\beta, t))\cdot N' f(\alpha', \beta',t)\,d\alpha'\,d\beta'
\end{equation}
is the double layered potential operator. It is clear that the $e_3$ component of $k$ as defined in \eqref{k} is zero. In fact, the real part of $k$ is also zero.
This is because
\begin{equation*}
\begin{aligned}
&\iint K(\xi'-\xi) \times(\xi'_{\alpha'}\times \xi'_{\beta'}) z' e_3\,d\alpha'\,d\beta'=\iint (\xi'_{\alpha'} \xi'_{\beta'}\cdot K-\xi'_{\beta'}\xi'_{\alpha'}\cdot K)z'e_3\,d\alpha'\,d\beta' \\
&= -2\iint (\xi'_{\alpha'} \partial_{\beta'}\Gamma(\xi'-\xi)-\xi'_{\beta'}\partial_{\alpha'}\Gamma(\xi'-\xi))z'e_3\,d\alpha'\,d\beta'\\&
=2 \iint \Gamma(\xi'-\xi)(\xi'_{\alpha'} z_{\beta'}-\xi'_{\beta'}z_{\alpha'})e_3\,d\alpha'\,d\beta'
=2\iint \Gamma(\xi'-\xi)(N_1'e_1+N_2'e_2)\,d\alpha'\,d\beta'
\end{aligned}
\end{equation*}
So
\begin{equation}\label{hilbertz}
\frak H z e_3=\frak K z e_3+2\iint \Gamma(\xi'-\xi)(N_1'e_1+N_2'e_2)\,d\alpha'\,d\beta'
\end{equation}
This shows that the mapping $k$ defined in \eqref{k} has only the $e_1$ and $e_2$ components $k=(k_1, k_2)=k_1 e_1+k_2 e_2$. If $\Sigma(t)$ is a graph of small steepness, i.e. if $z_\alpha$ and $z_\beta$ are small, then the Jacobian of $k=k(\cdot,t)$: $J(k)=J(k(t))=\partial_\alpha k_1\partial_\beta k_2-
\partial_\alpha k_2\partial_\beta k_1>0$ and $k(\cdot,t) :\mathbb R^2\to \mathbb R^2$ defines a valid coordinate change. We will make this point  more precise in Lemma~\ref{diffeo}.

Denote $\nabla_\bot=(\partial_\alpha,\partial_\beta)$, $U_g f(\alpha,\beta, t)=f(g(\alpha,\beta, t), t)=f\circ g(\alpha,\beta, t)$. 
Assume that $k=k(\cdot,t):\mathbb R^2\to \mathbb R^2$ defined in \eqref{k} is a diffeomorphism satisfying $J(k(t))>0$. Let $k^{-1}$ be such that $k\circ k^{-1}(\alpha, \beta, t)=\alpha e_1+\beta e_2$. Define 
\begin{equation}\label{zeta}
\zeta=\xi\circ k^{-1}=\frak x e_1+\frak y e_2+\frak z e_3,\quad u= \xi_t\circ k^{-1},\quad \text{and}\quad w= \xi_{tt}\circ k^{-1} .
\end{equation}
Let 
\begin{equation}\label{bA}
b=k_t\circ k^{-1} \quad A\circ ke_3={\frak a}J(k)e_3=\frak a k_\alpha\times k_\beta, \quad\text{and}\quad \mathcal N=\zeta_\alpha\times\zeta_\beta.
\end{equation} 
By a simple  application of the chain rule, we have
\begin{equation}
U_k^{-1}\partial_t U_k=\partial_t +b\cdot \nabla_\bot , \quad\text{and}\quad U_k^{-1}(\frak{a}N\times \nabla)U_k=A\mathcal N\times \nabla=A(\zeta_\beta\partial_\alpha-\zeta_\alpha\partial_\beta),
\end{equation}
and $U_k^{-1}\frak H U_k =\mathcal H $, with 
\begin{equation}
\mathcal H f(\alpha,\beta,t)=\iint K(\zeta(\alpha',\beta',t)-\zeta(\alpha,\beta,t))(\zeta'_{\alpha'}\times\zeta'_{\beta'}) f(\alpha',\beta',t)\,d\alpha'\,d\beta'
\end{equation}
Let $\chi=\pi\circ k^{-1}$. Applying coordinate change $U_k^{-1}$ to equation \eqref{cubic1}. We get
\begin{equation}\label{cubic2}
\begin{aligned}
((\partial_t+ b\cdot \nabla_\bot)^2-A \mathcal N\times \nabla)\chi=\iint K(\zeta'-\zeta)\,(u-u')\times
(\zeta'_{\beta'}\partial_{\alpha'}-\zeta'_{\alpha'}\partial_{\beta'})\overline{u'}\,d\alpha'd\beta'\\
-\iint K(\zeta'-\zeta)\,(u-u')\times
(u'_{\beta'}\partial_{\alpha'}-u'_{\alpha'}\partial_{\beta'})\frak z'\,d\alpha'd\beta'e_3\\
-\iint ((u'-u)\cdot\nabla) K(\zeta'-\zeta)\,(u-u')\times
(\zeta'_{\beta'}\partial_{\alpha'}-\zeta'_{\alpha'}\partial_{\beta'})\frak z'\,d\alpha'd\beta'e_3
\end{aligned}
\end{equation}

We show in the following proposition that $b$, $A-1$ are consisting of only quadratic and higher order terms. Let $\mathcal K=\Re \mathcal H=U_k^{-1}\frak K U_k$, $P=\alpha e_1+\beta e_2$, and
\begin{equation}\label{lambda}
\Lambda^*=(I+\frak H)z e_3,\quad \Lambda=(I+\frak H)z e_3-\frak Kze_3,\quad \lambda^*= (I+\mathcal H)\frak z e_3, \quad \lambda=\lambda^*-\mathcal K\frak ze_3
\end{equation}
Therefore 
\begin{equation}\label{zeta1}
\zeta=P+\lambda.
\end{equation}
 Let the velocity $u=u_1 e_1+u_2 e_2+ u_3 e_3$.
\begin{proposition}\label{propAb}Let $b=k_t\circ k^{-1}$ and  $A\circ k={\frak a}J(k)$. We have \footnote{Formulas for $b$ and $A$ similar to those in 2D \cite{wu3} are also available and can be obtained in a similar way: 
\begin{equation*}
(I-\mathcal H) b= (I-\mathcal H)( [\partial_t+b\cdot\nabla_\bot,\mathcal K]\frak z e_3-[\partial_t+b\cdot\nabla_\bot,\mathcal H]\frak z e_3+ \mathcal K u_3 e_3)
\end{equation*}
\begin{equation*}
\begin{aligned}
(I-\mathcal H&)(A  e_3)=
e_3+[\partial_t+b\cdot\nabla_\bot, \mathcal H]u+[A \mathcal N\times \nabla, \mathcal H]\lambda^*\\&+
(I-\mathcal H)(-A \zeta_\beta\times (\partial_\alpha\mathcal K\frak z e_3)+A\zeta_\alpha\times (\partial_\beta\mathcal K \frak z e_3)+A\partial_\alpha\lambda\times \partial_\beta\lambda)
\end{aligned}
\end{equation*}
However we choose to use those in Proposition~\ref{propAb}.}
\begin{align}
b&= \frac12(\mathcal H-\overline{\mathcal H})\overline u-[\partial_t+b\cdot\nabla_\bot,\mathcal H]\frak z e_3+ [\partial_t+b\cdot\nabla_\bot,\mathcal K]\frak z e_3+ \mathcal K u_3 e_3\label{b}\\
(A-1) e_3 &=\frac12(-\mathcal H+\overline{\mathcal H})\overline w+
\frac12 ([\partial_t+b\cdot\nabla_\bot, \mathcal H]u-\overline{[\partial_t+b\cdot\nabla_\bot, \mathcal H]u})\label{A}\\&
+[A \mathcal N\times \nabla, \mathcal H]\frak z e_3
-A \zeta_\beta\times (\partial_\alpha\mathcal K\frak z e_3)+A\zeta_\alpha\times (\partial_\beta\mathcal K \frak z e_3)+A\partial_\alpha\lambda\times \partial_\beta\lambda\tag*{}
\end{align}
Here $\overline{\mathcal H}f=e_3\mathcal H (e_3f)=\iint e_3K\mathcal N' e_3 f'$.
\end{proposition}

\begin{proof}
Taking derivative to $t$ to \eqref{k}, we get
\begin{equation}\label{139}
\begin{aligned}
k_t&=\xi_t-\partial_t(I+\frak H)z e_3+\partial_t\frak K z e_3\\&=
\xi_t-z_t e_3-\frak H z_t e_3-[\partial_t, \frak H] z e_3+\partial_t\frak K z e_3
\end{aligned}
\end{equation}
Now
\begin{equation}\label{140}
\xi_t-z_t e_3-\frak H z_t e_3=\frac12(\xi_t+\overline\xi_t)-\frac12\frak H(\xi_t-\overline\xi_t)=
\frac12\overline\xi_t+\frac12\frak H\overline\xi_t=\frac12(\frak H-\overline{\frak H})\overline\xi_t
\end{equation}
Combining \eqref{139}, \eqref{140} we get
\begin{equation}\label{-140}
k_t=\frac12(\frak H-\overline{\frak H})\overline\xi_t-[\partial_t, \frak H] z e_3+[\partial_t,\frak K] z e_3+\frak K z_t e_3
\end{equation}
Making the change of coordinate $U_k^{-1}$, we get \eqref{b}.

Notice that $A\circ k e_3=\frak a k_\alpha\times k_\beta$. From the definition $k=\xi-\Lambda^*+\frak K z e_3=\xi-\Lambda$, we get
\begin{align*} k_\alpha\times &k_\beta=\xi_\alpha\times\xi_\beta+\xi_\beta\times \partial_\alpha\Lambda^*-\xi_\alpha\times \partial_\beta\Lambda^*\\&-\xi_\beta\times (\partial_\alpha\frak Kze_3)+\xi_\alpha\times (\partial_\beta\frak K z e_3)+\partial_\alpha\Lambda\times \partial_\beta\Lambda
\end{align*}
Using \eqref{hilbertz} and \eqref{analytic1}, we have 
$$\xi_\beta\times \partial_\alpha\Lambda^*-\xi_\alpha\times \partial_\beta\Lambda^*=\xi_\beta \partial_\alpha\Lambda^*-\xi_\alpha\partial_\beta\Lambda^*=(N\times\nabla)\Lambda^*
$$
From \eqref{ww1}, and  the fact that $\frak a N\times \nabla z e_3=- \frak a N_1 e_1-\frak a N_2 e_2$, we obtain 
\begin{equation*}
\begin{aligned}
\frak a \xi_\alpha\times\xi_\beta +\frak a (N\times\nabla)\Lambda^*&=
\xi_{tt}+ e_3+(I+\frak H)(\frak a N\times\nabla)z e_3+[ \frak a N\times\nabla, \frak H]z e_3\\&
=
\xi_{tt}+ e_3-\frac12 (I+\frak H)(\xi_{tt}+\overline{\xi}_{tt})+[ \frak a N\times\nabla, \frak H]z e_3
\end{aligned}
\end{equation*}
and furthermore from \eqref{ww2},
\begin{equation*}
\begin{aligned}
&\xi_{tt}-\frac12 (I+\frak H)(\xi_{tt}+\overline{\xi}_{tt})=\frac12(\xi_{tt}-\frak H\xi_{tt})-\frac12(\overline \xi_{tt}+\frak H\overline \xi_{tt})\\&=\frac12[\partial_t, \frak H]\xi_t-\frac12(\overline \xi_{tt}-\overline{\frak H\xi_{tt}})-\frac12(\frak H\overline\xi_{tt}+\overline{\frak H\xi_{tt}})
=\frac12[\partial_t, \frak H]\xi_t-\frac12\overline{[\partial_t, \frak H]\xi_t}+\frac12(\overline {\frak H}-\frak H)\overline\xi_{tt}
\end{aligned}
\end{equation*}
Combining the above calculations and make the change of coordinates $U_k^{-1}$, we obtain \eqref{A}.

\end{proof}

From Proposition~\ref{propAb}, we see that $b$ and $A-1$ are consisting of terms of quadratic and higher orders. Therefore the left hand side of equation \eqref{cubic2} is
$$(\partial_{t}^2-e_2\partial_\alpha+e_1\partial_\beta)\chi-\partial_\beta\lambda\partial_\alpha\chi+\partial_\alpha\lambda\partial_\beta\chi+\text{cubic and higher order terms}
$$
The quadratic term $\partial_\beta\lambda\partial_\alpha\chi-\partial_\alpha\lambda\partial_\beta\chi$
is new in 3D. We notice that this is one of the null forms studied in \cite{kl4} and we find that it is also null for our equation and can be written as the factor $1/t$ times a quadratic expression involving some "invariant vector fields" for $\partial_{t}^2-e_2\partial_\alpha+e_1\partial_\beta$. Therefore this term is cubic in nature and  equation \eqref{cubic2} is of the type "linear + cubic and higher order perturbations".
 On the other hand, we are curious whether $\partial_\beta\lambda\partial_\alpha\chi-\partial_\alpha\lambda\partial_\beta\chi$ can be transformed into a term of cubic and higher orders by a physically meaningful and mathematically concise transformation.  As we know, although the quantity  
$(I-\frak H)ze_3$ is concise in expression and mathematically it is the projection of $ze_3$ into the space of analytic functions in the air region, we do not know its physical meaning.  Could the transformation for $\partial_\beta\lambda\partial_\alpha\chi-\partial_\alpha\lambda\partial_\beta\chi$, if exist, 
 together with $(I-\frak H)ze_3$ offer us some big picture that has a good physical meaning? Motivated by these questions, we looked further. However after taking into considerations of all possible cancellation properties and relations, we can conclude that this term actually cannot be transformed away by a
concise transformation in the physical space (although there is a transformation explicit in the Fourier space). Of course, we do not rule out the possibility that there are some possible relations we overlooked. On the other hand, our effort lead to some results that can be interesting for those readers interested in understanding the nature of a bilinear normal form change, or the quadratic resonance for the water wave equation. We will present this calculation in a separate paper.

We therefore treat the term $\partial_\beta\lambda\partial_\alpha\chi-\partial_\alpha\lambda\partial_\beta\chi$ by the method of invariant vector fields. In fact,  in section~\ref{energy}, we will see that it is more natural to treat $(\partial_t+b\cdot\nabla_\bot)^2-A\mathcal N\times \nabla $ as the main operator for the water wave equation than treating it as a perturbation of the linear operator $\partial_t^2-e_2\partial_\alpha+e_1\partial_\beta$. We obtain a uniform bound for all time of a properly constructed energy that involves invariant vector fields of $\partial_t^2-e_2\partial_\alpha+e_1\partial_\beta$ by combining  energy estimates for 
the equation \eqref{cubic2} and a generalized Sobolev inequality that gives a $L^2\to L^\infty$ estimate with the decay rate  $1/t$. We point out that not only does the projection $(I-\frak H)$ give us the quantity $(I-\frak H)ze_3$, but it is also used  in various ways to project away  "quadratic noises" in the course of deriving the energy estimates. 
The global in time existence follows from a local well-posedness result, the uniform boundedness of the energy  and a continuity argument. We state our main theorem.

Let $|D|=\sqrt{-\partial_\alpha^2-\partial_\beta^2}$, $H^{s}(\mathbb R^2)=\{ f\, |\, (I+|D|)^{s}f\in L^2(\mathbb R^2)\}$, with $\|f\|_{H^{s}}=\|f\|_{H^{s}(\mathbb R^2)}=\|(I+|D|)^{s}f\|_{L^2(\mathbb R^2)}.$

Let $s\ge 27$, $\max\{[\frac s2]+1, 17\}\le l\le s-10$. Assume that initially  
\begin{equation}\label{ic}
\xi(\alpha,\beta,0)=\xi^0=(\alpha,\beta, z^0(\alpha,\beta)),\quad \xi_t(\alpha,\beta,0)=\frak u^0(\alpha,\beta),\quad \xi_{tt}(\alpha,\beta,0)=\frak w^0(\alpha,\beta),
\end{equation}
 and the  data in \eqref{ic} satisfy the compatibility condition (5.29)-(5.30) of \cite{wu2}. Let $\Gamma=\partial_\alpha,\,\partial_\beta,\, \alpha\partial_\alpha+\beta\partial_\beta,\, \alpha\partial_\beta-\beta\partial_\alpha$. Assume that
\begin{equation}
\sum_{|j|\le s-1\atop\partial=\partial_\alpha,\partial_\beta}\|\Gamma^j |D|^{1/2} z^0\|_{L^2(\mathbb R^2)}+\|\Gamma^j\partial z^0\|_{H^{1/2}(\mathbb R^2)}+\|\Gamma^j \frak u^0\|_{H^{3/2}(\mathbb R^2)}+\|\Gamma^j\frak w^0\|_{H^1(\mathbb R^2)}<\infty
\end{equation}
Let
\begin{equation}
\epsilon=\sum_{|j|\le l+3\atop\partial=\partial_\alpha,\partial_\beta}\|\Gamma^j |D|^{1/2} z^0\|_{L^2(\mathbb R^2)}+\|\Gamma^j\partial z^0\|_{L^2(\mathbb R^2)}+\|\Gamma^j \frak u^0\|_{H^{1/2}(\mathbb R^2)}+\|\Gamma^j\frak w^0\|_{L^2(\mathbb R^2)}.
\end{equation}
\begin{theorem}[Main Theorem]
There exists $\epsilon_0>0$, such that for $0\le \epsilon\le \epsilon_0$, the initial value problem \eqref{ww1}-\eqref{ww2}-\eqref{ic} has a unique classical solution globally in time. For each time $0\le t<\infty$,  the interface is a graph, the solution has the same regularity as the initial data and remains small. Moreover the $L^\infty$ norm of the steepness and the acceleration of the interface, the derivative of the velocity on the interface decay at the rate $\frac 1t$. 

\end{theorem}

A more detailed and precise statement of the main Theorem is given in Theorem~\ref{globalexistence} and the remarks at the end of this paper.

\section{Basic analysis preparations}

For a function $f=f(\alpha,\beta, t)$, we use the notation $f=f(\cdot, t)=f(t)$, 
$$\|f(t)\|_2=\|f(t)\|_{L^2}=\|f(\cdot,t)\|_{L^2(\mathbb R^2)},\quad\quad|f(t)|_\infty=\|f(t)\|_{L^\infty}=\|f(\cdot,t)\|_{L^\infty(\mathbb R^2)}.
$$

\subsection{Vector fields and a generalized Sobolev inequality}

As in \cite{wu3}, we will use the method of invariant vector fields. We know the linear part of the operator 
$\mathcal P=(\partial_t+ b\cdot \nabla_\bot)^2-A \mathcal N\times \nabla$ is $\frak P=\partial_t^2- e_2\partial_\alpha+e_1\partial_\beta$.
Although the invariant vector fields of $\frak P$ was not known, it is not difficult to find them.\footnote{One may find using the method in \cite{cz} that for the scaler operator $\partial_t^2+|D|$, the following are invariants: $\partial_t, \ \partial_\alpha,\ \partial_\beta,\  L_0,\ \Upsilon,\ \alpha\partial_t+\frac 12 t \partial_\alpha |D|^{-1}, \ \beta\partial_t+\frac 12 t \partial_\beta |D|^{-1}.$ Those for $\frak P$ are then obtained by properly modifying this set.} Using a combined method as that in \cite{cz, st}, we find that
the set of operators
\begin{equation}\label{vectorfields}
\Gamma=\{ \partial_t, \quad \partial_\alpha,  \quad \partial_\beta, \quad L_0= \frac 12 t\partial_t+\alpha\partial_\alpha+\beta\partial_\beta,\quad \text{and } 
 \varpi=\alpha\partial_\beta-\beta\partial_\alpha-\frac12 e_3\}
\end{equation}
 satisfy
\begin{equation}\label{com0}
[\partial_t, \frak P]=[\partial_\alpha, \frak P]=[\partial_\beta, \frak P]=[\varpi, \frak P]=0, \quad [L_0, \frak P]=-\frak P.
\end{equation}
Let $ \Upsilon=\alpha\partial_\beta-\beta\partial_\alpha$.  So $\varpi=\Upsilon-\frac12 e_3$. We have 
\begin{equation}\label{commutativity1}
\begin{aligned} 
&[\partial_t,\partial_\alpha]=[\partial_t,\partial_\beta]=[\partial_t,\Upsilon]=[\partial_\alpha, \partial_\beta]=[L_0,\Upsilon]= 0\\
&[\partial_t, L_0]=\frac12\partial_t, \quad[\partial_\alpha,L_0]=[\Upsilon, \partial_\beta]=\partial_\alpha,\quad [\partial_\beta,L_0]=[\partial_\alpha, \Upsilon]=\partial_\beta
\end{aligned}
\end{equation}
Furthermore, we have
\begin{equation}\label{comgt}
\begin{aligned}
& [\partial_t, \partial_t+b\cdot\nabla_\bot]=b_t\cdot\nabla_\bot,\quad [\partial, \partial_t+b\cdot\nabla_\bot]=(\partial b)\cdot\nabla_\bot,\quad \text{for }\partial=\partial_\alpha,\partial_\beta
\\& [L_0,\partial_t+b\cdot\nabla_\bot]=(L_0b-\frac12b)\cdot\nabla_\bot-\frac12(\partial_t+b\cdot\nabla_\bot),\\&
[\varpi, \partial_t+b\cdot\nabla_\bot]= (\varpi b-\frac 12 e_3b)\cdot \nabla_\bot
\end{aligned}
\end{equation}
Let $\mathcal P^\pm=(\partial_t+b\cdot\nabla_\bot )^2\pm A\mathcal N\times \nabla$. Notice that $\mathcal P=\mathcal P^-$.  
 We have
\begin{equation}\label{commgp}
\begin{aligned}
& [ \partial_t, \mathcal P^\pm]=\pm \{(A\zeta_\beta)_t\partial_\alpha
-(A\zeta_\alpha)_t\partial_\beta\}
+\{\partial_t(\partial_t+b\cdot\nabla_\bot )b-b_t\cdot\nabla_\bot b\}\cdot\nabla_\bot\\&+
b_t\cdot\{(\partial_t+b\cdot\nabla_\bot)\nabla_\bot+ \nabla_\bot(\partial_t+b\cdot\nabla_\bot)\}\\
&[\partial, \mathcal P^\pm]=\pm \{(\partial(A\zeta_\beta))\partial_\alpha
-(\partial(A\zeta_\alpha))\partial_\beta\}+\{\partial(\partial_t+b\cdot\nabla_\bot)b-(\partial b)\cdot\nabla_\bot b\}\cdot\nabla_\bot
\\
&+
(\partial b)\cdot\{  (\partial_t+b\cdot\nabla_\bot)\nabla_\bot+ \nabla_\bot(\partial_t+b\cdot\nabla_\bot)\},\qquad \text{for }\partial=\partial_\alpha,\partial_\beta\\
&[ L_0, \mathcal P^\pm]=-\mathcal P^\pm\pm \{L_0(A\zeta_\beta)\partial_\alpha
-L_0(A\zeta_\alpha)\partial_\beta\}+\{(\partial_t+b\cdot\nabla_\bot)(L_0 b-\frac12 b)\}\cdot\nabla_\bot
\\&
+(L_0 b-\frac12 b) 
\cdot\{  (\partial_t+b\cdot\nabla_\bot)\nabla_\bot+ \nabla_\bot(\partial_t+b\cdot\nabla_\bot)\}
\\&
[\varpi,  \mathcal P^\pm]= \pm (\Upsilon A)(\zeta_\beta\partial_\alpha-\zeta_\alpha\partial_\beta)\pm A (\partial_\beta(\varpi\lambda+\frac12 \lambda e_3)\partial_\alpha-\partial_\alpha(\varpi\lambda+\frac12 \lambda e_3)    \partial_\beta)\\&
+(\varpi b-\frac12 e_3 b) 
\cdot\{  (\partial_t+b\cdot\nabla_\bot)\nabla_\bot+ \nabla_\bot(\partial_t+b\cdot\nabla_\bot)\}
\\&+\{(\partial_t+b\cdot\nabla_\bot)(\varpi -\frac12 e_3) b)\}\cdot\nabla_\bot
\end{aligned}
\end{equation}
For any positive integer $m$, and any operator $P$,
\begin{equation}\label{comjp}
[\Gamma^m,  P]=\sum_{j=1}^m\Gamma^{m-j}[\Gamma, P]\Gamma^{j-1}.
\end{equation}

Let $$\text{\bf K}f(\alpha, \beta,t)=p.v.\iint k(\alpha,\beta, \alpha', \beta' ;t)f(\alpha',\beta',t)\,d\alpha'\,d\beta'$$
where  for some $\iota=0, 1$, or $2$, 
$|(\alpha,\beta)-(\alpha',\beta')|^\iota k(\alpha,\beta, \alpha',\beta';t)$ is  bounded, and $k$ is  smooth away from the diagonal $\Delta=\{(\alpha,\beta)=(\alpha', \beta')\}$. 
We have for $f$ vanish fast at spatial infinity,
\begin{equation}\label{comgk}
\begin{aligned}
&[\partial_t, \text{\bf K}]f(\alpha,\beta, t)=\iint\partial_t k(\alpha,\beta,\alpha', \beta';t)f(\alpha',\beta',t)\,d\alpha'\,d\beta'
\\&[\partial,\text{\bf K}]f(\alpha,\beta, t)=\iint(\partial+\partial')k(\alpha,\beta,\alpha', \beta';t)f(\alpha',\beta',t)\,d\alpha'\,d\beta',\qquad\partial=\partial_\alpha,\partial_\beta\\
&
[L_0,\text{\bf K}]f(\alpha,\beta, t)=2\text{\bf K} f(\alpha,\beta, t)\\&+
\iint(\alpha\partial_\alpha+\beta\partial_\beta+\alpha'\partial_\alpha'+\beta'\partial_\beta'+\frac12t\partial_t   )k(\alpha,\beta,\alpha', \beta';t)f(\alpha',\beta',t)\,d\alpha'\,d\beta'
\\&
[\Upsilon, \text{\bf K}]f(\alpha,\beta, t)=\iint (\Upsilon+\Upsilon') k(\alpha,\beta,\alpha', \beta';t)f(\alpha',\beta',t)\,d\alpha'\,d\beta'
\\&
[\partial_t+b\cdot\nabla_\bot, \text{\bf K}]f=\iint (\partial_t+b\cdot\nabla_\bot+b'\cdot\nabla'_\bot)k(\alpha,\beta,\alpha', \beta';t)f(\alpha',\beta',t)\,d\alpha'\,d\beta'\\&
+\iint k(\alpha,\beta,\alpha', \beta';t)\text{div}' b' \,f(\alpha',\beta',t)\,d\alpha'\,d\beta'
\end{aligned}
\end{equation}
One of the  operators  in  equations \eqref{cubic2}, \eqref{quasi1} and \eqref{quasi2} is of the following type: 
$$\bold B(g,f)=p.v.\iint K(\zeta'-\zeta)(g-g')\times (\zeta'_{\beta'}\partial_{\alpha'}-\zeta'_{\alpha'}\partial_{\beta'})f(\alpha',\beta',t)\,d\alpha'\,d\beta'$$
where $\Re g=0$. We have for $\Gamma=\partial_t,\partial_\alpha,\partial_\beta, L_0, \varpi$, 
\begin{equation}\label{comgq}
\begin{aligned}
\Gamma&\bold B(g,f)=\iint K(\zeta'-\zeta)(\dot \Gamma g-\dot \Gamma' g')\times (\zeta'_{\beta'}\partial_{\alpha'}-\zeta'_{\alpha'}\partial_{\beta'})f(\alpha',\beta',t)\,d\alpha'\,d\beta'\\&+\iint K(\zeta'-\zeta)(g-g')\times (\zeta'_{\beta'}\partial_{\alpha'}-\zeta'_{\alpha'}\partial_{\beta'})\Gamma' f(\alpha',\beta',t)\,d\alpha'\,d\beta'\\&+\iint K(\zeta'-\zeta)(g-g')\times (\partial_{\beta'}\dot \Gamma'\lambda'\partial_{\alpha'}-\partial_{\alpha'}\dot \Gamma'\lambda'\partial_{\beta'})f(\alpha',\beta',t)\,d\alpha'\,d\beta'\\&+\iint ((\dot \Gamma'\lambda'-\dot \Gamma\lambda)\cdot\nabla) K(\zeta'-\zeta)(g-g')\times (\zeta'_{\beta'}\partial_{\alpha'}-\zeta'_{\alpha'}\partial_{\beta'})f(\alpha',\beta',t)\,d\alpha'\,d\beta'
\end{aligned}
\end{equation}
where $\dot\Gamma g=\partial_t g,\ \partial_\alpha g,\ \partial_\beta g,\ (L_0-I)g, \ \varpi g+\frac12 g e_3$ respectively. \eqref{comgq} is straightforward with an application of \eqref{comgk}, the definition $\zeta=P+\lambda$, and in the case $\Gamma=L_0$, the fact $(\xi\cdot\nabla) K(\xi)=-2K(\xi)$ and \eqref{commutativity1}; in the case $\Gamma=\varpi$, the fact
$((e_3\times\xi)\cdot \nabla)K(\xi)=\frac12(e_3K(\xi)-K(\xi)e_3)$, \eqref{commutativity1} and $-e_3 \,a\times b+ a\times b\, e_3+2(e_3\times a)\times b+2a\times (e_3\times b)=0$, for $a,\, b\in \mathbb R^3$.

Before we derive the commutativity relations between $L_0$, $\varpi$ and $\mathcal H$, we record
\begin{lemma}\label{lemmavector} Let $\Omega$ be a  $C^2$ domain in $\mathbb R^2$, with its boundary $\Sigma=\partial\Omega$ being parametrized by 
$\xi=\xi(\alpha,\beta)$, $(\alpha,\beta)\in \mathbb R^2$. For  any vector $\eta$, and function $f$ on $\mathbb R^2$, we have
\begin{equation}\label{vector2}(\eta\times \xi_{\beta})f_{ \alpha}-(\eta\times
\xi_{\alpha})f_{\beta}= (\xi_\alpha\times\xi_\beta)(\eta\cdot\nabla_\xi )f- 
(\eta\cdot(\xi_\alpha\times\xi_\beta))\,\mathcal D_\xi f.
\end{equation} 
\begin{equation}\label{vector3}
-(\eta\cdot \nabla) K(\xi)\,
(\xi'_{\alpha'}\times\xi'_{\beta'})+(\xi'_{\alpha'}\cdot\nabla) K(\xi)\,(\eta\times
\xi'_{\beta'})+(\xi'_{\beta'}\cdot\nabla) K(\xi)\,(\xi'_{\alpha'}\times\eta)=0,
\end{equation}
for $\xi\ne 0$.
\end{lemma}

\eqref{vector2} is proved in the same way as the identity (5.17) in \cite{wu2}. We omit the details. \eqref{vector3} is the identity (3.5) in \cite{wu2} .

We have the following  commutativity relations between $L_0$, $\varpi$ and $\mathcal H$.
\begin{proposition}\label{propcomgh} Let $f\in C^1(\mathbb R^2\times [0, T])$ be a $\mathcal C(V_2)$ valued  function vanishing at spatial infinity. Then
\begin{align}
[L_0, \mathcal H]f&=\iint K(\zeta'-\zeta)\,((L_0-I)\lambda-  (L'_0-I)\lambda'  )\times
(\zeta'_{\beta'}\partial_{\alpha'}-\zeta'_{\alpha'}\partial_{\beta'})f'\,d\alpha'd\beta'.\label{commutelh}\\
 [\varpi, \mathcal H]f&=\iint
K(\zeta'-\zeta)\,(  
\varpi\lambda+\frac12 \lambda e_3-\varpi'\lambda'-\frac12 \lambda' e_3
)\times 
(\zeta'_{\beta'}\partial_{\alpha'}-\zeta'_{\alpha'}\partial_{\beta'})f'\,d\alpha'd\beta'\label{commutevh}
\end{align}
\end{proposition}
\begin{proof}
Using \eqref{commuteth}, \eqref{commuteah}, \eqref{commutebh} and argue similarly as in the proof of \eqref{commutenh}, we can show  that
\begin{equation*}
[L_0, \mathcal H]f=\iint K(\zeta'-\zeta)\,(L_0\zeta-  L'_0\zeta'  )\times
(\zeta'_{\beta'}\partial_{\alpha'}-\zeta'_{\alpha'}\partial_{\beta'})f'\,d\alpha'd\beta'.
\end{equation*}
Using integration by parts and \eqref{vector3}, we can check the following identity:
$$\iint K(\zeta'-\zeta) \, (\zeta-\zeta')\times(\zeta'_{\beta'}\partial_{\alpha'}-\zeta'_{\alpha'}\partial_{\beta'})f'\,d\alpha'd\beta'=0$$
\eqref{commutelh}  then follows from  the fact that $(L_0-I)\zeta=(L_0-I)\lambda$.

\eqref{commutevh} is obtained similarly. First we have by using \eqref{commuteah}, \eqref{commutebh} that
\begin{equation*}
[\Upsilon, \mathcal H]f=\iint K(\zeta'-\zeta)\,(\Upsilon\zeta-  \Upsilon'\zeta'  )\times
(\zeta'_{\beta'}\partial_{\alpha'}-\zeta'_{\alpha'}\partial_{\beta'})f'\,d\alpha'd\beta'.
\end{equation*}
We now check the identity
$$\frac12 [e_3, \mathcal H]f=\iint K(\zeta'-\zeta)\,(e_3\times (\zeta-\zeta'))\times(\zeta'_{\beta'}\partial_{\alpha'}-\zeta'_{\alpha'}\partial_{\beta'})f'\,d\alpha'd\beta'.$$
Using integration by parts, we have
\begin{equation*}
\begin{aligned}
&\iint K(\zeta'-\zeta)\,(e_3\times (\zeta-\zeta'))\times(\zeta'_{\beta'}\partial_{\alpha'}-\zeta'_{\alpha'}\partial_{\beta'})f'\,d\alpha'\,d\beta'
\\&=-
\iint (\partial_\alpha'K\,(e_3\times (\zeta-\zeta'))\times\zeta'_{\beta'}-\partial_{\beta'} K\,(e_3\times (\zeta-\zeta'))\times  \zeta'_{\alpha'})f'\,d\alpha'\,d\beta'
\\&+\iint K(\zeta'-\zeta)\,((e_3\times \zeta'_{\alpha'})\times\zeta'_{\beta'}- (e_3\times \zeta'_{\beta'})\times    \zeta'_{\alpha'})f'\,d\alpha'\,d\beta'
\\&=\iint((e_3\times (\zeta'-\zeta))\cdot \nabla)K(\zeta'-\zeta)\mathcal N' f'\,d\alpha'\,d\beta'+\iint K(\zeta'-\zeta) e_3\times \mathcal N' f'\,d\alpha'\,d\beta'
\\&=\frac12\iint (e_3 K \mathcal N' -K\mathcal N' e_3)f' \,d\alpha'\,d\beta'=\frac12 [e_3, \mathcal H]f
\end{aligned}
\end{equation*}
Here in the second step we used \eqref{vector3}, and the identity $(a\times b)\times c=b\,a\cdot c-a\,b\cdot c$. In the last step we used the fact $((e_3\times\xi)\cdot \nabla)K(\xi)=\frac12(e_3K(\xi)-K(\xi)e_3)$ and $e_3\times \mathcal N=\frac12 (e_3 \mathcal N- \mathcal N e_3)$. \eqref{commutevh} therefore follows since $\Upsilon \zeta-e_3\times\zeta=\Upsilon \lambda-e_3\times\lambda=\varpi\lambda+\frac12\lambda e_3$.

\end{proof}

In what follows, we denote the vector fields in \eqref{vectorfields} by $\Gamma_i$ $i=1, \dots, 5$, or simply suppress the subscript and write as $\Gamma$. We shall write 
$$\Gamma^k=\Gamma_1^{k_1}\Gamma_2^{k_2}\Gamma_3^{k_3}\Gamma_4^{k_4}\Gamma_5^{k_5}$$
for $k=(k_1, k_2, k_3, k_4, k_5)$. For a nonnegative integer $k$, we shall also use $\Gamma^k$ to indicate a $k-$product of $\Gamma_i$. $i=1,\dots,5$.

We now develop a generalized Sobolev inequality. Let $\alpha_1=\alpha$, $\alpha_2=\beta$. We introduce 
\begin{equation}\label{o0j}
\Omega_{0j}^\pm
=\pm\alpha_j\partial_t+\frac 12 t\partial_{\alpha_j} |D|^{-1} H,\quad j=1,2
\end{equation}
where     $H=(e_2\partial_{\alpha_1}-e_1\partial_{\alpha_2})|D|^{-1}$. Therefore $H^2=I$. We also denote $\Omega_{0j}^-=\Omega_{0j}$.\footnote{One may check that $[\Omega_{01}(e_2\partial_{\alpha_1}-e_1\partial_{\alpha_2})-\frac12 \partial_t e_2,\frak P]=0$, 
$[\Omega_{02}(e_2\partial_{\alpha_1}-e_1\partial_{\alpha_2})+\frac12 \partial_t e_1,\frak P]=0$. These are some of the invariant vector fields for $\frak P$, not included in \eqref{vectorfields}.}
 Let $\frak P^\pm=\partial_t^2\pm(e_2\partial_{\alpha_1}-e_1\partial_{\alpha_2})$.  Notice that $\frak P=\frak P^-$. We know
 \begin{equation}\label{371}
\mathcal P^\pm=\frak P^\pm+\partial_t (b\cdot\nabla_\bot)+b\cdot\nabla_\bot(\partial_t+b\cdot\nabla_\bot)\pm A(\lambda_\beta\partial_\alpha-\lambda_\alpha\partial_\beta)\pm (A-1)(e_2\partial_\alpha-e_1\partial_\beta)
\end{equation}
  Let $P_d(\partial)$ be a  
polynomial of $\partial_{\alpha_j}$, $j=1,2$, homogenous of degree $d$, with coefficients in $\mathbb R$. We have
\begin{lemma}\label{propbasicoj} 1.
\begin{equation}\label{eqbasicoj}
\begin{aligned}
&(\partial_{\alpha_1}^2+ \partial_{\alpha_2}^2)\Omega^\pm_{01}=\pm(\partial_{\alpha_1}(2\partial_t+L_0\partial_t-\frac12 t\frak P^\pm)+\partial_{\alpha_2}\Upsilon \partial_t)\\&
(\partial_{\alpha_1}^2+ \partial_{\alpha_2}^2)\Omega^\pm_{02}=\pm(\partial_{\alpha_2}(2\partial_t+L_0\partial_t-\frac12 t\frak P^\pm)-\partial_{\alpha_1}\Upsilon \partial_t)
\end{aligned}
\end{equation}
2.
\begin{equation}\label{basicoj} 
\||D|\Omega_{0j}^\pm F(t)\|_{L^2}\le 2\sum_{k\le 1}\|\partial_t\Gamma^k F(t)\|_{L^2}+t\|\frak P^\pm F(t)\|_{L^2}
\end{equation}
3.
\begin{equation}\label{comgp}
[\hat\Gamma, P_{d+l}(\partial)|D|^{-d}]=R |D|^l,\quad\text{for }\hat\Gamma=L_0, \Upsilon,\ l=0,1,\quad  [\Omega_{0j},  P_{d+1}(\partial)|D|^{-d}]=R\partial_t,
\end{equation}
where $R$ is a finite sum of operators of the type $P_{k}(\partial)|D|^{-k}$, and need not be the same for different  $\hat\Gamma$, $\Omega_{0j}$, $j=1,2$ or $l=0,1$.
\end{lemma}
\begin{proof}  \eqref{comgp} is straightforward  using Fourier transform. We prove \eqref{eqbasicoj} for $\Omega_{01}=\Omega_{01}^-$, 
the other cases follows similarly. We have
\begin{equation*}
\begin{aligned}
&\partial_{\alpha_1}\Omega_{02}-\partial_{\alpha_2}\Omega_{01}=\Upsilon \partial_t\\
&\partial_{\alpha_1}\Omega_{01}+\partial_{\alpha_2}\Omega_{02}=-2\partial_t-L_0\partial_t+\frac12 t\frak P
\end{aligned}
\end{equation*}
Therefore
$$(\partial_{\alpha_1}^2+ \partial_{\alpha_2}^2)\Omega_{01}=\partial_{\alpha_1}(-2\partial_t-L_0\partial_t+\frac12 t\frak P)-\partial_{\alpha_2}\Upsilon \partial_t.$$
\eqref{basicoj} is straightforward from here. 
\end{proof}

\begin{proposition}[\bf {generalized Sobolev inequality}]\label{propgsobolev}
Let $f\in C^\infty(\mathbb R^{2+1})$ be a $\mathcal C(V_2)$ valued function, vanishing at spatial infinity. We have for $l=1,2$,
\begin{equation}\label{gsobolev}
\begin{aligned}
&(1+t+|\alpha_1|+|\alpha_2|)|\partial_{\alpha_l} f(\alpha_1,\alpha_2,t)|\\&
\lesssim \sum_{k\le 4, j=1,2}(\|\Gamma^k \partial_t f(t)\|_{L^2} +\|\Gamma^k \partial_{\alpha_j} f(t)\|_{L^2}) +t\sum_{k\le 3}\|\frak P \Gamma^k f(t)\|_{L^2}
\end{aligned}
\end{equation}
Here $a\lesssim b$ means that there is a universal constant $c$, such that $a\le cb$.
\end{proposition}
\begin{proof}
Let $r^2=\alpha_1^2+\alpha_2^2\,$, $r\partial_r=\alpha_1\partial_{\alpha_1}+\alpha_2\partial_{\alpha_2}$. We have
\begin{equation*}
\sum_1^2\frac{\alpha_j}r\Omega_{0j}=-r\partial_t+\frac12t\partial_r |D|^{-1}H,\qquad L_0=\frac12 t\partial_t+r\partial_r
\end{equation*}
therefore
\begin{equation}\label{215}
rL_0+\frac t2\sum_1^2\frac{\alpha_j}r\Omega_{0j}=r^2\partial_r+\frac{t^2}4\partial_r|D|^{-1}H.
\end{equation}
Also
\begin{equation*}
\sum_1^2\Omega_{0j}\partial_{\alpha_j}=-r\partial_r\partial_t-\frac12t|D|H,\qquad \sum_1^2\alpha_jL_0\partial_{\alpha_j}=\frac12 t\,r\partial_t\partial_r+r^2\partial_r^2
\end{equation*}
gives
\begin{equation}\label{216}
\frac12 t \sum_1^2\Omega_{0j}\partial_{\alpha_j}|D|^{-1}H+\sum_1^2\alpha_jL_0\partial_{\alpha_j}|D|^{-1}H=r^2\partial_r^2|D|^{-1}H-\frac14 t^2.
\end{equation}
Let $g$ be a $\mathcal C(V_2)$ valued function, $h=\partial_r |D|^{-1}Hg$.  From \eqref{215},\eqref{216} we have ( $i$ is the complex number in this proof)
\begin{equation}\label{217}
r^2\partial_r(g+ih)-\frac14t^2i\,(g+ih)=F,
\end{equation}
where
$$F=rL_0g+\frac t2\sum_1^2\frac{\alpha_j}r\Omega_{0j}g+i\,( \frac12 t \sum_1^2\Omega_{0j}\partial_{\alpha_j}|D|^{-1}Hg+\sum_1^2\alpha_jL_0\partial_{\alpha_j}|D|^{-1}Hg).
$$
Using \eqref{217} we get $r^2\partial_r(e^{\frac{it^2}{4r}}(g+i\,h))=e^{\frac{it^2}{4r}}F$,
therefore (Recall by definition, components of a $\mathcal C(V_2)$ valued function are real valued.)
\begin{equation}\label{218}
|g(re^{i\theta},t)|\le |(g+i\,h)(re^{i\theta},t)|\le \int_r^\infty\frac 1{s^2}|F(se^{i\theta},t)|\,ds
\end{equation}
and this implies that
\begin{equation}\label{219}
\begin{aligned}
 |g(re^{i\theta},t)|^2&\lesssim \frac1{r^2}\int_r^\infty s(|L_0 g|^2+\sum_1^2|L_0\partial_{\alpha_j}|D|^{-1}Hg|^2)\,ds\\&+
 \frac{t^2}{r^4}\sum_{j=1}^2\int_r^\infty s(|\Omega_{0j}g|^2+|\Omega_{0j}\partial_{\alpha_j}|D|^{-1}Hg|^2)\,ds.
 \end{aligned}
 \end{equation}
Now in \eqref{219} we let $g=\Upsilon^k\partial_{\alpha_l}f$. Using Lemma 1.2 on page 40 of \cite{so},  and \eqref{comgp}, \eqref{commutativity1},   
 we obtain
 \begin{equation}\label{220}
 \begin{aligned}
 |\partial_{\alpha_l}&f(re^{i\theta_0}, t)|^2\lesssim \sum_{k\le 2} \int_0^{2\pi} |\Upsilon^k \partial_{\alpha_l}f(re^{i\theta}, t)|^2\,d\theta\lesssim
 \frac 1{r^2}\sum_{j=1}^2\sum_{k\le 2, m\le 1} \|L_0^m\Upsilon^k \partial_{\alpha_j}f(t)\|^2_{L^2}\\&
 +\frac{t^2}{r^4}(\sum_{j=1}^2\sum_{k\le 2}\||D|\Omega_{0j}\Upsilon^k f(t)\|^2_{L^2}+\|\partial_t\Upsilon^k f(t)\|^2_{L^2})
 \end{aligned}
 \end{equation}
A further application of \eqref{basicoj} gives
\begin{equation}\label{221}
|\partial_{\alpha_l}f(re^{i\theta_0}, t)|^2\lesssim \frac 1{r^2}\sum_{k\le 3}\sum_{j=1}^2\|\Gamma^k\partial_{\alpha_j}f(t)\|^2_{2}+\frac{t^2}{r^4}\sum_{k\le 3}\|\Gamma^k \partial_t f(t)\|^2_{2}+ \frac{t^4}{r^4}\sum_{k\le 2}\|\frak P\Gamma^k f(t)\|_{2}^2.
\end{equation}
 From a similar argument we also have
 \begin{equation}\label{222}
 \begin{aligned}
 |\partial_{\alpha_m}&|D|^{-1}H\partial_{\alpha_l}f(re^{i\theta_0},t)|^2\lesssim \frac 1{r^2}\sum_{k\le 3}\sum_{j=1}^2\|\Gamma^k\partial_{\alpha_j}f(t)\|^2_{2}
 \\&
 +\frac{t^2}{r^4}\sum_{k\le 3}\|\Gamma^k \partial_t f(t)\|^2_{2}+ \frac{t^4}{r^4}\sum_{k\le 2}\|\frak P\Gamma^k f(t)\|_{2}^2.
 \end{aligned}
\end{equation}

Case 0: $|t|+r\le 1$. \eqref{gsobolev} follows from the standard Sobolev embedding.

Case 1: $t\le r$ and $|t|+r\ge 1$. \eqref{gsobolev} follows from \eqref{221}

Case 2: $r\le t$ and $|t|+r\ge 1$. We use \eqref{216}. We have 
\begin{equation}\label{223}
\frac14 t^2 \partial_{\alpha_l}f =r^2\partial_r^2|D|^{-1}H\partial_{\alpha_l}f-\frac12 t \sum_{j=1}^2\Omega_{0j}\partial_{\alpha_j}|D|^{-1}H\partial_{\alpha_l}f-\sum_{j=1}^2\alpha_jL_0\partial_{\alpha_j}|D|^{-1}H\partial_{\alpha_l}f
\end{equation}
Using \eqref{222} to estimate the first term, the standard Sobolev embedding and Lemma~\ref{propbasicoj} 
to estimate the second and third term on the right hand side of \eqref{223}. We obtain \eqref{gsobolev}.

\end{proof}
  
  \begin{proposition}\label{proprotation}
Let $f$, $g$ be real valued functions. We have
\begin{equation}\label{rotation}
\begin{aligned}
\partial_{\alpha_1}&f\partial_{\alpha_2}g- \partial_{\alpha_2}f\partial_{\alpha_1}g=\frac2t\{\mp\partial_t (e_2\partial_{\alpha_1}-e_1\partial_{\alpha_2}) f\,\Upsilon g\\&+\Omega_{01}^\pm(e_2\partial_{\alpha_1}-e_1\partial_{\alpha_2})f\,\partial_{\alpha_2}g-  \Omega_{02}^\pm(e_2\partial_{\alpha_1}-e_1\partial_{\alpha_2})f\,\partial_{\alpha_1}g\}
\end{aligned}
\end{equation}
\end{proposition}

The proof is straightforward from definition. We omit the details.

\subsection{Estimates of the Cauchy type integral operators} 
Let $J\in C^1(\mathbb R^d; \mathbb R^l)$, $A_i\in C^1(\mathbb R^d), i=1,\dots, m$, $F\in C^\infty(\mathbb R^l)$. Define (for $x, y\in \mathbb R^d$) 
\begin{equation}\label{defcauchy1}
C_1(J, A, f)(x)=p.v.\int F(\frac{J(x)-J(y)}{|x-y|})\frac{\Pi_{i=1}^m(A_i(x)-A_i(y))}{|x-y|^{d+m}}f(y)\,dy.
\end{equation}
Assume that $k_1(x,y)=F(\frac{J(x)-J(y)}{|x-y|})\frac{\Pi_{i=1}^m(A_i(x)-A_i(y))}{|x-y|^{d+m}}$ is odd, i.e. $k_1(x,y)=-k_1(y,x)$.
\begin{proposition}\label{propcauchy1}
  There exist constants $c_1=c_1(F, \|\nabla J\|_{L^\infty}) $, $c_2=c_2(F, \|\nabla J\|_{L^\infty})$, such that 
1.  For any $f\in L^2(\mathbb R^d),\ \nabla A_i\in L^\infty(\mathbb R^d), \ 1\le i\le m, $
\begin{equation}\label{cauchy1}
\|C_1(J, A, f)\|_{L^2(\mathbb R^d)}\le c_1\|\nabla A_1\|_{L^\infty(\mathbb R^d)}\dots\|\nabla A_m\|_{L^\infty(\mathbb R^d)}\|f\|_{L^2(\mathbb R^d)}. 
\end{equation}

2.  For any $ f\in L^\infty(\mathbb R^d), \ \nabla A_i\in L^\infty(\mathbb R^d), \ 2\le i\le m,\ \nabla A_1\in L^2(\mathbb R^d)$, 
\begin{equation}\label{cauchy2}
\|C_1(J, A, f)\|_{L^2(\mathbb R^d)}\le c_2\|\nabla A_1\|_{L^2(\mathbb R^d)}\|\nabla A_2\|_{L^\infty(\mathbb R^d)}\dots\|\nabla A_m\|_{L^\infty(\mathbb R^d)}\|f\|_{L^\infty(\mathbb R^d)}. 
\end{equation}

\end{proposition}

\begin{proof} \eqref{cauchy1} is a result of Coifman, McIntosh and Meyer \cite{cmm, cdm, ke}. 

We prove \eqref{cauchy2} by the method of rotations. We only write for $d=2,\ m=1$, the same argument applies to general cases. Let $R_\theta f(x)=f(e^{i\theta} x)$, $x=(x_1,x_2)=x_1+i\,x_2\in \mathbb R^2$, 
$$K(J, A, f)(x)=p.v.\int_{\mathbb R^1} F(\frac{J(x)-J(x+r)}{r})\frac{A(x)-A(x+r)}{r^2}f(x+r)\,dr.
$$
We have, from the assumption that $k_1(x,y)$ is odd, that
$$C_1(J, A, f)(x)=\int_0^{\pi} R_\theta^{-1} K(R_\theta J, R_\theta A, R_\theta f)(x)\,d\theta$$
\eqref{cauchy2} now follows from the inequality (3.21) in \cite{wu3}.

\end{proof}

Let $J$, $A_i$, $F$ be as above, define (for $x, y\in \mathbb R^d$)
\begin{equation}\label{defcauchy2}
C_2(J, A, f)(x)=p.v.\int F(\frac{J(x)-J(y)}{|x-y|})\frac{\Pi_{i=1}^m(A_i(x)-A_i(y))}{|x-y|^{d+m-1}}\partial_{y_k} f(y)\,dy.
\end{equation}
Assume that $k_2(x,y)=F(\frac{J(x)-J(y)}{|x-y|})\frac{\Pi_{i=1}^m(A_i(x)-A_i(y))}{|x-y|^{d+m-1}}$ is even, i.e. $k_2(x,y)=k_2(y,x)$. 
\begin{proposition}\label{propcauchy2}
  There exist constants $c_1=c_1(F, \|\nabla J\|_{L^\infty}) $, $c_2=c_2(F, \|\nabla J\|_{L^\infty})$, such that 
1.  For any $f\in L^2(\mathbb R^d),\ \nabla A_i\in L^\infty(\mathbb R^d), \ 1\le i\le m, $
\begin{equation}\label{cauchy3}
\|C_2(J, A, f)\|_{L^2(\mathbb R^d)}\le c_1\|\nabla A_1\|_{L^\infty(\mathbb R^d)}\dots\|\nabla A_m\|_{L^\infty(\mathbb R^d)}\|f\|_{L^2(\mathbb R^d)}. 
\end{equation}
2.  For any $ f\in L^\infty(\mathbb R^d), \ \nabla A_i\in L^\infty(\mathbb R^d), \ 2\le i\le m,\ \nabla A_1\in L^2(\mathbb R^d)$, 
\begin{equation}\label{cauchy4}
\|C_1(J, A, f)\|_{L^2(\mathbb R^d)}\le c_2\|\nabla A_1\|_{L^2(\mathbb R^d)}\|\nabla A_2\|_{L^\infty(\mathbb R^d)}\dots\|\nabla A_m\|_{L^\infty(\mathbb R^d)}\|f\|_{L^\infty(\mathbb R^d)}. 
\end{equation}

\end{proposition}
Proposition~\ref{propcauchy2} follows from Proposition~\ref{propcauchy1} and integration by parts. 
We also have the following $L^\infty$ estimate for $C_1(J, A, f)$ as defined in \eqref{defcauchy1}.
\begin{proposition}\label{propcauchy3}
There exists a constant  $c=c(F, \|\nabla J\|_{L^\infty}, \|\nabla^2 J\|_{L^\infty})$, such that for any real number $r>0$,
\begin{equation}\label{cauchy5}
\begin{aligned}
 \|C_1(J, A, f)\|_{L^\infty}\le& c\big(\prod_{i=1}^m(\|\nabla A_i\|_{L^\infty}+\|\nabla^2 A_i\|_{L^\infty})(\|f\|_{L^\infty}
+\|\nabla f\|_{L^\infty})\\&+
\prod_{i=1}^m\|\nabla A_i\|_{L^\infty}\|f\|_{L^\infty}\ln r
+\prod_{i=1}^m\|\nabla A_i\|_{L^\infty}\|f\|_{L^2}\,\frac1{r^{d/2}}\big).
\end{aligned}
\end{equation}
\end{proposition} 
The proof of Proposition~\ref{propcauchy3} is an easy modification of that of Proposition 3.4 in \cite{wu3}. We omit.

At last, we record the standard Sobolev embedding. 

\begin{proposition}\label{propsobolev} For any $f\in C^\infty(\mathbb R^2)$,
\begin{equation}\label{sobolev}
\|f\|_{L^\infty}\lesssim \|f\|_{L^2}+\|\nabla f\|_{L^2}+\|\nabla^2 f\|_{L^2}
\end{equation}
\end{proposition}

\subsection{Regularities and relations among various quantities}\label{relations}
In this subsection, we study the relations and $L^2$, $L^\infty$ regularities of various quantities involved. 
We first present the quasi-linear equation for $u=\xi_t\circ k^{-1}$ and a formula for $\frak a_t$. These are very much the same as those derived in \cite{wu2}. We also give the equations for  $\lambda^*$ and $\frak v=(\partial_t+b\cdot\nabla_\bot )\chi$. 
\begin{proposition} 
We have 
1. \begin{equation}\label{quasi1}
((\partial_t+ b\cdot \nabla_\bot)^2+A \mathcal N\times \nabla)u=U_k^{-1}(\frak a_t N)
\end{equation}
where
\begin{equation}
\begin{aligned}\label{at}
(I-\mathcal H)&(U_k^{-1}(\frak a_t N))=2\iint
K(\zeta'-\zeta)(w-w')\times (
\zeta'_{\beta'}\partial_{ \alpha'}-
\zeta'_{\alpha'}\partial_{ \beta'})u'\,d\alpha'd\beta'\\&+
\iint K(\zeta'-\zeta)\,\{((u-u')\times
u'_{\beta'})u'_{ \alpha'}-((u-u')\times
u'_{\alpha'})u'_{ \beta'}\}\,d\alpha'd\beta'\\&+
2\iint K(\zeta'-\zeta)\,(u-u')\times(
\zeta'_{\beta'}\partial_{ \alpha'}-
\zeta'_{\alpha'}\partial_{ \beta'})w'\,d\alpha'd\beta'\\&+
\iint
((u'-u)\cdot\nabla)K(\zeta'-\zeta) (u-u')\times(\zeta'_{\beta'}\partial_\alpha'-\zeta'_{\alpha'}\partial_\beta')u'
\,d\alpha'd\beta' 
\end{aligned}
\end{equation}
2.
\begin{equation}\label{quasi2}
\begin{aligned}
((\partial_t&+ b\cdot \nabla_\bot)^2+A \mathcal N\times \nabla)\lambda^*=-(\mathcal H-\overline{\mathcal  H})\overline w\\&
-e_3  \iint K(\zeta'-\zeta)\,(u-u')\times
(\zeta'_{\beta'}\partial_{\alpha'}-\zeta'_{\alpha'}\partial_{\beta'})u'\,d\alpha'd\beta' e_3\\&
+2\iint
K(\zeta'-\zeta)(w-w')\times (
\zeta'_{\beta'}\partial_{ \alpha'}-
\zeta'_{\alpha'}\partial_{ \beta'})\frak z'\,d\alpha'd\beta'e_3\\&+
\iint K(\zeta'-\zeta)\,\{((u-u')\times
u'_{\beta'})\frak z'_{ \alpha'}-((u-u')\times
u'_{\alpha'})\frak z'_{ \beta'}\}\,d\alpha'd\beta'e_3\\&+
2\iint K(\zeta'-\zeta)\,(u-u')\times(
\zeta'_{\beta'}\partial_{ \alpha'}-
\zeta'_{\alpha'}\partial_{ \beta'})u_3'\,d\alpha'd\beta'e_3\\&+
\iint
((u'-u)\cdot\nabla)K(\zeta'-\zeta) (u-u')\times(\zeta'_{\beta'}\partial_\alpha'-\zeta'_{\alpha'}\partial_\beta')\frak z'
\,d\alpha'd\beta' e_3
\end{aligned}
\end{equation}
3. \begin{equation}\label{cubic3}
 \begin{aligned}
((\partial_t&+ b\cdot \nabla_\bot)^2-A \mathcal N\times \nabla)\frak v=\frac{\frak a_t}{\frak a}\circ k^{-1}A \mathcal N\times \nabla\chi\\&+A(u_\beta\chi_\alpha-u_\alpha\chi_\beta)+(\partial_t+ b\cdot \nabla_\bot)((\partial_t+ b\cdot \nabla_\bot)^2-A \mathcal N\times \nabla)\chi
\end{aligned}
\end{equation}
\end{proposition}
\begin{proof} \eqref{quasi1} is derived from \eqref{ww1} \eqref{ww2}. Taking derivative to $t$ to \eqref{ww1},  we have $\xi_{ttt}-\frak a N_t=\frak a_t N$. Using \eqref{analytic1}, \eqref{ww2}, we derive $$N_t=-\xi_\beta\times\xi_{t\alpha}+\xi_\alpha\times\xi_{t\beta}=-\xi_\beta\xi_{t\alpha}+\xi_\alpha\xi_{t\beta}=-N\times\nabla \xi_t.$$
Therefore
\begin{equation}\label{142}
\xi_{ttt}+\frak aN\times \nabla \xi_t=\frak a_t N
\end{equation}
Now to derive an equation for $\frak a_tN$, we apply $(I-\frak H)$ to both sides of \eqref{142}. We get
$$(I-\frak H)(\frak a_t N)=(I-\frak H)(\xi_{ttt}+(\frak aN\times\nabla)\xi_t)=[\partial_t^2+\frak a N\times\nabla,\frak H]\xi_t$$
 \eqref{at} then follows from \eqref{commutetth}, \eqref{commutenh} and an application of the coordinate change $U_k^{-1}$.  An application of the coordinate change $U_k^{-1}$ to \eqref{142} gives  \eqref{quasi1} .
 
 We can  derive the equation for $\lambda^*$ in a similar way as that for $\chi$. We have
\begin{equation}\label{141}
(\partial_t^2+\frak a N\times\nabla)\Lambda^*=(I+\frak H)(\partial_t^2+\frak a N\times\nabla)z e_3+[ \partial_t^2+\frak a N\times\nabla, \frak H]ze_3
\end{equation}
Notice that $(\partial_t^2+\frak a N\times\nabla)z e_3=-\overline{\xi_{tt}} $ and 
$$(I+\frak H)\overline{\xi_{tt}}=e_3(\xi_{tt}-\frak H\xi_{tt})e_3+ (\frak H-e_3\frak H e_3)\overline{\xi_{tt}}=e_3[\partial_t, \frak H]\xi_t e_3+(\frak H-\overline{\frak H} )\overline{\xi_{tt}}$$
\eqref{quasi2} again follows from Lemma~\ref{lemma 1.2} and then an application of the change of coordinate $U_k^{-1}$ to \eqref{141}. We remark that the right hand sides of both \eqref{quasi1}, \eqref{quasi2} are of terms that are at least quadratic.

\eqref{cubic3} is obtained by taking derivative $\partial_t$ to \eqref{cubic1}, then make the change of variable $U_k^{-1}$. 
\end{proof}

We present some useful identities in the following. 
\begin{proposition}\label{propon} 
For $f^\pm$ satisfying $f^\pm=\pm\mathcal Hf^\pm$, and $g$ being vector valued, we have
\begin{equation}\label{on}
 \iint K(\zeta'-\zeta)(g-g')\times(\zeta'_{\beta}\partial_{\alpha'}-\zeta'_{\alpha'}\partial_{\beta'}){f'}^\pm\,d\alpha' d\beta'=(\pm I-\mathcal H)(g\cdot\nabla_\xi^\pm f)
\end{equation}
\end{proposition}
\begin{proof} We only prove for $f$ satisfying $f=\mathcal Hf$. We know
$f(\alpha, \beta,t)=F(\zeta(\alpha,\beta,t),t)$ for some $F$  analytic in $\Omega(t)$. 
From \eqref{vector2}, we have
\begin{equation*}
\iint K\,(g-g')\times(\zeta'_{\beta}\partial_{\alpha'}-\zeta'_{\alpha'}\partial_{\beta'}){f'}\,d\alpha' d\beta'=\iint K \mathcal N' (g-g')\cdot\nabla_\xi^+ f'=(I-\mathcal H)(g\cdot\nabla_\xi^+ f),
\end{equation*}
where in the last step we used the fact that $\partial_{\xi_i}^+ f=\mathcal H \partial_{\xi_i}^+ f$, since $\partial_{\xi_i}^+ f$ is the trace on $\Sigma(t)$ of the analytic function $\partial_{\xi_i} F$, $i=1,2,3$.
\end{proof}

Define 
\begin{equation}\label{hdual}
\mathcal H^*f=-\iint \zeta_\alpha\times\zeta_\beta K(\zeta'-\zeta) f(\alpha',\beta',t)\,d\alpha'\,d\beta'=-\iint\mathcal N K f'\,d\alpha'\,d\beta'
\end{equation}

\begin{proposition}\label{prophdual} For  $\mathcal C(V_2)$ valued smooth functions $f$ and $g$, 
we have 
\begin{equation}\label{hdualeq}
\iint f\cdot\{\mathcal H g\}=\iint \{\mathcal H^* f\}\cdot g\qquad\text{and}
\end{equation}
 \begin{equation}\label{hdualh}
(\mathcal H^*-\mathcal H)f=\iint \{K(\zeta'-\zeta)\cdot(\mathcal N+\mathcal N')+K(\zeta'-\zeta)\times (\mathcal N-\mathcal N')\}f'
\end{equation}

\end{proposition}
\begin{proof} Both identities are straightforward from definition. We omit the details.
\end{proof}
 Let $\sigma_i=\{\sigma\}_i$ denote the $e_i$ component of $\sigma$. 
\begin{lemma}\label{lemmadf}
 Let $\Omega$ be a $C^2$ domain in $\mathbb R^3$ with $\partial\Omega=\Sigma$ being parametrized by $\xi=\xi(\alpha,\beta)$, $(\alpha,\beta)\in \mathbb R^2$, and $N=\xi_\alpha\times\xi_\beta$, 
$\bold n=\frac{ N}{| N|}$. Assume that $F$ is a Clifford analytic function in $\Omega$. Then  the trace of $\nabla F_i
 \,: \, \nabla_\xi F_i=\nabla F_i (\xi(\alpha,\beta))$ 
 satisfies
 \begin{equation}\label{df}
\nabla_\xi F_i=\bold n(-\frac1{|N|}(\xi_\beta\partial_\alpha-\xi_\alpha\partial_\beta)F_i+\{\frac1{|N|}(\xi_\beta\partial_\alpha-\xi_\alpha\partial_\beta)F\}_i)\qquad i=1,2,3.
\end{equation}
\end{lemma}
\begin{proof} 
We know $\mathcal D F=0$ in $\Omega$.
Therefore $\bold n\mathcal D_\xi F=- \bold n\cdot \mathcal D_\xi F+ \bold n\times \mathcal D_\xi F=0$. This implies
$$\bold n\cdot \nabla_\xi F_i=\{\bold n\times\nabla_\xi  F\}_i=\{\frac1{|N|}(\xi_\beta\partial_\alpha-\xi_\alpha\partial_\beta)F\}_i.$$
Therefore
\begin{equation}
\begin{aligned}
\nabla_\xi F_i&= -\bold n\bold n\nabla_\xi F_i=\bold n( \bold n\cdot \nabla_\xi F_i-\bold n\times \nabla_\xi F_i)\\&
= \bold n (\{\frac1{|N|}(\xi_\beta\partial_\alpha-\xi_\alpha\partial_\beta)F\}_i-\frac1{|N|}(\xi_\beta\partial_\alpha-\xi_\alpha\partial_\beta)F_i).
\end{aligned}
\end{equation}
\end{proof}

The following identities give relations among various quantities.
\begin{lemma}\label{identities} We have
\begin{gather}
\overline \lambda+\chi=(\overline{\mathcal H}-\mathcal H)\frak z e_3+\mathcal K\frak z e_3,\qquad \overline {\lambda^*}+\chi=(\overline{\mathcal H}-\mathcal H)\frak z e_3
\label{lambdachi}\\
\partial_\alpha\frak z =-\mathcal N\cdot e_1+(\partial_\alpha\lambda\times \partial_\beta\lambda)\cdot e_1,\qquad \partial_\beta\frak z =-\mathcal N\cdot e_2+(\partial_\alpha\lambda\times \partial_\beta\lambda)\cdot e_2\label{alphaz}\\
 \mathcal N=e_3+\partial_\alpha\lambda\times e_2-\partial_\beta\lambda\times e_1+ \partial_\alpha\lambda\times \partial_\beta\lambda\label{nlambda}\\
  A\mathcal N- e_3=w,
\label{aw}
\end{gather}
\begin{align}
2(\overline u+(\partial_t+b\cdot\nabla_\bot)\chi)&=(\mathcal H-\overline{\mathcal H})\overline u-2 [\partial_t+b\cdot\nabla_\bot, \mathcal H] \frak z e_3\label{uchi}\\
2(\overline w+(\partial_t+b\cdot\nabla_\bot)\frak v)&=(\mathcal H-\overline{\mathcal H})\overline w+[\partial_t+b\cdot\nabla_\bot, \mathcal H-\overline{\mathcal H}]\overline u\notag\\&-2(\partial_t+b\cdot\nabla_\bot)[\partial_t+b\cdot\nabla_\bot, \mathcal H] \frak z e_3\label{wchi}\\
(\mathcal H-\overline{\mathcal H})f&=-2\iint K\cdot\mathcal N' f'+2\iint (K_1\mathcal N'_2-K_2\mathcal N'_1)e_3 f'\label{hbarh}
\end{align}
where $K=K_1 e_1+ K_2 e_2+K_3 e_3$,  $\mathcal N=\mathcal N_1 e_1+ \mathcal N_2 e_2+\mathcal N_3 e_3$, and $f$ is a function.
\end{lemma}
\begin{proof}
\eqref{lambdachi}, 
\eqref{nlambda}, \eqref{hbarh} are straightforward from definition, \eqref{aw} is   \eqref{ww1} with a change of coordinate $U_k^{-1}$. Notice that the $e_3$ component of $\lambda$ is $\frak z$, therefore \eqref{alphaz} follows straightforwardly from \eqref{nlambda}. 

We now derive \eqref{uchi} from the definition of $\pi=(I-\frak H)ze_3$. We have
\begin{equation}\label{241}
\begin{aligned}
2\partial_t\pi&=2(I-\frak H)z_t e_3-2[\partial_t, \frak H] z e_3=(I-\frak H)(\xi_t-\overline\xi_t)-2[\partial_t, \frak H] z e_3\\&=-(I-\frak H)\overline\xi_t-2[\partial_t, \frak H] z e_3=-2\overline\xi_t+(\frak H-\overline{\frak H})\overline\xi_t-
2[\partial_t, \frak H] z e_3
\end{aligned}
\end{equation}
Here in the last step we used \eqref{ww2}. \eqref{uchi} follows from \eqref{241} with a change of coordinate $U_k^{-1}$.  \eqref{wchi} is obtained  by taking derivative $\partial_t+b\cdot\nabla_\bot$ to \eqref{uchi}. 
\end{proof}

In what follows, we let $l\ge 4$, $l+2\le q\le 2l$, $\xi=\xi(\alpha, \beta, t)$, $t\in [0,T]$ be a solution of the water wave system 
\eqref{ww1}-\eqref{ww2}. Assume that the mapping $k(\cdot, t):\mathbb R^2\to \mathbb R^2$ defined in \eqref{k} is a diffeomorphism and its Jacobian $J(k(t))>0$, for $t\in [0, T]$. Assume for $\partial=\partial_\alpha,\ \partial_\beta$,
\begin{equation}\label{assum}
\Gamma^j\partial \lambda, \ \Gamma^j\partial \frak z,\ \Gamma^j(\partial_t+b\cdot\nabla_\bot)\chi, \ \Gamma^j(\partial_t+b\cdot\nabla_\bot)\frak v\in C([0, T], L^2(\mathbb R^2)),\quad \text{for  } |j|\le q.
\end{equation}
Let $t\in [0, T]$ be fixed. Assume that at this time $t$,
\begin{equation}\label{assumm}
\begin{aligned}
&\sum_{|j|\le l+2\atop
\partial=\partial_\alpha,\partial_\beta}(\|\Gamma^j\partial \lambda(t)\|_2+\|\Gamma^j\partial \frak z(t)\|_2
+\|\Gamma^j\frak v(t)\|_2+\|\Gamma^j(\partial_t+b\cdot\nabla_\bot)\frak v(t)\|_2)\le M\\
&|\zeta(\alpha,\beta,t)-\zeta(\alpha', \beta', t)|\ge\frac 14(|\alpha-\alpha'|+|\beta-\beta'|)\qquad\text{for }\alpha,\beta,\alpha',\beta'\in \mathbb R
\end{aligned}
\end{equation}
For the rest of this paper, the inequality $a\lesssim b$ means that there is a constant $c=c(M_0)$ depending on $M_0$, or  a universal constant $c$, such that $a\le c\,b$. $a\simeq b$ means $a\lesssim b$ and $b\lesssim a$.

\begin{lemma}\label{lemmacomgt} We have for $m\le 2l$, and any function $\phi\in C_0^\infty( \mathbb R^2\times [0,T])$,
\begin{equation}\label{estcomgt}
\begin{aligned}
&\|[\partial_t+b\cdot\nabla_\bot, \Gamma^m]\phi(t)\|_2\lesssim \sum_{j\le l+2}\|\Gamma^j b(t)\|_2\sum_{|j|\le m-1\atop \partial=\partial_\alpha,\partial_\beta}\|\partial \Gamma^j\phi(t)\|_2\\&+\sum_{|j|\le m}\|\Gamma^j b(t)\|_2\sum_{|j|\le l+1\atop \partial=\partial_\alpha,\partial_\beta}\|\partial \Gamma^j\phi(t)\|_2+\sum_{|j|\le m-1}\|(\partial_t+b\cdot\nabla_\bot)\Gamma^j\phi(t)\|_2\\
&\|[\partial_t+b\cdot\nabla_\bot, \Gamma^m]\phi(t)\|_2\lesssim \sum_{j\le l+2}\|\Gamma^j b(t)\|_2\sum_{|j|\le m-1\atop\partial=\partial_\alpha,\partial_\beta}\|\partial \Gamma^j\phi(t)\|_2\\&
+\sum_{|j|\le m}\|\Gamma^j b(t)\|_2\sum_{|j|\le l+1\atop\partial=\partial_\alpha,\partial_\beta}\|\partial \Gamma^j\phi(t)\|_2+\sum_{|j|\le m-1}\|\Gamma^j(\partial_t+b\cdot\nabla_\bot)\phi(t)\|_2
\end{aligned}
\end{equation}
\end{lemma}
\begin{proof}
\eqref{estcomgt} is an easy consequence of the identities \eqref{comgt}, \eqref{comjp} and Proposition~\ref{propsobolev}:
\begin{equation*}
[\partial_t+b\cdot\nabla_\bot, \Gamma^m]\phi=\sum_{j=1}^m\Gamma^{m-j}[\partial_t+b\cdot\nabla_\bot,\Gamma]\Gamma^{j-1}\phi
\end{equation*}

\end{proof}

The following proposition gives the $L^2$ estimates of various quantities in terms of that of $\chi$ and $\frak v$. Let $t\in [0, T]$ be the time when \eqref{assumm} holds.
\begin{proposition}\label{propl2est} Let $m\le q$. There is a $M_0>0$, sufficiently small, such that for $M\le M_0$,
\begin{align}
&\sum_{\partial=\partial_\alpha,\partial_\beta}(\|\Gamma^m\partial\lambda(t)\|_2+\|\Gamma^m\partial\lambda^*(t)\|_2+\|\Gamma^m\partial\chi(t)\|_2+\|\Gamma^m\partial\frak z(t)\|_2)\tag*{}\\&
+\|\Gamma^m u(t)\|_2+\|\Gamma^m w(t)\|_2+\|\Gamma^m (\partial_t+b\cdot\nabla_\bot)\lambda(t)\|_2+\|\Gamma^m (\partial_t+b\cdot\nabla_\bot)\lambda^*(t)\|_2
\tag*{}\\&\quad\lesssim \sum_{|j|\le m}(\|(\partial_t+b\cdot\nabla_\bot)\Gamma^j\chi(t)\|_2+\|(\partial_t+b\cdot\nabla_\bot)\Gamma^j\frak v(t)\|_2)\label{zlcl2}\\
&\|\Gamma^m b(t)\|_2+\|\Gamma^m(\partial_t+b\cdot\nabla_\bot)b(t)\|_2+
\|\Gamma^m (A-1)(t)\|_2
\tag*{}\\&\quad\lesssim M_0 \sum_{|j|\le m}(\|(\partial_t+b\cdot\nabla_\bot)\Gamma^j\chi(t)\|_2+\|(\partial_t+b\cdot\nabla_\bot)\Gamma^j\frak v(t)\|_2)\label{bl2}\\&
\sum_{|j|\le m}\|\Gamma^j(\partial_t+b\cdot\nabla_\bot)\chi (t)\|_2+\|\Gamma^j(\partial_t+b\cdot\nabla_\bot)\frak v (t)\|_2\tag*{}\\&\quad
\simeq \sum_{|j|\le m}(\|(\partial_t+b\cdot\nabla_\bot)\Gamma^j\chi(t)\|_2+\|(\partial_t+b\cdot\nabla_\bot)\Gamma^j\frak v(t)\|_2)\label{comgtc}
\end{align}
\end{proposition}
\begin{proof}  We prove Proposition~\ref{propl2est}  in  five steps. Notice that 
\begin{equation}\label{gbarh}
[\Gamma,\overline{\mathcal H}]=e_3[\Gamma,\mathcal H]e_3,\qquad [\Gamma,\mathcal K]=Re[\Gamma, \mathcal H]
\end{equation}

Step 1. We first show that for 
$m\le q$, there is a $M_0$ sufficiently small, such that if
$M\le M_0$, 
\begin{equation}\label{l2lz}
\sum_{|j|\le m\atop\partial=\partial_\alpha,\partial_\beta}\|\Gamma^j\partial \lambda(t)\|_2+\|\Gamma^j\partial\chi(t)\|_2\lesssim \sum_{|j|\le m\atop\partial=\partial_\alpha,\partial_\beta}\|\Gamma^j\partial \frak z(t)\|_2.
\end{equation}

Let $\partial=\partial_\alpha$ or $\partial_\beta$. From the definition $\lambda=(I+\mathcal H)\frak z e_3- \mathcal K\frak z e_3$, we have
$$\partial \lambda =(I+\mathcal H)\partial \frak z e_3+[\partial, \mathcal H]\frak z e_3-[\partial, \mathcal K]\frak z e_3-\mathcal K \partial \frak z e_3$$
Using \eqref{comjp} we get
\begin{equation*}
\begin{aligned}
\Gamma^m\partial\lambda&=\sum_{j=1}^m\Gamma^{m-j}[\Gamma,\mathcal H]\Gamma^{j-1}\partial\frak z e_3+(I+\mathcal H)\Gamma^m\partial \frak z e_3+\Gamma^m[\partial,\mathcal H]\frak z e_3\\&
-\sum_{j=1}^m\Gamma^{m-j}[\Gamma,\mathcal K]\Gamma^{j-1}\partial\frak z e_3-\mathcal K\Gamma^m\partial \frak z e_3-\Gamma^m[\partial,\mathcal K]\frak z e_3
\end{aligned}
\end{equation*}
Therefore from \eqref{comgk}, Lemma~\ref{lemma 1.2}, \eqref{gbarh}, Propositions~\ref{propcomgh}, ~\ref{propcauchy1},~\ref{propcauchy2}, ~\ref{propsobolev}, we have
\begin{equation*}
\begin{aligned}
\|\Gamma^m\partial\lambda(t)\|_2\lesssim &\sum_{|j|\le m\atop\partial=\partial_\alpha,\partial_\beta}\|\Gamma^j\partial \lambda(t)\|_2\sum_{|j|\le l+2}\|\Gamma^j\partial \frak z(t)\|_2\\&
+(1+\sum_{|j|\le l+2\atop\partial=\partial_\alpha,\partial_\beta}\|\Gamma^j\partial \lambda(t)\|_2)\sum_{|j|\le m}\|\Gamma^j\partial \frak z(t)\|_2
\end{aligned}
\end{equation*}
This gives us 
\begin{equation}\label{251}
\sum_{|j|\le m\atop\partial=\partial_\alpha,\partial_\beta}\|\Gamma^j\partial \lambda(t)\|_2\lesssim \sum_{|j|\le m\atop\partial=\partial_\alpha,\partial_\beta}\|\Gamma^j\partial \frak z(t)\|_2
\end{equation}
 when $M_0$ is sufficiently small.  The proof for the part of estimate for $\chi$ in \eqref{l2lz} follows from a similar calculation and an application of  \eqref{251}. We therefore obtain \eqref{l2lz}.
 
 Step 2. We show that for $m\le q$, there is a sufficiently small $M_0$, such that if $M\le M_0$,
 \begin{equation}\label{l2uzc}
 \sum_{|j|\le m}\|\Gamma^j u(t)\|_2\lesssim M_0\sum_{|j|\le m\atop\partial=\partial_\alpha,\partial_\beta}\|\Gamma^j\partial \frak z(t)\|_2
+\sum_{|j|\le m}\|\Gamma^j(\partial_t+b\cdot\nabla_\bot)\chi(t)\|_2.
\end{equation}
\begin{equation}\label{l2wzc}
\begin{aligned}
 \sum_{|j|\le m}\|\Gamma^j w(t)\|_2&\lesssim M_0\sum_{|j|\le m\atop\partial=\partial_\alpha,\partial_\beta}\|\Gamma^j\partial \frak z(t)\|_2\\&
+\sum_{|j|\le m}(\|\Gamma^j(\partial_t+b\cdot\nabla_\bot)\chi(t)\|_2 +  \|\Gamma^j(\partial_t+b\cdot\nabla_\bot)\frak v(t)\|_2).
\end{aligned}
\end{equation}

We first prove \eqref{l2uzc}. From \eqref{uchi}, similar to Step 1 by using \eqref{comjp}, then apply \eqref{comgk}, Lemma~\ref{lemma 1.2}, \eqref{gbarh}, Propositions~\ref{propcomgh}, ~\ref{propcauchy1},~\ref{propcauchy2}, ~\ref{propsobolev}, and furthermore \eqref{l2lz}, we have for $M_0$ sufficiently small,
\begin{equation}\label{252}
\begin{aligned}
\|\Gamma^m(\overline u+(\partial_t+b\cdot\nabla_\bot)\chi)(t)\|_2\lesssim &\sum_{|j|\le m}\|\Gamma^ju(t)\|_2\sum_{|j|\le l+2\atop\partial=\partial_\alpha,\partial_\beta}\|\Gamma^j\partial \frak z(t)\|_2\\&
+\sum_{|j|\le l+2}\|\Gamma^ju(t)\|_2\sum_{|j|\le m\atop\partial=\partial_\alpha,\partial_\beta}\|\Gamma^j\partial \frak z(t)\|_2
\end{aligned}
\end{equation}
Now in \eqref{252} let $m=l+2$. We get for $M_0$ sufficiently small,
\begin{equation}\label{253}
 \sum_{|j|\le l+2}\|\Gamma^j u(t)\|_2\lesssim \sum_{|j|\le l+2}\|\Gamma^j(\partial_t+b\cdot\nabla_\bot)\chi(t)\|_2.
\end{equation}
Applying \eqref{253} to the right hand side of \eqref{252} and we obtain \eqref{l2uzc}.

Similar to the proof of \eqref{l2uzc}, we start from \eqref{wchi}, and use furthermore the estimates \eqref{l2lz},\eqref{l2uzc}, and the fact that $(\partial_t+b\cdot \nabla_\bot )\frak z=u_3$, we have \eqref{l2wzc}.

Step 3. We have for $m\le q$, there is a sufficiently small $M_0$, such that if $M\le M_0$,
\begin{equation}\label{l2azc}
\begin{aligned}
 &\sum_{|j|\le m}(\|\Gamma^j (A-1)(t)\|_2+ \|\Gamma^j(\partial_t+b\cdot\nabla_\bot) b(t)\|_2
 +\|\Gamma^j b(t)\|_2)\\&
 \lesssim M_0\sum_{|j|\le m\atop\partial=\partial_\alpha,\partial_\beta}(\|\Gamma^j\partial \frak z(t)\|_2
+\|\Gamma^j(\partial_t+b\cdot\nabla_\bot)\chi(t)\|_2 +  \|\Gamma^j(\partial_t+b\cdot\nabla_\bot)\frak v(t)\|_2).
\end{aligned}
\end{equation}

Starting from Proposition~\ref{propAb}, the proof of 
\eqref{l2azc} is similar to that in Steps 1 and 2, and uses the results in Steps 1 \& 2. We omit the details. 

We have 

Step 4. There is $M_0$ sufficiently small, such that for $m\le q$, $M\le M_0$, 
\begin{equation}\label{zc}
 \sum_{|j|\le m\atop\partial=\partial_\alpha,\partial_\beta}\|\Gamma^j\partial \frak z(t)\|_2\lesssim \sum_{|j|\le m}(
\|\Gamma^j(\partial_t+b\cdot\nabla_\bot)\chi(t)\|_2 +  \|\Gamma^j(\partial_t+b\cdot\nabla_\bot)\frak v(t)\|_2).
\end{equation}

\eqref{zc} is obtained using \eqref{alphaz}, \eqref{aw}. We have
$$\partial_\alpha\frak z =-w\cdot e_1+(A-1)\mathcal N\cdot e_1+(\partial_\alpha\lambda\times \partial_\beta\lambda)\cdot e_1$$
therefore
\begin{equation*}
\begin{aligned}
\|\Gamma^m&\partial_\alpha \frak z(t)\|_2\lesssim \|\Gamma^m w(t)\|_2+\sum_{|j|\le m}\|\Gamma^j(A-1)(t)\|_2+\\&
\sum_{|j|\le m\atop\partial=\partial_\alpha,\partial_\beta}\|\Gamma^j\partial\lambda(t)\|_2\sum_{|j|\le l+2\atop\partial=\partial_\alpha,\partial_\beta}(\|\Gamma^j(A-1)(t)\|_2+    \|\Gamma^j\partial\lambda(t)\|_2)
\end{aligned}
\end{equation*}
Now we apply estimates in Steps 1-3, we get \eqref{zc}. Finally 

Step 5. Apply Lemma~\ref{lemmacomgt} to $\phi=\chi$ and $\phi=(\partial_t+b\cdot\nabla_\bot)\chi$, and use results in Steps 1-4, we obtain \eqref{comgtc}. From definition we know $\lambda^*=2\frak z e_3-\chi=2\lambda_3 e_3-\chi$ and $(\partial_t+b\cdot\nabla_\bot)\lambda=u-b$. Combine Steps 1-4 and apply \eqref{comgtc},   we obtain \eqref{zlcl2}, \eqref{bl2}. This finishes the proof of Proposition~\ref{propl2est}.

\end{proof}

We now give the $L^\infty$ estimates for various quantities in terms of that of $\nabla_\bot\chi$, and $\nabla_\bot\frak v$. 
Let $t\in [0, T]$ be the time when \eqref{assumm} holds. Define
\begin{equation}\label{energym}
E_m(t)=
\sum_{|j|\le m}(\|(\partial_t+b\cdot\nabla_\bot)\Gamma^j\chi(t)\|_2^2+\|(\partial_t+b\cdot\nabla_\bot)\Gamma^j\frak v(t)\|_2^2)
\end{equation}

\begin{proposition}\label{liest}  There exist a $M_0>0$ small enough, such that if $M\le M_0$,

1. for $2\le m\le l$ we have
\begin{equation}\label{lilzc}
\sum_{|j|\le m\atop \partial=\partial_\alpha,\partial_\beta}(|\Gamma^j\partial \lambda(t)|_\infty+|\Gamma^j\partial \lambda^*(t)|_\infty+|\Gamma^j\partial \frak z(t)|_\infty)\lesssim
\sum_{|j|\le m\atop \partial=\partial_\alpha,\partial_\beta}|\Gamma^j\partial \chi(t)|_\infty;
\end{equation}
2. for $2\le m\le l-1$, we have
\begin{equation}\label{liuc}
\sum_{|j|\le m\atop \partial=\partial_\alpha,\partial_\beta}|\Gamma^j\partial u(t)|_\infty\lesssim
\sum_{|j|\le m\atop \partial=\partial_\alpha,\partial_\beta}(|\Gamma^j\partial \chi(t)|_\infty+|\Gamma^j\partial \frak v(t)|_\infty);
\end{equation}
3. for $2\le m\le l-2$, we have
\begin{align}
 &\sum_{|j|\le m\atop }|\Gamma^j w(t)|_\infty\lesssim
\sum_{|j|\le m+1\atop \partial=\partial_\alpha,\partial_\beta}(|\Gamma^j\partial \chi(t)|_\infty+|\Gamma^j\partial \frak v(t)|_\infty)\label{liwc}\\&
\sum_{|j|\le m}(|\Gamma^j(A-1)(t)|_\infty+|\Gamma^j  (\partial_t+b\cdot\nabla_\bot)b(t)|_\infty)\tag*{}\\&\qquad
\lesssim
E_{m+2}^{1/2}(t)\sum_{|j|\le m+1\atop \partial=\partial_\alpha,\partial_\beta}(|\Gamma^j\partial \chi(t)|_\infty+|\Gamma^j\partial \frak v(t)|_\infty)\label{liatbc}
\end{align}
\begin{equation}\label{libc}
\text{and}\quad \sum_{|j|\le m}|\Gamma^j b(t)|_\infty
\lesssim
E_{m+2}^{1/2}(t)\sum_{|j|\le m+2\atop \partial=\partial_\alpha,\partial_\beta}|\Gamma^j\partial \chi(t)|_\infty;
\end{equation}
4. for $l+1\le m\le q-2$, 
we have
\begin{equation}\label{lilzc1}
\begin{aligned}
&\sum_{|j|\le m\atop\partial=\partial_\alpha,\partial_\beta}
|\Gamma^j\partial\lambda(t)|_\infty+|\Gamma^j\partial\lambda^*(t)|_\infty+|\Gamma^j\partial \frak z(t)|_\infty\lesssim 
\sum_{|j|\le m\atop\partial=\partial_\alpha,\partial_\beta}|\Gamma^j\partial\chi(t)|_\infty\\&\qquad\qquad+E^{1/2}_{m+2}(t)\{\frac1t+ \sum_{|j|\le [\frac{m+2}2]+1\atop\partial=\partial_\alpha,\partial_\beta}|\Gamma^j\partial\chi(t)|_\infty(1+\ln t)\};
\end{aligned}
\end{equation}
5. for $l\le m\le q-4$, 
we have
\begin{equation}\label{liuc1}
\begin{aligned}
&\sum_{|j|\le m\atop\partial=\partial_\alpha,\partial_\beta}
|\Gamma^j\partial u(t)|_\infty\lesssim 
\sum_{|j|\le m\atop\partial=\partial_\alpha,\partial_\beta}|\Gamma^j\partial\frak v(t)|_\infty\\&\qquad\qquad+E^{1/2}_{m+3}(t)\{\frac1t+ \sum_{|j|\le [\frac{m+2}2]+1\atop\partial=\partial_\alpha,\partial_\beta}(|\Gamma^j\partial\chi(t)|_\infty+|\Gamma^j\partial\frak v(t)|_\infty)(1+\ln t)\},
\end{aligned}
\end{equation}

here $[s]$ is the largest integer $<s$. 
\end{proposition}
\begin{proof} We will again use the identities in Lemma~\ref{identities}. 

Step 1. We use \eqref{lambdachi} to prove \eqref{lilzc} for $2\le m\le l$.  Taking derivative  $\partial$  to \eqref{lambdachi}, $\partial=\partial_\alpha$ or $\partial_\beta$, we get
\begin{equation}\label{255}
\partial\overline\lambda+\partial\chi=[\partial, \overline{\mathcal H}-\mathcal H]\frak z e_3+(\overline{\mathcal H}-\mathcal H)\partial \frak z e_3 +[\partial, \mathcal K]\frak z e_3+\mathcal K \partial\frak z e_3
\end{equation}
Using \eqref{commuteah},  \eqref{commutebh}, \eqref{gbarh},  \eqref{comjp}, \eqref{comgk} and Propositions~\ref{propcauchy1}, ~\ref{propsobolev}, we obtain
\begin{equation*}
\begin{aligned}
\sum_{|j|\le m\atop\partial=\partial_\alpha,\partial_\beta}|\Gamma^j&(\partial\overline\lambda+\partial\chi)(t)|_\infty\lesssim \sum_{|j|\le m\atop\partial=\partial_\alpha,\partial_\beta} |\Gamma^j\partial\lambda(t)|_\infty\sum_{|j|\le m+2\atop\partial=\partial_\alpha,\partial_\beta}\|\Gamma^j\partial\frak z(t)\|_2\\&+  \sum_{|j|\le m\atop\partial=\partial_\alpha,\partial_\beta} |\Gamma^j\partial\frak z(t)|_\infty\sum_{|j|\le m+2\atop\partial=\partial_\alpha,\partial_\beta}(\|\Gamma^j\partial\lambda(t)\|_2+\|\Gamma^j\partial\frak z(t)\|_2)
\end{aligned}
\end{equation*}
 Using Proposition~\ref{propl2est}, we have that for $M_0>0$ small enough,
\begin{equation}\label{256}
\sum_{|j|\le m\atop\partial=\partial_\alpha,\partial_\beta}|\Gamma^j\partial\lambda(t)|_\infty\lesssim \sum_{|j|\le m\atop\partial=\partial_\alpha,\partial_\beta}|\Gamma^j\partial\chi(t)|_\infty+M_0
\sum_{|j|\le m\atop\partial=\partial_\alpha,\partial_\beta} |\Gamma^j\partial\frak z(t)|_\infty
\end{equation}
Similar argument also gives that
\begin{equation}\label{257}
\sum_{|j|\le m\atop\partial=\partial_\alpha,\partial_\beta}|\Gamma^j\partial\lambda^*(t)|_\infty\lesssim \sum_{|j|\le m\atop\partial=\partial_\alpha,\partial_\beta}|\Gamma^j\partial\chi(t)|_\infty+M_0
\sum_{|j|\le m\atop\partial=\partial_\alpha,\partial_\beta} |\Gamma^j\partial\frak z(t)|_\infty
\end{equation}
On the other hand, from the definition we have $2\frak z e_3=\lambda^*+\chi$, this implies
\begin{equation}\label{258}
\sum_{|j|\le m\atop\partial=\partial_\alpha,\partial_\beta}|\Gamma^j\partial\frak z(t)|_\infty\lesssim \sum_{|j|\le m\atop\partial=\partial_\alpha,\partial_\beta}|\Gamma^j\partial\chi(t)|_\infty+
\sum_{|j|\le m\atop\partial=\partial_\alpha,\partial_\beta} |\Gamma^j\partial\lambda^*(t)|_\infty
\end{equation}
Combine \eqref{257}, \eqref{258}, we have for $M_0$ small enough, 
$$\sum_{|j|\le m\atop\partial=\partial_\alpha,\partial_\beta}|\Gamma^j\partial\lambda^*(t)|_\infty\lesssim \sum_{|j|\le m\atop\partial=\partial_\alpha,\partial_\beta}|\Gamma^j\partial\chi(t)|_\infty\qquad \text{and}$$
$$ \sum_{|j|\le m\atop\partial=\partial_\alpha,\partial_\beta}|\Gamma^j\partial\frak z(t)|_\infty\lesssim \sum_{|j|\le m\atop\partial=\partial_\alpha,\partial_\beta}|\Gamma^j\partial\chi(t)|_\infty$$
Applying to \eqref{256}, we obtain \eqref{lilzc}.

Step 2. We prove \eqref{liuc} for $2\le m\le l-1$. The argument is similar to Step 1.

Starting from \eqref{uchi}, using \eqref{commuteth}, \eqref{gbarh},  \eqref{comjp}, \eqref{comgk} and Propositions~\ref{propcauchy1}, ~\ref{propcauchy2},  ~\ref{propsobolev}, we have 
\begin{equation*}
\begin{aligned}
|\Gamma^m(\partial_\alpha \overline u+&\partial_\alpha\frak v)(t)|_\infty\lesssim 
\sum_{|j|\le m+2\atop\partial=\partial_\alpha,\partial_\beta} (\|\Gamma^j\partial\lambda(t)\|_2+\|\Gamma^j\partial\frak z(t)\|_2)\sum_{|j|\le m\atop\partial=\partial_\alpha,\partial_\beta}|\Gamma^j\partial u(t)|_\infty\\&
+\sum_{|j|\le m\atop\partial=\partial_\alpha,\partial_\beta} (|\Gamma^j\partial\lambda(t)|_\infty+|\Gamma^j\partial\frak z(t)|_\infty)\sum_{|j|\le m+2\atop\partial=\partial_\alpha,\partial_\beta}\|\Gamma^j\partial u(t)\|_2
\end{aligned}
\end{equation*}
Argue similarly for $\partial_\beta u$ and using \eqref{lilzc} and Proposition~\ref{propl2est}, we obtain for $2\le m\le l-1$ and $M_0$ small enough,
\begin{equation*}
\sum_{|j|\le m\atop \partial=\partial_\alpha,\partial_\beta}|\Gamma^j\partial u(t)|_\infty\lesssim
\sum_{|j|\le m\atop \partial=\partial_\alpha,\partial_\beta}(|\Gamma^j\partial \chi(t)|_\infty+|\Gamma^j\partial \frak v(t)|_\infty).
\end{equation*}

Step 3. We prove \eqref{liwc} and \eqref{liatbc} for $2\le m\le l-2$. 

From \eqref{aw}, $w=(A-1)\mathcal N+\mathcal N-e_3$, using \eqref{nlambda}, \eqref{lilzc} 
and Proposition~\ref{propl2est}, we get
\begin{equation}\label{259}
\sum_{|j|\le m}|\Gamma^j w(t)|_\infty\lesssim
\sum_{|j|\le m\atop\partial=\partial_\alpha,\partial_\beta}(|\Gamma^j(A-1)(t)|_\infty+|\Gamma^j\partial\lambda(t)|_\infty) 
\lesssim  \sum_{|j|\le m\atop\partial=\partial_\alpha,\partial_\beta}(|\Gamma^j(A-1)(t)|_\infty+|\Gamma^j\partial\chi(t)|_\infty).
\end{equation}
On the other hand, from \eqref{A}, using similar argument as in Steps 1 and 2 and using \eqref{lilzc}, \eqref{liuc}, Proposition~\ref{propl2est}, we have for $M_0$ small enough, $2\le m\le l-2$,
\begin{equation}\label{260}
\begin{aligned}
\sum_{|j|\le m}&|\Gamma^j(A-1)(t)|_\infty\lesssim E^{1/2}_{m+2}(t)\sum_{|j|\le m}|\Gamma^j w(t)|_\infty\\&\quad+ 
E^{1/2}_{m+2}(t) (\sum_{|j|\le m+1\atop\partial=\partial_\alpha,\partial_\beta}(|\Gamma^j \partial\chi(t)|_\infty+|\Gamma^j \partial\frak v(t)|_\infty)
\end{aligned}
\end{equation}
Combining \eqref{259}, \eqref{260}, we obtain for $M_0$ small enough,
\begin{equation}\label{259*}
\sum_{|j|\le m}|\Gamma^j(A-1)(t)|_\infty\lesssim E^{1/2}_{m+2}(t) (\sum_{|j|\le m+1\atop\partial=\partial_\alpha,\partial_\beta}(|\Gamma^j \partial\chi(t)|_\infty+|\Gamma^j \partial\frak v(t)|_\infty)
\end{equation}
\eqref{liwc} therefore follows from \eqref{259}, \eqref{259*}. Using \eqref{b}, the estimate for
$\sum_{|j|\le m}|\Gamma^j(\partial_t+b\cdot\nabla_\bot)b(t)|_\infty$  can be obtained similarly. We omit the details.

Step 4. We prove \eqref{libc}. We first put the terms $[\partial_t+b\cdot\nabla_\bot,\mathcal H]\frak z$, $[\partial_t+b\cdot\nabla_\bot,\mathcal K]\frak z$ in \eqref{b} in an appropriate form for carrying out our estimate. We know $2\frak z e_3=\lambda^*+\chi$
and $\lambda^*$ (or $\chi$) is the trace of an analytic function in $\Omega(t)$ (or $\Omega(t)^c$) respectively.  
Using \eqref{commuteth}, and Proposition~\ref{propon}, We have
\begin{equation}\label{261}
\begin{aligned}
2[\partial_t&+b\cdot\nabla_\bot,\mathcal H]\frak z e_3=[\partial_t+b\cdot\nabla_\bot,\mathcal H]\lambda^*+[\partial_t+b\cdot\nabla_\bot,\mathcal H]\chi
\\&=u\cdot\nabla_\xi^+\lambda^* -\mathcal H(u\cdot\nabla_\xi^+\lambda^*)-
u\cdot\nabla_\xi^-\chi -\mathcal H(u\cdot\nabla_\xi^-\chi)
\end{aligned}
\end{equation}
Notice that $[\partial_t+b\cdot\nabla_\bot,\mathcal K]\frak z=\Re [\partial_t+b\cdot\nabla_\bot,\mathcal H]\frak z$. Now using \eqref{b}, and Proposition~\ref{propsobolev}, Lemma~\ref{lemmadf}, Proposition~\ref{propl2est}, \eqref{lilzc}, we obtain
$$\sum_{|j|\le m}|\Gamma^jb(t)|_\infty
\lesssim E^{1/2}_{m+2}(t)\sum_{|j|\le m+2\atop\partial=\partial_\alpha,\partial_\beta}|\Gamma^j\partial\chi(t)|_\infty.
$$
Step 5. We prove \eqref{lilzc1}.  

Let $l+1\le m\le q-2$
 and $\partial=\partial_\alpha,\partial_\beta$. 
Applying Propositions~\ref{propsobolev},~\ref{propcauchy1},  ~\ref{propcauchy3} with $r=t$,  and \eqref{zlcl2},\eqref{lilzc} to \eqref{255}, we get the estimate for $\partial\lambda$:
\begin{equation*}|\Gamma^m\partial\lambda(t)|_\infty\lesssim |\Gamma^m\partial\chi(t)|_\infty+E^{1/2}_{m+2}(t)\{\frac1t+ \sum_{|j|\le [\frac{m+2}2]+1\atop\partial=\partial_\alpha,\partial_\beta}|\Gamma^j\partial\chi(t)|_\infty(1+\ln t)\}\end{equation*}
 Similar argument gives the estimate for $\partial\lambda^*$.
The estimate for $\partial \frak z$ follows since $\frak z=\lambda^*+\chi$. 

Step 6. \eqref{liuc1} is obtained similarly  by using \eqref{uchi}. We omit the details. 

\end{proof}

For the $L^2$, $L^\infty$ estimates of $\frac{\frak a_t}{\frak a}\circ k^{-1}$, we have the following Lemma.
\begin{lemma}\label{lemmaih}
Let $f$ be real valued such that 
\begin{equation}\label{eqih}
(I-\mathcal H)(f\mathcal N)=g,
\end{equation}
$t\in [0, T]$ be the time when \eqref{assumm} holds. There exists a $M_0>0$, such that if $M\le M_0$,

1. for  $0\le m\le l+2$, we have
\begin{equation}\label{ih1}
\sum_{|j|\le m}\|\Gamma^j f(t)\|_{2}\lesssim \sum_{|j|\le m} \|\Gamma^j g(t)\|_{2}.
\end{equation}
2. for $l+2< m\le q$,
\begin{equation}\label{ih2}
\sum_{|j|\le m}\|\Gamma^j f(t)\|_{2}\lesssim \sum_{|j|\le l+2} \|\Gamma^j g(t)\|_{2}\sum_{|j|\le m\atop\partial=\partial_\alpha,\partial_\beta} \|\Gamma^j \partial\lambda(t)\|_{2}     +\sum_{|j|\le m} \|\Gamma^j g(t)\|_{2}.
\end{equation}
3. for $0\le m\le l$,
\begin{equation}\label{ih3}
 \sum_{|j|\le m}|\Gamma^j f(t)|_{\infty}\lesssim \sum_{|j|\le m} |\Gamma^k g(t)|_{\infty}+\sum_{|j|\le 1\atop\partial=\partial_\alpha,\partial_\beta}| \Gamma^j\partial\lambda(t)|_{\infty}\sum_{|j|\le m+2}\|\Gamma^j g(t)\|_{2}
 \end{equation}

\end{lemma}

\begin{proof}
The proof follows similar idea as that of Lemma 3.8 in \cite{wu3}. From \eqref{eqih}, we have
\begin{equation}\label{271}
(I-\mathcal H)((\Gamma^j f)\,\mathcal N)=\Gamma^jg+[\Gamma^j, \mathcal H](f\mathcal N)-(I-\mathcal H)(\Gamma^j (f\mathcal N)-(\Gamma^j f)\mathcal N)
\end{equation}
Let $R=\Gamma^jg+[\Gamma^j, \mathcal H](f\mathcal N)-(I-\mathcal H)(\Gamma^j (f\mathcal N)-(\Gamma^j f)\mathcal N)$. Multiplying $e_3$ both left and right to both sides of \eqref{271}, we obtain
\begin{equation}\label{272}
(I+\overline{\mathcal H})((\Gamma^j f)\,\overline{\mathcal N})=\overline R
\end{equation}
Here we used the fact that $f$ is real valued. Therefore
\begin{equation}\label{273}
\begin{aligned}
2\Gamma^j f e_3&=\Gamma^j f(\overline{\mathcal N}+e_3)+\Gamma^j f(-{\mathcal N}+e_3)+\mathcal H(
\Gamma^j f(\overline{\mathcal N}+{\mathcal N}))\\&+(\overline{\mathcal H}-{\mathcal H})(\Gamma^j f\overline{\mathcal N})+R-\overline R
\end{aligned}
\end{equation}
Lemma~\ref{lemmaih} is then obtained by applying \eqref{comjp}, Lemma~\ref{lemma 1.2}, Proposition~\ref{propcomgh}, \eqref{hbarh}, Propositions~\ref{propcauchy2}, ~\ref{propsobolev}~\ref{propl2est} to \eqref{273} and by an inductive argument. We omit the details.

\end{proof}

Let $\mathcal K^*$ be the adjoint of the double layered potential operator $\mathcal K$:
\begin{equation}\label{dualdoublelayer}
\mathcal K^*f(\alpha,\beta,t)=\iint \mathcal N\cdot K(\zeta'-\zeta) f(\alpha',\beta',t)\,d\alpha'\,d\beta'
\end{equation}
\begin{proposition}\label{propdn}
Let $f(\cdot,t)$ be a real valued function on $\mathbb R^2$. Then 1.
\begin{equation}\label{dn}
(I\pm\mathcal K^*)(\mathcal N\cdot \nabla_\xi^\pm f)=\pm \iint (\mathcal N\times K(\zeta'-\zeta)) \cdot
(\zeta'_{\beta'}\partial_{\alpha'}f'-\zeta'_{\alpha'}
\partial_{\beta'}f')\,d\alpha'd\beta' .
\end{equation}
 2. At $t\in [0,T]$ when \eqref{assumm} holds, \begin{equation}\label{dn1}
 \|\mathcal N\cdot \nabla_\xi^+ f(t)+\mathcal N\cdot \nabla_\xi^- f(t)\|_2\lesssim \sum_{\partial=\partial_\alpha,\partial_\beta}  (|\partial \lambda(t)|_\infty+|\partial\frak z(t)|_\infty)\sum_{\partial=\partial_\alpha,\partial_\beta}\|\partial f(t)\|_2
 \end{equation}
 3. At $t\in [0,T]$ when \eqref{assumm} holds,
 \begin{equation}\label{dn2}
 \|\mathcal N\cdot \nabla_\xi^+ f(t)+\mathcal N\cdot \nabla_\xi^- f(t)\|_2\lesssim \sum_{1\le j\le 3\atop\partial=\partial_\alpha,\partial_\beta}  (|\partial ^j\lambda(t)_\infty+|\partial^j\frak z(t)|_\infty)\|f(t)\|_2
 \end{equation}
 \end{proposition}
 \begin{proof}
From definition, we know $ N\cdot \nabla_\xi^+ f$ and $ N\cdot \nabla_\xi^- f$  are the normal derivatives of the harmonic extensions of $f$ into $\Omega(t)$ and $\Omega(t)^c$ respectively. \eqref{dn} is basically the equality (3.13) in \cite{wu2}. Therefore
 $$\mathcal N\cdot \nabla_\xi^+ f+\mathcal N\cdot \nabla_\xi^- f=\mathcal K^*(-\mathcal N\cdot \nabla_\xi^+ f+\mathcal N\cdot \nabla_\xi^- f)\qquad\text{and}$$
 \begin{equation*}
 \begin{aligned}
 -\mathcal N\cdot \nabla_\xi^+ f&+\mathcal N\cdot \nabla_\xi^- f=\mathcal K^*(\mathcal N\cdot \nabla_\xi^+ f+\mathcal N\cdot \nabla_\xi^- f)\\&-2 
 \iint (\mathcal N\times K(\zeta'-\zeta)) \cdot
(\zeta'_{\beta'}\partial_{\alpha'}f'-\zeta'_{\alpha'}
\partial_{\beta'}f')\,d\alpha'd\beta' .
  \end{aligned}
  \end{equation*}
 This implies
 \begin{equation}\label{274}
 \begin{aligned}
 \mathcal N\cdot \nabla_\xi^+ f&+\mathcal N\cdot \nabla_\xi^- f=\mathcal K^{*2}(\mathcal N\cdot \nabla_\xi^+ f+\mathcal N\cdot \nabla_\xi^- f)\\&-2\mathcal K^* \iint (\mathcal N\times K(\zeta'-\zeta)) \cdot
(\zeta'_{\beta'}\partial_{\alpha'}f'-\zeta'_{\alpha'}
\partial_{\beta'}f')\,d\alpha'd\beta' .
  \end{aligned}
  \end{equation}
From \eqref{274}, \eqref{dn1} is straightforward with an application of Proposition~\ref{propcauchy1}. 

To prove \eqref{dn2}, we rewrite
\begin{equation}\label{275}
\begin{aligned}
&\iint (\mathcal N\times K(\zeta'-\zeta)) \cdot
(\zeta'_{\beta'}\partial_{\alpha'}f'-\zeta'_{\alpha'}
\partial_{\beta'}f')\,d\alpha'd\beta' \\&=\partial_\alpha \iint (\mathcal N\times K) \cdot
\zeta'_{\beta'}f'\,d\alpha'd\beta'-\partial_\beta \iint (\mathcal N\times K) \cdot
\zeta'_{\alpha'}f'\,d\alpha'd\beta'\\&
-\iint (\partial_\alpha +\partial_{\alpha'} )(\mathcal N\times K(\zeta'-\zeta)) \cdot
\zeta'_{\beta'}f'\,d\alpha'd\beta'\\&\quad+\iint(\partial_\beta+\partial_{\beta'}) (\mathcal N\times K(\zeta'-\zeta)) \cdot
\zeta'_{\alpha'}f'\,d\alpha'd\beta'\
\end{aligned}
\end{equation}
 Here we just used integration by parts. \eqref{dn2} now follows from \eqref{274}, \eqref{275} with a further application of integration by parts and an application of  Proposition~\ref{propcauchy1}. (Notice that $\mathcal N\cdot \zeta_\alpha=\mathcal N\cdot \zeta_\beta=0$)
 
\end{proof}

\section{Energy estimates}\label{energy}

In this section, we use the expanded set of vector fields $\Gamma=\{ \partial_t, \quad \partial_\alpha,  \quad \partial_\beta, \quad L_0= \frac 12 t\partial_t+\alpha\partial_\alpha+\beta\partial_\beta,\quad 
 \varpi=\alpha\partial_\beta-\beta\partial_\alpha-\frac12 e_3\}$ to construct energy functional and derive energy estimates for the water wave system \eqref{ww1}-\eqref{ww2}.
  Our strategy is to construct two energy estimates, the first one involves a full range of derivatives
 and we will show using Proposition~\ref{propgsobolev} that it grows no faster than $(1+t)^\epsilon$ provided the energy involving some lower orders of derivatives is bounded by $c\epsilon^2$.  The second one involves the aforementioned lower orders of derivatives, and we will show it stays bounded by $c\epsilon^2$ for all time provided initially it is bounded by $\frac c2\epsilon^2$ and the energy involving the full range of derivatives does not grow faster than $(1+t)^\delta$ for some $\delta<1$. The estimates will be carried out using \eqref{cubic2}, \eqref{cubic3}.

 We first present the following basic energy estimates. The first one will be used to derive the energy estimate for the full range of derivatives, and second one the lower orders of derivatives.
\begin{lemma}[Basic energy inequality I] \label{propbasicenergy1}
Assume that  $\theta$ is  real valued and satisfying
\begin{equation}\label{301}
(\partial_t+b\cdot\nabla_\bot)^2\theta+A\mathcal N\cdot \nabla_\xi^+\theta=\bold G
\end{equation}
and $\theta$ is smooth and decays fast at spatial infinity. Let
\begin{equation}\label{302}
E(t)=\iint\frac1A|(\partial_t+b\cdot\nabla_\bot)\theta(\alpha,\beta,t)|^2+\theta (\mathcal N\cdot \nabla_\xi^+)
\theta(\alpha,\beta,t)\,d\alpha\,d\beta
\end{equation}
Then
\begin{equation}\label{basic1}
\frac{dE}{dt}\le\iint \frac2A\bold G\,(\partial_t+b\cdot\nabla_\bot)\theta\,d\alpha d\beta+(\|\frac{\frak a_t}{\frak a}\circ k^{-1}\|_{L^\infty} +2\|\nabla \bold v(t)\|_{L^\infty(\Omega(t)})E(t)
\end{equation}
\end{lemma}
\begin{proof}
Let $\theta^\hbar$ be the harmonic extension of $\theta$ to $\Omega(t)$. Make a change of coordinate to \eqref{301} and \eqref{302}, and use the Green's identity, we have
$$(\partial_t^2+\frak a N\cdot\nabla^+_\xi)(\theta\circ k)=\bold G\circ k\qquad\text {and}$$
$$E(t)=\iint\frac1{\frak a}|\partial_t(\theta\circ k)|^2\,d\alpha\,d\beta+\int_{\Omega(t)}|\nabla\theta^\hbar|^2\,dV $$
We know
$$\frac{d}{dt}\iint\frac1{\frak a}|\partial_t(\theta\circ k)|^2\,d\alpha\,d\beta=\iint \frac2{\frak a}\partial_t(\theta\circ k)\, \partial^2_t(\theta\circ k)-\frac{\frak a_t}{\frak a^2}|\partial_t(\Theta\circ k)|^2\,d\alpha\,d\beta.$$
To calculate $\frac{d}{dt}\int_{\Omega(t)}|\nabla\theta^\hbar|^2\,dV$,
we introduce the fluid map $X(\cdot,t): \Omega(0)\to \Omega(t)$  satisfying $\partial_tX(\cdot,t)=\bold v(\cdot,t)$, $X(\cdot,0)=I$. From the incompressibility of $\bold v$ we know the Jacobian of $X(\cdot,t)$: $J(X(t))=1$. Let $D_t=\partial_t+\bold v\cdot \nabla$, we know
\begin{equation}\label{304}
D_t\nabla \theta^\hbar-\nabla D_t\theta^\hbar=-\sum_{j=1}^3\nabla\bold v_j\partial_{\xi_j}\theta^\hbar
\end{equation}
Now applying the above calculation, we have
\begin{equation}\label{305}
\begin{aligned}
\frac{d}{dt}&\int_{\Omega(t)}|\nabla\theta^\hbar|^2\,dV=\frac{d}{dt}\int_{\Omega(0)}|\nabla\theta^\hbar(X(\cdot,t),t)|^2\,dV\\&
=2\int_{\Omega(0)}D_t\nabla\theta^\hbar\cdot\nabla\theta^\hbar(X(\cdot,t),t)\,dV=2\int_{\Omega(t)}D_t\nabla\theta^\hbar\cdot\nabla\theta^\hbar\,dV\\&
=2\int_{\Omega(t)}\nabla D_t\theta^\hbar\cdot\nabla\theta^\hbar\,dV-2\sum_{j=1}^3\int_{\Omega(t)}\nabla\bold v_j\partial_{\xi_j}\theta^\hbar\cdot\nabla\theta^\hbar\,dV\\&
=2\iint \partial_t(\theta\circ k) ( N\cdot \nabla_\xi^+)
(\theta\circ k)\,d\alpha\,d\beta
-2\sum_{j=1}^3\int_{\Omega(t)}\nabla \bold v_j\partial_{\xi_j}\theta^\hbar\cdot\nabla\theta^\hbar\,dV
\end{aligned}
\end{equation}
In the last step we used the divergence Theorem. So
\begin{equation}
\frac{dE}{dt}=\int \frac2{\frak a}\partial_t(\theta\circ k)\, \bold G\circ k-\frac{\frak a_t}{\frak a^2}|\partial_t(\Theta\circ k)|^2\,d\alpha\,d\beta-2\sum_{j=1}^3\int_{\Omega(t)}\nabla \bold v_j\partial_{\xi_j}\theta^\hbar\cdot\nabla\theta^\hbar\,dV
\end{equation}
Making a change of variable $U_k^{-1}$ and using the Green's identity gives us \eqref{basic1}.
\end{proof}

\begin{lemma}[Basic energy equality II]\label{propbasicenergy2}
Assume that $\Theta$ is a smooth $\mathcal C(V_2)$ valued function satisfying $ \Theta=-\mathcal H\Theta$, and
\begin{equation}\label{equationenergy}
((\partial_t+ b\cdot \nabla_\bot)^2-A \mathcal N\times \nabla)\Theta=G
\end{equation}
Let
\begin{equation}\label{308}
E(t)=\iint\frac1A|(\partial_t+ b\cdot \nabla_\bot)\Theta|^2-\Theta\cdot\{ (\mathcal N\times \nabla)\Theta\}
(\alpha,\beta, t)\,d\alpha\,d\beta\,\end{equation} 
Then
\begin{equation}\label{basic2}
\begin{aligned}
&\frac{dE}{dt}=\iint\{\frac2AG\cdot\{(\partial_t+ b\cdot \nabla_\bot)\Theta\}-\frac{\frak a_t}{\frak a}\circ k^{-1}\frac1A |(\partial_t+ b\cdot \nabla_\bot)\Theta|^2\}d\alpha d\beta\\&-
\iint\{ (\Theta\cdot(u_\beta\Theta_\alpha)-\Theta\cdot(u_\alpha\Theta_\beta))+\mathcal N\times\nabla\Theta\cdot [\partial_t+b\cdot\nabla_\bot, \mathcal H]\Theta\}\,d\alpha\,d\beta\\&+
\frac12\iint\{ (\mathcal  N\cdot \nabla_\xi^++\mathcal  N\cdot \nabla_\xi^-)\Theta\}\cdot [\partial_t+b\cdot\nabla_\bot, \mathcal H]\Theta\,d\alpha\,d\beta
\end{aligned}
\end{equation}
\end{lemma}
\begin{proof}
By making a change of coordinates $U_k$ we know $\Theta$ satisfies
$$(\partial_t^2-\frak a N\times \nabla )\Theta\circ k=G\circ k\qquad\text{and}$$
$$E(t)=\iint\frac 1{\frak a}(\Theta\circ k)_t\cdot(\Theta\circ k)_t-\Theta\circ k\cdot\{(\xi_\beta\partial_\alpha-\xi_\alpha\partial_\beta) (\Theta\circ k)\}\,d\alpha\,d\beta$$
Therefore
\begin{equation}
\begin{aligned}
\frac{dE}{dt}&=\iint\{\frac 2{\frak a}(\Theta\circ k)_{t}\cdot(\Theta\circ k)_{tt}-\frac{\frak a_t}{{\frak a}^2}|(\Theta\circ k)_t|^2-\Theta\circ k\cdot\{(\xi_{t\beta}\partial_\alpha-\xi_{t\alpha}\partial_\beta) (\Theta\circ k)\}\\&-
(\Theta\circ k)_t\cdot\{(N\times\nabla) (\Theta\circ k)\}
-\Theta\circ k\cdot\{(N\times\nabla) (\Theta\circ k)_t\}\}\,d\alpha\,d\beta
\end{aligned}
\end{equation}
Now from the assumption $\Theta=-\mathcal H\Theta=\frac12(I-\mathcal H)\Theta$, we have $$(\Theta\circ k)_t=\frac12(I-\frak H)(\Theta\circ k)_t-\frac12 [\partial_t, \frak H]\Theta\circ k$$ 
and 
\begin{equation}\label{hdth}
[\partial_t, \frak H]\Theta\circ k=\frak H([\partial_t, \frak H]\Theta\circ k)\end{equation}
Using integration by parts, and the fact that for $\Phi$ satisfying $\Phi=\pm\frak H\Phi$, 
$N\times \nabla \Phi=N\cdot \nabla_\xi^\pm \Phi$,\footnote{ We know $\Phi=\pm\frak H\Phi$ implies $\Phi$ is analytic in  $\Omega(t)$ or $\Omega(t)^c$, i.e. $\mathcal D_\xi \Phi=0$, 
therefore $ N\times \nabla\Phi= N\cdot \nabla_\xi^\pm\Phi$.} and $N\cdot\nabla_\xi^\pm$ is self-adjoint, we have 
\begin{equation*}
\begin{aligned}
&
\iint\Theta\circ k\cdot\{(N\times\nabla) (\Theta\circ k)_t\}\,d\alpha\,d\beta \\&
=\frac12\iint\Theta\circ k\cdot\{(N\times\nabla)   (I-\frak H)(\Theta\circ k)_t\}\,d\alpha\,d\beta -\frac12\iint\Theta\circ k\cdot\{(N\times\nabla)     [\partial_t, \frak H]\Theta\circ k\}\,d\alpha\,d\beta\\&
= \frac12\iint(N\times\nabla)\Theta\circ k\cdot  \{ (I-\frak H)(\Theta\circ k)_t\}\,d\alpha\,d\beta -\frac12\iint(N\cdot\nabla_\xi^+)\Theta\circ k\cdot  \{   [\partial_t, \frak H]\Theta\circ k\}\,d\alpha\,d\beta\\&
=\iint(N\times\nabla)\Theta\circ k\cdot  \{(\Theta\circ k)_t\}\,d\alpha\,d\beta
+\iint(N\times\nabla)\Theta\circ k\cdot  \{  [\partial_t, \frak H]\Theta\circ k  \}\,d\alpha\,d\beta
 \\&-\frac12\iint(N\cdot\nabla_\xi^++N\cdot\nabla_\xi^-)\Theta\circ k\cdot  \{   [\partial_t, \frak H]\Theta\circ k\}\,d\alpha\,d\beta\\&
\end{aligned}
\end{equation*}
Sum up the above calculation and make a change of variable $U_k^{-1}$ gives us \eqref{basic2}.
\end{proof}

In what follows in this section we make the following assumptions on the solution. Let $l\ge 6$, $l+2\le q\le 2l$, $\xi=\xi(\alpha, \beta, t)$, $t\in [0,T]$ be a solution of the water wave system 
\eqref{ww1}-\eqref{ww2}. Assume that the mapping $k(\cdot, t):\mathbb R^2\to\mathbb R^2$ defined in \eqref{k} is a diffeomorphism and its Jacobian $J(k(t))>0$, for $t\in [0, T]$. Assume for $\partial=\partial_\alpha,\ \partial_\beta$, 
\begin{equation}\label{assum1}
\Gamma^j\partial \lambda, \ \Gamma^j\partial \frak z,\ \Gamma^j(\partial_t+b\cdot\nabla_\bot)\chi, \ \Gamma^j(\partial_t+b\cdot\nabla_\bot)\frak v\in C([0, T], L^2(\mathbb R^2)),\quad \text{for  } |j|\le q.
\end{equation}
and
\begin{equation}\label{assumm1}
\begin{aligned}
&\sup_{[0,T]}\sum_{|j|\le l+2\atop
\partial=\partial_\alpha,\partial_\beta}(\|\Gamma^j\partial \lambda(t)\|_2+\|\Gamma^j\partial \frak z(t)\|_2
+\|\Gamma^j\frak v(t)\|_2+\|\Gamma^j(\partial_t+b\cdot\nabla_\bot)\frak v(t)\|_2)\le M\\
&|\zeta(\alpha,\beta,t)-\zeta(\alpha', \beta', t)|\ge\frac 14(|\alpha-\alpha'|+|\beta-\beta'|)\qquad\text{for }\alpha,\beta,\alpha'\beta'\in \mathbb R
\end{aligned}
\end{equation}
where $0<M\le M_0$, $M_0$ is the constant such that all the estimates derived in subsection~\ref{relations} holds and such that $|A-1|\le \frac 12$.

We have the following 
\begin{lemma}\label{lemmaliest}
Let $2\le m\le \min\{2l-7, q-5\}$, $t\in [0, T]$. There exists $M_0>0$ sufficiently small, such that for $M\le M_0$,
\begin{equation}\label{ineqliest}
t\sum_{|i|\le m\atop\partial=\partial_\alpha,\partial_\beta}(|\partial\Gamma^i \chi(t)|_\infty+ |\partial\Gamma^i \frak v(t)|_\infty)\lesssim E_{5+m}^{1/2}(t)(1+ t \sum_{|i|\le [\frac {m+3}2]+2\atop\partial=\partial_\alpha,\partial_\beta}|\partial\Gamma^i \chi(t)|_\infty+ |\partial\Gamma^i \frak v(t)|_\infty)
\end{equation}
In particular, for $5\le m\le \min\{2l-11, q-5\}$, we have
\begin{equation}\label{li}
\sum_{|i|\le m\atop\partial=\partial_\alpha,\partial_\beta}(|\partial\Gamma^i \chi(t)|_\infty+ |\partial\Gamma^i \frak v(t)|_\infty)\lesssim \frac 1{1+t} E_{5+m}^{1/2}(t)
\end{equation}
\end{lemma}
\begin{proof}
Let  $t\in [0, T]$, $i\le m\le \min\{2l-7, q-5\}$. 
From  Propositions~\ref{propgsobolev} and~\ref{propl2est}, we have
\begin{equation}
t(|\partial\Gamma^i \chi(t)|_\infty+ |\partial\Gamma^i \frak v(t)|_\infty)\lesssim E_{5+i}^{1/2}(t)+t\sum_{|k|\le 3+i}(\|\frak P\Gamma^k\chi(t)\|_2+\|\frak P\Gamma^k\frak v(t)\|_2)
\end{equation}
We estimate $\|\frak P\Gamma^k\chi(t)\|_2$ and $\|\frak P\Gamma^k\frak v(t)\|_2$ using \eqref{cubic2}, \eqref{cubic3}. We know for $\phi=\chi,\ \frak v$, 
$$\frak P\Gamma^k\phi=\Gamma^k\mathcal P\phi+[\mathcal P, \Gamma^k]\phi+(\frak P-\mathcal P)\Gamma^k\phi$$
Let $k\le 3+m\le \min\{2l-4, q-2\}$. We know $E_{2+l}(t)\lesssim M_0^2$. Using \eqref{comjp}, \eqref{commgp}, \eqref{371}, Propositions~\ref{propl2est},~\ref{liest}, we have for $\phi=\chi,\ \frak v$, 
\begin{equation}\label{3-1}
\|[\mathcal P, \Gamma^k]\phi(t)\|_2\lesssim E_k^{1/2}(t) \sum_{|i|\le [\frac k2]+2\atop\partial=\partial_\alpha,\partial_\beta}(|\partial\Gamma^i \chi(t)|_\infty+ |\partial\Gamma^i \frak v(t)|_\infty),\qquad\text{and}
\end{equation}
\begin{equation}\label{3-2}
\begin{aligned}
&\|(\frak P-\mathcal P) \Gamma^k\chi(t)\|_2\lesssim E_{k+1}^{1/2}(t) \sum_{|i|\le 4\atop\partial=\partial_\alpha,\partial_\beta}(|\partial\Gamma^i \chi(t)|_\infty+ |\partial\Gamma^i \frak v(t)|_\infty)
\\&
\|(\frak P-\mathcal P) \Gamma^k\frak v(t)\|_2\lesssim E_{k+2}^{1/2}(t) \sum_{|i|\le 4\atop\partial=\partial_\alpha,\partial_\beta}(|\partial\Gamma^i \chi(t)|_\infty+ |\partial\Gamma^i \frak v(t)|_\infty).
\end{aligned}
\end{equation}
We now estimate $\|\Gamma^k\mathcal P\chi(t)\|_2$. From \eqref{cubic2}, \eqref{comgq}, we know there are two types of terms in $\Gamma^k\mathcal P\chi$. One are terms of cubic and higher orders. Collectively, we name such terms as $C$. Another type are quadratic terms of the following form:
$$Q_j=\iint K(\zeta'-\zeta)(\dot\Gamma^j u-{\dot\Gamma}'^j u')\times (\zeta'_{\beta'}\partial_{\alpha'}- \zeta'_{\alpha'}\partial_{\beta'}){\Gamma'}^{k-j}\bar {u'}\,d\alpha'\,d\beta'$$
For the cubic terms $C$, we use Propositions~\ref{propsobolev},~\ref{propcauchy1},~\ref{propcauchy2},~\ref{propl2est},~\ref{liest}. We have
\begin{equation}\label{3-3}
\|C(t)\|_2\lesssim E_{k}^{1/2}(t) \sum_{|i|\le [\frac k2]+2\atop\partial=\partial_\alpha,\partial_\beta}(|\partial\Gamma^i \chi(t)|_\infty+ |\partial\Gamma^i \frak v(t)|_\infty).
\end{equation}
For the quadratic terms $Q_j$ with $j\le [\frac k2]+1$, we use Proposition~\ref{propcauchy2}, and ~\ref{propl2est},~\ref{liest}. We have
\begin{equation}\label{3-4}
\|Q_j(t)\|_2\lesssim E_{k}^{1/2}(t) \sum_{|i|\le [\frac k2]+1\atop\partial=\partial_\alpha,\partial_\beta}(|\partial\Gamma^i \chi(t)|_\infty+ |\partial\Gamma^i \frak v(t)|_\infty).
\end{equation}
To estimate $\|Q_j(t)\|_2$ for $k\ge j> [\frac k2]+1$, we rewrite it by using Proposition~\ref{propon}:
\begin{equation*}
\begin{aligned}
&Q_j= \frac 12\iint K(\zeta'-\zeta)(\dot\Gamma^j u-{\dot\Gamma}'^j u')\times (\zeta'_{\beta'}\partial_{\alpha'}- \zeta'_{\alpha'}\partial_{\beta'}){((I+\mathcal H')+(I-\mathcal H'))\Gamma'}^{k-j}\bar {u'}\,d\alpha'\,d\beta'
\\&= \frac 12 (I-\mathcal H)((\dot\Gamma^j u\cdot\nabla^+_\xi)(I+\mathcal H)\Gamma^{k-j}\bar u)-\frac 12 (I+\mathcal H)((\dot\Gamma^j u\cdot\nabla^-_\xi)(I-\mathcal H)\Gamma^{k-j}\bar u)
\end{aligned}
\end{equation*}
Applying Proposition~\ref{propcauchy1}, Lemma~\ref{lemmadf}, we obtain
\begin{equation*}
\begin{aligned}
\|Q_j(t)\|_2\lesssim &\|\dot\Gamma^j u(t)\|_2 \sum_{\partial=\partial_\alpha,\partial_\beta}(|\partial(I+\mathcal H)\Gamma^{k-j}\bar u(t)|_\infty+|\partial(I-\mathcal H)\Gamma^{k-j}\bar u(t)|_\infty)\\&
\lesssim \|\dot\Gamma^j u(t)\|_2 \sum_{\partial=\partial_\alpha,\partial_\beta}(|\partial(I+\mathcal H)\Gamma^{k-j}\bar u(t)|_\infty+|\partial\Gamma^{k-j}\bar u(t)|_\infty)
\end{aligned}
\end{equation*}
We know from the fact $-\overline{\mathcal H}\bar u=\bar u$\,\footnote{ \eqref{ww2} gives $\mathcal Hu=u$, therefore  $-\overline{\mathcal H}\bar u=\bar u$.} that
$$(I+\mathcal H)\Gamma^{k-j}\bar u=\Gamma^{k-j} (-\bar{\mathcal H}+\mathcal H)\bar u-[\Gamma^{k-j},\mathcal H]\bar u$$
Applying \eqref{hbarh}, \eqref{comjp}, Propositions~\ref{propsobolev},~\ref{propcauchy1},~\ref{propcauchy2}, ~\ref{propl2est},~\ref{liest},
 we get
 $$|\partial(I+\mathcal H)\Gamma^{k-j}\bar u(t)|_\infty\lesssim E_{k-j+3}^{1/2}(t) \sum_{|i|\le k-j+2\atop\partial=\partial_\alpha,\partial_\beta}|\partial\Gamma^i \chi(t)|_\infty\lesssim \sum_{|i|\le k-j+2\atop\partial=\partial_\alpha,\partial_\beta}|\partial\Gamma^i \chi(t)|_\infty
 $$
 Here we used the fact $E_{l+2}(t)\lesssim M_0^2$. Therefore for $k\ge j> [\frac k2]+1$,
 \begin{equation}\label{3-5}
  \|Q_j(t)\|_2\lesssim  E_{k}^{1/2}(t) \sum_{|i|\le [\frac k2]+2\atop\partial=\partial_\alpha,\partial_\beta}(|\partial\Gamma^i \chi(t)|_\infty+ |\partial\Gamma^i \frak v(t)|_\infty)
  \end{equation}
  Sum up \eqref{3-3}-\eqref{3-5}, we have
  \begin{equation}\label{3-6}
  \|\Gamma^k\mathcal P\chi(t)\|_2\lesssim E_{k}^{1/2}(t) \sum_{|i|\le [\frac k2]+2\atop\partial=\partial_\alpha,\partial_\beta}(|\partial\Gamma^i \chi(t)|_\infty+ |\partial\Gamma^i \frak v(t)|_\infty)
  \end{equation}
A  similar argument gives that
  \begin{equation}\label{3-7}
  \|\Gamma^k\mathcal P\frak v(t)\|_2\lesssim E_{k+1}^{1/2}(t) \sum_{|i|\le [\frac k2]+2\atop\partial=\partial_\alpha,\partial_\beta}(|\partial\Gamma^i \chi(t)|_\infty+ |\partial\Gamma^i \frak v(t)|_\infty)
  \end{equation}
Combine \eqref{3-1}-\eqref{3-7}, we obtain 
\begin{equation}\label{3-10}
\begin{aligned}
&\|\frak P\Gamma^k\chi(t)\|_2\lesssim E_{k+1}^{1/2}(t)\sum_{|i|\le \max\{[\frac k2]+2,4\}\atop\partial=\partial_\alpha,\partial_\beta}(|\partial\Gamma^i \chi(t)|_\infty+ |\partial\Gamma^i \frak v(t)|_\infty)\\&
\|\frak P\Gamma^k\frak v(t)\|_2\lesssim E_{k+2}^{1/2}(t)\sum_{|i|\le \max\{[\frac k2]+2,4\}\atop\partial=\partial_\alpha,\partial_\beta}(|\partial\Gamma^i \chi(t)|_\infty+ |\partial\Gamma^i \frak v(t)|_\infty).
\end{aligned}
\end{equation}
This gives us
\eqref{ineqliest}. For $t\le 1$ \eqref{li} can be obtained from the Sobolev embedding and Proposition~\ref{propl2est}. For $t\ge 1$, \eqref{li} is obtained by first applying \eqref{ineqliest} to the case 
$5\le m\le l-3$, i.e. when $[\frac{m+3}2]+2\le m$ and $5+m\le l+2$ and using the fact $ E_{l+2}(t)\lesssim M_0^2$; then applying \eqref{ineqliest}  to the case
$m\le \min\{2l-11, q-5\}$. We know in this case $[\frac{m+3}2]+2\le l-3$.
  
 \end{proof}
 
 In what follows we establish two energy estimates. 
\subsection{The first energy estimate}
The first energy estimate concerns a full range of derivatives. We  use Lemma~\ref{propbasicenergy1} and \eqref{cubic2}, \eqref{cubic3}.

Assume that $\phi$ is a $\mathcal C(V_2)$ valued function satisfying the equation 
\begin{equation}\label{300}
((\partial_t+ b\cdot \nabla_\bot)^2-A \mathcal N\times \nabla)\phi=G^\phi.
\end{equation}
Let $\Phi^j=(I-\mathcal H)\Gamma^j\phi$. We know for $\mathcal P=(\partial_t+ b\cdot \nabla_\bot)^2-A \mathcal N\times \nabla$,
\begin{equation}\label{300j}
\begin{aligned}
((\partial_t+ b\cdot \nabla_\bot)^2-A \mathcal N\times \nabla)\Phi^j=-[\mathcal P,\mathcal H]\Gamma^j\phi+(I-\mathcal H)[\mathcal P, \Gamma^j]\phi+(I-\mathcal H)\Gamma^j G^\phi.
\end{aligned}
\end{equation}
Notice that $\Phi^j=-\mathcal H\Phi^j$ implies  $\mathcal N\times \nabla\Phi^j=\mathcal N\cdot \nabla_\xi^-\Phi^j$. Therefore  
\begin{equation}\label{309}
((\partial_t+ b\cdot \nabla_\bot)^2+A \mathcal N\cdot \nabla_\xi^+)\Phi^j=\bold G_j^\phi
\end{equation}
where
$$\bold G_j^\phi=-[\mathcal P,\mathcal H]\Gamma^j\phi+(I-\mathcal H)[\mathcal P, \Gamma^j]\phi+(I-\mathcal H)\Gamma^j G^\phi+A(\mathcal N\cdot \nabla_\xi^++\mathcal N\cdot \nabla_\xi^-)\Phi^j$$
Define
\begin{equation}\label{fj1}
F_j^\phi(t)=\iint\frac1A|(\partial_t+b\cdot\nabla_\bot)\Phi^j(\alpha,\beta,t)|^2+\Phi^j\cdot (\mathcal N\cdot \nabla_\xi^+)
\Phi^j(\alpha,\beta,t)\,d\alpha\,d\beta
\end{equation}
We know $ \iint\Phi^j\cdot (\mathcal N\cdot \nabla_\xi^+)
\Phi^j(\alpha,\beta,t)\,d\alpha\,d\beta=\int_{\Omega(t)}|\nabla\{\Phi^j\}^\hbar|^2\,dV    \ge 0$. 
Let
\begin{equation}\label{fullenergy}
\mathcal F_n(t)=\sum_{|j|\le n} (F_j^{\frak v}(t)+F_j^\chi(t))
\end{equation}
and $V^j=(I-\mathcal H)\Gamma^j\frak v$, $\Pi^j=(I-\mathcal H)\Gamma^j\chi$. We have
\begin{lemma}\label{lemmaef1} Let $n\le q$, $t\in [0,T]$.
There exists $M_0>0$ small enough, such that for $M\le M_0$,
\begin{equation}\label{ef1}
\sum_{|j|\le n} \iint \frac 1A(|(\partial_t+b\cdot\nabla_\bot)V^j(\alpha,\beta,t)|^2+|(\partial_t+b\cdot\nabla_\bot)\Pi^j(\alpha,\beta,t)|^2) \,d\alpha\,d\beta\simeq E_n(t)
\end{equation}
\end{lemma}
\begin{proof} 
Notice that $\Gamma^j\phi=\frac12 (I+\mathcal H)\Gamma^j\phi+\frac12\Phi^j$. 
We know $\mathcal H\chi=-\chi$. So for $\phi=\chi,\ \frak v$,
\begin{equation}\label{310}(I+\mathcal H)\Gamma^j\chi=-[\Gamma^j, \mathcal H]\chi,\qquad (I+\mathcal H)\Gamma^j\frak v=-[\Gamma^j, \mathcal H]\frak v-\Gamma^j[\partial_t+b\cdot\nabla_\bot,\mathcal H]\chi
\end{equation}
\eqref{ef1} follows by applying Lemma~\ref{lemma 1.2}, Proposition~\ref{propcomgh}, \eqref{comjp},
Propositions~\ref{propcauchy1}, ~\ref{propcauchy2}, ~\ref{propsobolev}, and \eqref{zlcl2} to \eqref{310}.
\end{proof}
We now state the following energy estimate.
\begin{proposition}\label{propfullenergy}
Let $3\le n\le \min\{2l-4, q\}$, 
$t\in [0,T]$. There exists $M_0>0$ sufficiently small, such that for $M\le M_0$,
\begin{equation}\label{fullest}
\frac{d\mathcal F_n(t)}{dt}\lesssim  \sum_{|i|\le [\frac n2]+2\atop \partial=\partial_\alpha,\partial_\beta}(|\Gamma^i\partial \chi(t)|_\infty+|\Gamma^i\partial \frak v(t)|_\infty)
 \mathcal F_n(t)
\end{equation}
\end{proposition}
\begin{proof}
Let $\phi=\chi,\ \frak v$, $|j|\le n$. From \eqref{309}, applying Lemma~\ref{propbasicenergy1} to each component of $\Phi^j$ then sum up, we get
\begin{equation}
\frac{dF_j^\phi(t)}{dt}\lesssim 
\| \bold G_j^\phi(t)\|_2 \{F_j^\phi(t)\}^{1/2}  +(\|\frac{\frak a_t}{\frak a}\circ k^{-1}\|_{L^\infty} +2\|\nabla \bold v(t)\|_{L^\infty(\Omega(t)})F_j^\phi(t)
\end{equation}
Notice that $\bold v$ is Clifford analytic in $\Omega(t)$. From Lemma~\ref{lemmadf}, and the maximum principle, we have 
\begin{equation}\label{311}
\|\nabla \bold v(t)\|_{L^\infty(\Omega(t))}\lesssim |\partial_\alpha u(t)|_\infty+|\partial_\beta u(t)|_\infty.
\end{equation}
Applying Lemma~\ref{lemmaih} \eqref{ih3}, Proposition~\ref{propsobolev},~\ref{propcauchy1} ~\ref{propcauchy2}, ~\ref{propl2est} ~\ref{liest} to \eqref{at}, we obtain
\begin{equation}\label{312}
|\frac{\frak a_t}{\frak a}\circ k^{-1}|_{\infty}\lesssim  \sum_{|i|\le 3\atop \partial=\partial_\alpha,\partial_\beta}(|\Gamma^i\partial \chi(t)|_\infty+|\Gamma^i\partial \frak v(t)|_\infty).
\end{equation}
We now estimate $\| \bold G_j^\phi(t)\|_2 $ for $\phi=\chi,\ \frak v$. We carry it out in four steps.
Let 
\begin{equation}\label{G}
\bold G_j^\phi=G_{j,1}^\phi+G_{j,2}^\phi+G_{j,3}^\phi+G_{j,4}^\phi
\end{equation}
where $G_{j,1}^\phi=-[\mathcal P,\mathcal H]\Gamma^j\phi$, $G_{j,2}^\phi=(I-\mathcal H)[\mathcal P, \Gamma^j]\phi$, $G_{j,3}^\phi=(I-\mathcal H)\Gamma^j G^\phi$, and $G_{j,4}^\phi=A(\mathcal N\cdot \nabla_\xi^++\mathcal N\cdot \nabla_\xi^-)\Phi^j$.

Step 1. We have 
\begin{equation}\label{313}
\begin{aligned}
&\|G_{j,1}^\chi(t)\|_2\lesssim \sum_{\partial=\partial_\alpha,\partial_\beta}
|\partial u(t)|_\infty(\|\partial \Gamma^j\chi(t)\|_2+\|(\partial_t+b\cdot\nabla_\bot)\Gamma^j\chi(t)\|_2)
\\&
\|G_{j,1}^{\frak v}(t)\|_2\lesssim \sum_{\partial=\partial_\alpha,\partial_\beta} 
|\partial u(t)|_\infty(\| \Gamma^j\frak v(t)\|_2+\|(\partial_t+b\cdot\nabla_\bot)\Gamma^j\frak v(t)\|_2)
\end{aligned}
\end{equation}
This is obtained by using Lemma~\ref{lemma 1.2} and applying Propositions~\ref{propcauchy1},~\ref{propcauchy2},~\ref{propsobolev}.

Step 2. We have that for $\phi=\chi,\ \frak v$,
\begin{equation}\label{314}
\|G_{j,2}^\phi(t)\|_2\lesssim \sum_{|i|\le [\frac n2]+2\atop\partial=\partial_\alpha,\partial_\beta}(|\Gamma^i\partial \chi(t)|_\infty+|\Gamma^i\partial \frak v(t)|_\infty)E_{|j|}(t)^{1/2}
\end{equation}
This is basically \eqref{3-1}.\footnote{Notice that \eqref{3-1}, \eqref{3-6} (used in Step 4.) in fact hold  for $k\le \min\{2l-4,q\}$.}
 The estimate for the operator $(I-\mathcal H)$ can be obtained by applying Proposition~\ref{propcauchy1}. 

Step 3. From Propositions~\ref{propdn}, ~\ref{propcauchy1}, we have
\begin{equation}\label{315}
\begin{aligned}
&\|G_{j,4}^\chi(t)\|_2\lesssim \sum_{\partial=\partial_\alpha,\partial_\beta}  (|\partial \lambda(t)_\infty+|\partial\frak z(t)|_\infty)\sum_{\partial=\partial_\alpha,\partial_\beta}\|\partial\Gamma^j \chi(t)\|_2\\&
\|G_{j,4}^\frak v(t)\|_2\lesssim \sum_{1\le i\le 3\atop\partial=\partial_\alpha,\partial_\beta}  (|\partial ^i\lambda(t)_\infty+|\partial^i\frak z(t)|_\infty)\|\Gamma^j\frak v(t)\|_2
\end{aligned}
\end{equation}

Step 4.  We have that
\begin{align}
&\|G_{j,3}^\chi(t)\|_2\lesssim \sum_{|i|\le [\frac n2]+2\atop\partial=\partial_\alpha,\partial_\beta}(|\Gamma^i\partial \chi (t)|_\infty+|\Gamma^i\partial \frak v(t)|_\infty)E^{1/2}_{|j|}(t).\label{318}\\&
\|G_{j,3}^\frak v(t)\|_2\lesssim \sum_{|i|\le [\frac n2]+2\atop\partial=\partial_\alpha,\partial_\beta}(|\Gamma^i\partial \chi (t)|_\infty+|\Gamma^i\partial \frak v(t)|_\infty)E_{|j|}^{1/2}(t)\label{319}
\end{align}

\eqref{318} is obtained by using Proposition~\ref{propcauchy1} and \eqref{3-6}.

However we cannot derive \eqref{319} from \eqref{3-7}, since there is a loss of derivative in \eqref{3-7}. This "loss of derivative" is due to the term $Au_\beta\chi_\alpha-A u_\alpha\chi_\beta$ in $G^\frak v$. To obtain \eqref{319} we need to take advantage of the projection operator $I-\mathcal H$.
We rewrite the term
\begin{equation}\label{322}
\begin{aligned}
&(I-\mathcal H)\Gamma^j(Au_\beta\chi_\alpha-Au_\alpha\chi_\beta)=
(I-\mathcal H)(A\Gamma^ju_\beta\,\chi_\alpha-A\Gamma^ju_\alpha\,\chi_\beta)\\&+
(I-\mathcal H)(\Gamma^j(Au_\beta\chi_\alpha-Au_\alpha\chi_\beta)
-A\Gamma^ju_\beta\,\chi_\alpha+A\Gamma^ju_\alpha\,\chi_\beta)
\end{aligned}
\end{equation}
in which we further rewrite, using the fact $u=\mathcal Hu$, 
\begin{equation}
\begin{aligned}
&(I-\mathcal H)(A\Gamma^ju_\beta\,\chi_\alpha-A\Gamma^ju_\alpha\,\chi_\beta)=[\Gamma^j\partial_\beta, \mathcal H]u \,A\chi_\alpha-[\Gamma^j\partial_\alpha, \mathcal H]u \,A\chi_\beta\\&+\sum_{i=1}^3([A\partial_\alpha\chi_i,\mathcal H]\Gamma^j\partial_\beta u -[A\partial_\beta\chi_i,\mathcal H]\Gamma^j\partial_\alpha u)e_i
\end{aligned}
\end{equation}
Here $\chi_i$ is the $e_i$ component of $\chi$. Now with all the terms in appropriate forms, \eqref{319} results by applying Propositions~\ref{propcauchy1},~\ref{propcauchy2},~\ref{propsobolev}, and Proposition~\ref{propl2est},~\ref{liest}.

Sum up Steps 1-4 and \eqref{311}, \eqref{312}, and applying Propositions~\ref{propl2est},~\ref{liest}, Lemma~\ref{lemmaef1}, we get \eqref{fullest}.
\end{proof} 

\subsection{The second energy estimate} We now give an estimate that involves some lower orders of derivatives. We use Lemma~\ref{propbasicenergy2} and \eqref{cubic2}, \eqref{cubic3}.

Assume that $\phi$ satisfies \eqref{300} and let $\Phi^j=(I-\mathcal H)\Gamma^j\phi$. We know   $\Phi^j$ satisfies \eqref{300j}, i.e.
$$\mathcal P\Phi^j=G_j^\phi$$
where 
\begin{equation}
G_j^\phi=-[\mathcal P,\mathcal H]\Gamma^j\phi+(I-\mathcal H)[\mathcal P, \Gamma^j]\phi+(I-\mathcal H)\Gamma^j G^\phi
\end{equation}
Define
\begin{equation}
\bold F^\phi_j(t)= \iint\frac1A|(\partial_t+ b\cdot \nabla_\bot)\Phi^j|^2-\Phi^j\cdot (\mathcal N\times\nabla)\Phi^j
(\alpha,\beta, t)\,d\alpha\,d\beta\label{fj2}
\end{equation}
We know $-\iint\Phi^j\cdot (\mathcal N\times\nabla)\Phi^j
(\alpha,\beta, t)\,d\alpha\,d\beta=\int_{\Omega(t)^c}|\nabla \{\Phi^j\}^\hbar|^2\,dV\ge 0$.
Let
\begin{equation}\label{partenergy}
\frak F_n(t)=\sum_{|j|\le n} (\bold F_j^{\frak v}(t)+\bold F_j^\chi(t))
\end{equation}
We have 
\begin{proposition}\label{proppartenergy}
Let $l\ge 15$, $q\ge l+9$, $t\in [0, T]$. There exists $M_0>0$ small enough, such that for $M\le M_0$,
\begin{equation}\label{partenergyineq}
\frac{d\frak F_{l+2}(t)}{dt}\lesssim E_{l+2}^{1/2}(t)E_{l+3}^{1/2}(t)E_{l+9}^{1/2}(t)(\frac{1+\ln(t+1)}{t+1})^2
\end{equation}
\end{proposition}
\begin{proof}
 Let $\phi=\chi,\frak v$, $|j|\le l+2$. We also use $j$ to indicate $|j|$ in this proof. Assume $t\ge 1$. The argument can be easily modified for $t\le 1$.
Applying Lemma~\ref{propbasicenergy2} to $\Phi^j$, we have 
\begin{equation}\label{340}
\begin{aligned}
&\frac{d\bold F_j^\phi(t)}{dt}=\iint\{\frac2AG_j^\phi\cdot\{(\partial_t+ b\cdot \nabla_\bot)\Phi^j\}-\frac{\frak a_t}{\frak a}\circ k^{-1}\frac1A |(\partial_t+ b\cdot \nabla_\bot)\Phi^j|^2\}d\alpha d\beta\\&-
\iint\{ (\Phi^j\cdot(u_\beta\Phi^j_\alpha)-\Phi^j\cdot(u_\alpha\Phi^j_\beta))+\mathcal N\times\nabla\Phi^j\cdot [\partial_t+b\cdot\nabla_\bot, \mathcal H]\Phi^j\}\,d\alpha\,d\beta\\&+
\frac12\iint\{ (\mathcal  N\cdot \nabla_\xi^++\mathcal  N\cdot \nabla_\xi^-)\Phi^j\}\cdot [\partial_t+b\cdot\nabla_\bot, \mathcal H]\Phi^j\,d\alpha\,d\beta
\end{aligned}
\end{equation}
Using Lemma~\ref{lemmaih}, \eqref{at}, Propositions~\ref{propcauchy1}, ~\ref{propcauchy2},~\ref{propsobolev}, ~\ref{propcauchy3} with $r=t$, and Propositions~\ref{propl2est},~\ref{liest}, we obtain
\begin{equation}\label{341}
\begin{aligned}
&|\frac{\frak a_t}{\frak a}\circ k^{-1}A(t)|_\infty\le \sum_{|i|\le 1\atop\partial=\partial_\alpha,\partial_\beta}|\Gamma^i\partial u(t)|_\infty \sum_{|i|\le 1}|\Gamma^i w(t)|_\infty(1+\ln t)\\&
+\sum_{|i|\le 1\atop\partial=\partial_\alpha,\partial_\beta}(|\Gamma^i\partial u(t)|_\infty +|\Gamma^i w(t)|_\infty)(\|\partial u(t)\|_2+\| w(t)\|_2) \frac 1t \\&
+\sum_{|i|\le 2\atop\partial=\partial_\alpha,\partial_\beta}|\Gamma^i\partial u(t)|^2_\infty +\sum_{|i|\le 1\atop\partial=\partial_\alpha,\partial_\beta}|\Gamma^i\partial\lambda(t)|_\infty( \sum_{|i|\le 2}|\Gamma^i w(t)|_\infty+\sum_{|i|\le 2\atop\partial=\partial_\alpha,\partial_\beta}|\Gamma^i\partial u(t)|^2_\infty)\\&\le
\sum_{|i|\le 3\atop\partial=\partial_\alpha,\partial_\beta}(|\Gamma^i\partial\chi(t)|_\infty+|\Gamma^i\partial\frak v(t)|_\infty)\{(|\Gamma^i\partial\chi(t)|_\infty+|\Gamma^i\partial\frak v(t)|_\infty)(1+\ln t)+ E_1(t)^{1/2}\frac 1t\}
\end{aligned}
\end{equation}
We now estimate $\iint \mathcal N\times\nabla\Phi^j\cdot [\partial_t+b\cdot\nabla_\bot, \mathcal H]\Phi^j\,d\alpha\,d\beta$. We know 
$
[\partial_t+b\cdot\nabla_\bot, \mathcal H]\Phi^j=(I+\mathcal H)[\partial_t+b\cdot\nabla_\bot, \mathcal H]\Gamma^j \phi
$.
Therefore using Proposition~\ref{prophdual}, we have 
\begin{equation}\label{342}
\iint \mathcal N\times\nabla\Phi^j\cdot [\partial_t+b\cdot\nabla_\bot, \mathcal H]\Phi^j\,d\alpha\,d\beta=\iint\{ (I+\mathcal H^*)\mathcal N\times\nabla\Phi^j\}\cdot [\partial_t+b\cdot\nabla_\bot, \mathcal H]\Gamma^j \phi\,d\alpha\,d\beta
\end{equation}
Now 
\begin{equation*}
\begin{aligned}
 (I+\mathcal H^*)\mathcal N\times\nabla\Phi^j&=(\mathcal H^*-\mathcal H)\mathcal N\times\nabla\Phi^j+(I+\mathcal H)\mathcal N\times\nabla\Phi^j\\&=(\mathcal H^*-\mathcal H)\mathcal N\times\nabla\Phi^j-[\mathcal N\times\nabla, \mathcal H]\Phi^j.
 \end{aligned}
 \end{equation*}
Using \eqref{commuteah}, Proposition~\ref{propcauchy1}, we get
 \begin{equation}\label{343}
 \|(I+\mathcal H^*)\mathcal N\times\nabla\Phi^j(t)\|_2\lesssim \sum_{\partial=\partial_\alpha,\partial_\beta}(|\partial \lambda(t)|_\infty+|\partial\frak z(t)|_\infty)\|\partial \Gamma^j\phi(t)\|_2.
 \end{equation}
On the other hand, we have from \eqref{commuteth}, Proposition~\ref{propcauchy3} with $r=t$, 
\begin{equation}\label{344}
\|[\partial_t+b\cdot\nabla_\bot, \mathcal H]\Gamma^j \phi(t)\|_2\lesssim \|u(t)\|_2(\sum_{|i|\le j+1\atop\partial=\partial_\alpha,\partial_\beta}|\partial\Gamma^i\phi(t)|_\infty(1+\ln t)+\sum_{|i|\le j\atop\partial=\partial_\alpha,\partial_\beta}\|\partial\Gamma^i\phi(t)\|_2\frac1t).
\end{equation}
Combining \eqref{342}-\eqref{344}, and further use Propositions~\ref{propl2est},\ref{liest}, we obtain
\begin{equation}\label{345}
\begin{aligned}
&|\iint \mathcal N\times\nabla\Phi^j\cdot [\partial_t+b\cdot\nabla_\bot, \mathcal H]\Phi^j\,d\alpha\,d\beta|\\&\lesssim E^{1/2}_{j+1}(t)E^{1/2}_1(t)\sum_{|i|\le 2\atop\partial=\partial_\alpha,\partial_\beta}|\Gamma^i\partial\chi(t)|_\infty
\\&\qquad\qquad\times
(\sum_{|i|\le j+1\atop\partial=\partial_\alpha,\partial_\beta}(|\Gamma^i\partial\chi(t)|_\infty+|\Gamma^i\partial\frak v(t)|_\infty)(1+\ln t)+E_{j+1}^{1/2}(t)\frac 1t).
\end{aligned}
\end{equation}
The estimate of the term $\iint\{ (\mathcal  N\cdot \nabla_\xi^++\mathcal  N\cdot \nabla_\xi^-)\Phi^j\}\cdot [\partial_t+b\cdot\nabla_\bot, \mathcal H]\Phi^j\,d\alpha\,d\beta$ can be obtained from Proposition~\ref{propdn} and \eqref{344}. We have
\begin{equation}\label{346}
\begin{aligned}
&|\iint\{ (\mathcal  N\cdot \nabla_\xi^++\mathcal  N\cdot \nabla_\xi^-)\Phi^j\}\cdot [\partial_t+b\cdot\nabla_\bot, \mathcal H]\Phi^j\,d\alpha\,d\beta|
\\&\lesssim E^{1/2}_{j}(t)E^{1/2}_1(t)\sum_{|i|\le 2\atop\partial=\partial_\alpha,\partial_\beta}|\Gamma^i\partial\chi(t)|_\infty\\&\qquad\qquad\times
(\sum_{|i|\le j+1\atop\partial=\partial_\alpha,\partial_\beta}(|\Gamma^i\partial\chi(t)|_\infty+|\Gamma^i\partial\frak v(t)|_\infty)(1+\ln t)+E_{j+1}^{1/2}(t)\frac 1t).
\end{aligned}
\end{equation}
Now 
\begin{equation}\label{347}
\begin{aligned}
\iint\frac1AG_j^\phi\cdot\{(\partial_t+ b\cdot \nabla_\bot)\Phi^j\}&d\alpha d\beta=\iint(\frac1A-1)G_j^\phi\cdot\{(\partial_t+ b\cdot \nabla_\bot)\Phi^j\}d\alpha d\beta\\&+\iint G_j^\phi\cdot\{(\partial_t+ b\cdot \nabla_\bot)\Phi^j \}d\alpha d\beta.
\end{aligned}
\end{equation}
The term $\iint(\frac1A-1)G_j^\phi\cdot\{(\partial_t+ b\cdot \nabla_\bot)\Phi^j\}d\alpha d\beta$ can be estimated as the following:
$$|\iint(\frac1A-1)G_j^\phi\cdot\{(\partial_t+ b\cdot \nabla_\bot)\Phi^j\}d\alpha d\beta|\lesssim |A-1|_\infty\|G_j^\phi(t)\|_2\|(\partial_t+ b\cdot \nabla_\bot)\Phi^j(t)\|_2$$
We know $G_j^\phi=G_{j,1}^\phi+G_{j,2}^\phi+G_{j,3}^\phi$, where $G_{j,i}^\phi$ $i=1,2,3$ are as defined in \eqref{G}. Using \eqref{313}, \eqref{314}, \eqref{318}, \eqref{319}, and notice that  the $n$ in these inequalities can be replaced by $j$. We have by further applying Proposition~\ref{liest} that
\begin{equation}\label{348}
\begin{aligned}
&|\iint(\frac1A-1)G_j^\phi\cdot\{(\partial_t+ b\cdot \nabla_\bot)\Phi^j\}d\alpha d\beta|\\&\lesssim E_{j}(t)\sum_{|i|\le 3\atop\partial=\partial_\alpha,\partial_\beta}(|\Gamma^i\partial\chi(t)|_\infty+|\Gamma^i\partial\frak v(t)|_\infty)\sum_{|i|\le [\frac {j}2]+2\atop\partial=\partial_\alpha,\partial_\beta}(|\Gamma^i\partial\chi(t)|_\infty+|\Gamma^i\partial\frak v(t)|_\infty).
\end{aligned}
\end{equation}
We now estimate the terms $\iint G_j^\phi\cdot\{(\partial_t+ b\cdot \nabla_\bot)\Phi^j \}d\alpha d\beta$ and  $\iint\{ (\Phi^j\cdot(u_\beta\Phi^j_\alpha)-\Phi^j\cdot(u_\alpha\Phi^j_\beta))\}\,d\alpha\,d\beta$ for $\phi=\chi, \ \frak v$.  We carry out the estimates through six steps.

Step 1. We consider the term $\iint G_{j,1}^\phi\cdot\{(\partial_t+ b\cdot \nabla_\bot)\Phi^j \}d\alpha d\beta$ for $\phi=\chi,\ \frak v$.

We first put the term $\iint G_{j,1}^\phi\cdot\{(\partial_t+ b\cdot \nabla_\bot)\Phi^j \}d\alpha d\beta$ in the right form for estimates. We know $(\partial_t+ b\cdot \nabla_\bot)\Phi^j=(I-\mathcal H)(\partial_t+ b\cdot \nabla_\bot)\Gamma^j\phi-[\partial_t+ b\cdot \nabla_\bot,\mathcal H]\Gamma^j\phi$. Using Proposition~\ref{prophdual}, we have 
\begin{equation}\label{350}
\begin{aligned}
&\iint G_{j,1}^\phi\cdot\{(\partial_t+ b\cdot \nabla_\bot)\Phi^j \}d\alpha d\beta= \iint \{(I-\mathcal H)G_{j,1}^\phi\}\cdot\{(\partial_t+ b\cdot \nabla_\bot)\Gamma^j\phi \}d\alpha d\beta\\&
+\iint \{(\mathcal H-\mathcal H^*)G_{j,1}^\phi\}\cdot\{(\partial_t+ b\cdot \nabla_\bot)\Gamma^j\phi \}d\alpha d\beta\\&-
\iint G_{j,1}^\phi\cdot\{[\partial_t+ b\cdot \nabla_\bot,\mathcal H]\Gamma^j\phi \}d\alpha d\beta
\end{aligned}
\end{equation}
where by applying \eqref{commutetth}, \eqref{commutenh}, and the change of variable $U_k^{-1}$, we know
\begin{equation}\label{349}
\begin{aligned}
&-G_{j,1}^\phi=2\iint K(\zeta'-\zeta)\,(u-u')\times(\zeta'_{\beta'}\partial_{ \alpha'}-
\zeta'_{\alpha'}\partial_{ \beta'})(\partial_t'+b'\cdot\nabla'_\bot){\Gamma'}^j\phi'\,d\alpha'd\beta'\\&+\iint K(\zeta'-\zeta)\,\{(u-u')\times(u'_{\beta'}\partial_{ \alpha'}-u'_{\alpha'}\partial_{\beta'}){\Gamma'}^j\phi'\}\,d\alpha'd\beta'\\&+\iint
((u'-u)\cdot\nabla)K(\zeta'-\zeta) (u-u')\times(\zeta'_{\beta'}\partial_\alpha'-\zeta'_{\alpha'}\partial_\beta'){\Gamma'}^j\phi'
\,d\alpha'd\beta' 
\end{aligned}
\end{equation}
Let $J_1^\phi=2\iint K(\zeta'-\zeta)\,(u-u')\times(\zeta'_{\beta'}\partial_{ \alpha'}-
\zeta'_{\alpha'}\partial_{ \beta'})(\partial_t'+b'\cdot\nabla'_\bot){\Gamma'}^j\phi'\,d\alpha'd\beta'$. 
To estimate the term $\iint \{(I-\mathcal H)G_{j,1}^\phi\}\cdot\{(\partial_t+ b\cdot \nabla_\bot)\Gamma^j\phi \}d\alpha d\beta$,   we use Proposition~\ref{propon} to further rewrite 
\begin{equation}\label{351}
(I-\mathcal H)J_1^\phi= (I-\mathcal H)\iint K\,(u-u')\times(\zeta'_{\beta'}\partial_{ \alpha'}-
\zeta'_{\alpha'}\partial_{ \beta'})(I+\mathcal H')(\partial_t'+b'\cdot\nabla'_\bot){\Gamma'}^j\phi'\,d\alpha'd\beta',
\end{equation}
and notice that $\mathcal H \chi=-\chi$, so for $\phi=\chi, \frak v$, 
\begin{equation}\label{352}
\begin{aligned}
&(I+\mathcal H)(\partial_t+b\cdot\nabla_\bot){\Gamma}^j\chi=-[\partial_t+b\cdot\nabla_\bot,\mathcal H]\Gamma^j\chi-(\partial_t+b\cdot\nabla_\bot)[\Gamma^j, \mathcal H]\chi,
\\& (I+\mathcal H)(\partial_t+b\cdot\nabla_\bot){\Gamma}^j\frak v=(I+\mathcal H)[\partial_t+b\cdot\nabla_\bot, \Gamma^j]\frak v\\&\qquad\qquad\qquad+
[\mathcal H,\Gamma^j](\partial_t+b\cdot\nabla_\bot)\frak v -\Gamma^j[(\partial_t+b\cdot\nabla_\bot)^2,\mathcal H]\chi.
\end{aligned}
\end{equation}

With \eqref{350}--\eqref{352}, $ \iint G_{j,1}^\phi\cdot\{(\partial_t+ b\cdot \nabla_\bot)\Phi^j \}d\alpha d\beta$ is in the right form for estimates. Using Lemma~\ref{lemma 1.2}, Proposition~\ref{propcomgh}, \eqref{comjp}, Propositions~\ref{propcauchy1},~\ref{propcauchy2},~\ref{propcauchy3} with $r=t$, and ~\ref{propl2est}, we get
\begin{equation}
\begin{aligned}
&\|(I-\mathcal H)J_1^\chi(t)\|_2\lesssim \sum_{\partial=\partial_\alpha,\partial_\beta}|\partial u(t)|_\infty
\|(I+\mathcal H)(\partial_t+b\cdot\nabla_\bot){\Gamma}^j\chi(t)\|_2 \\&\lesssim
\sum_{\partial=\partial_\alpha,\partial_\beta}|\partial u(t)|_\infty E_j^{1/2}(t)\sum_{|i|\le j\atop\partial=\partial_\alpha,\partial_\beta}|\Gamma^i\partial\lambda(t)|_\infty\\&
+\sum_{\partial=\partial_\alpha,\partial_\beta}|\partial u(t)|_\infty E_j^{1/2}(t)(\sum_{|i|\le j+1\atop\partial=\partial_\alpha,\partial_\beta}|\partial\Gamma^i\chi(t)|_\infty(1+\ln t)+E^{1/2}_j(t)\frac1t)
\end{aligned}
\end{equation}
and 
\begin{equation}
\begin{aligned}
&\|(I-\mathcal H)J_1^\frak v(t)\|_2\lesssim \sum_{\partial=\partial_\alpha,\partial_\beta}|\partial u(t)|_\infty E_j^{1/2}(t)\sum_{|i|\le j\atop\partial=\partial_\alpha,\partial_\beta} |\partial\Gamma^i\lambda(t)|_\infty+\\&
\sum_{\partial=\partial_\alpha,\partial_\beta}|\partial u(t)|_\infty E_j^{1/2}(t)(\sum_{|i|\le [\frac{ j}2]+2\atop\partial=\partial_\alpha,\partial_\beta}(|\partial\Gamma^i\chi(t)|_\infty+|\partial\Gamma^i\frak v(t)|_\infty)(1+\ln t)+E^{1/2}_{j}(t)\frac1t)
\end{aligned}
\end{equation}
Applying Propositions~\ref{propcauchy1},~\ref{propcauchy2},~\ref{propcauchy3} with $r=t$, and ~\ref{propl2est},~\ref{liest} to other terms in \eqref{350} and using \eqref{313}, \eqref{344}, we obtain for $\phi=\chi,\ \frak v$,
\begin{equation}\label{355}
\begin{aligned}
&|\iint G_{j,1}^\phi\cdot\{(\partial_t+ b\cdot \nabla_\bot)\Phi^j \}d\alpha d\beta|\lesssim E_j(t)\sum_{|i|\le 2\atop\partial=\partial_\alpha,\partial_\beta}(|\partial\Gamma^i\chi(t)|_\infty+|\partial\Gamma^i\frak v(t)|_\infty)\times\\&
 \{(E_{j+2}^{1/2}(t)\sum_{|i|\le[\frac {j}2]+2 \atop\partial=\partial_\alpha,\partial_\beta}|\partial\Gamma^i\chi(t)|_\infty  +\sum_{|i|\le j+1\atop\partial=\partial_\alpha,\partial_\beta}(|\partial\Gamma^i\chi(t)|_\infty+|\partial\Gamma^i\frak v(t)|_\infty))(1+\ln t)+E^{1/2}_{j+2}(t)\frac1t\}
\end{aligned}
\end{equation}

Step 2. We consider the term $\iint G_{j,3}^\phi\cdot\{(\partial_t+ b\cdot \nabla_\bot)\Phi^j \}d\alpha d\beta$ for $\phi=\chi$.

From \eqref{cubic2}, we know $G^\chi$ consists of three terms $G^\chi=I_1+I_2+I_3$. In particular, the first term
\begin{equation}\label{316}
\begin{aligned}
&I_1=\iint K(\zeta'-\zeta)\,(u-u')\times
(\zeta'_{\beta'}\partial_{\alpha'}-\zeta'_{\alpha'}\partial_{\beta'})\overline{u'}\,d\alpha'd\beta'\\&=\frac12
\iint K(\zeta'-\zeta)\,(u-u')\times
(\zeta'_{\beta'}\partial_{\alpha'}-\zeta'_{\alpha'}\partial_{\beta'})\{(I+\mathcal H')\overline{u'}+(I-\mathcal H')\overline {u'}\}\,d\alpha'd\beta'.
\end{aligned}
\end{equation}
 Rewriting
\begin{equation}\label{317}
(I-\mathcal H)\Gamma^j I_1=[\Gamma^j,\mathcal H]I_1+\Gamma^j(I-\mathcal H) I_1,
\end{equation}
where using Proposition~\ref{propon}, we deduce
\begin{equation}\label{320}
\begin{aligned}
(I-\mathcal H)I_1&=(I-\mathcal H)\frac12
\iint K(\zeta'-\zeta)\,(u-u')\times
(\zeta'_{\beta'}\partial_{\alpha'}-\zeta'_{\alpha'}\partial_{\beta'})\{(I+\mathcal H')\overline {u'}\}\,d\alpha'd\beta'\\&
=\iint K(\zeta'-\zeta)\,(u-u')\times
(\zeta'_{\beta'}\partial_{\alpha'}-\zeta'_{\alpha'}\partial_{\beta'})\{(I+\mathcal H')\overline {u'}\}\,d\alpha'd\beta';
\end{aligned}
\end{equation}
furthermore from \eqref{ww2}, we know $(I+\mathcal H)\overline u=(-\overline{\mathcal H}+\mathcal H)\overline u$.
Now with \eqref{316}-\eqref{320}, all the terms in $ G_{j,3}^\chi(t)=(I-\mathcal H)\Gamma^j G^\chi$ are in appropriate forms for carrying out estimates.
Using  Propositions~\ref{propcauchy1},~\ref{propcauchy2},~\ref{propcauchy3}, and ~\ref{propl2est}, we obtain
\begin{equation*}
\begin{aligned}
\|G_{j,3}^\chi(t)\|_2
\lesssim &E_j^{1/2}(t)\sum_{|i|\le j\atop\partial=\partial_\alpha,\partial_\beta}(|\Gamma^i\partial\lambda(t)|_\infty+|\Gamma^i\partial\frak z(t)|_\infty)\\&\qquad\qquad\times
(\sum_{|i|\le [\frac{j}2]+2\atop\partial=\partial_\alpha,\partial_\beta}|\Gamma^i\partial u(t)|_\infty(1+\ln t)+E_{j}^{1/2}(t)\frac1t)
\end{aligned}
\end{equation*}
Further using Proposition~\ref{liest}, we get for $\phi=\chi$,
\begin{equation}\label{356}
\begin{aligned}
&|\iint G_{j,3}^\phi\cdot\{(\partial_t+ b\cdot \nabla_\bot)\Phi^j \}d\alpha d\beta|\\&\lesssim E_j(t)\{\sum_{|i|\le j\atop\partial=\partial_\alpha,\partial_\beta}|\Gamma^i\partial\chi(t)|_\infty+E_{j+2}^{1/2}(t)(\sum_{|i|\le [\frac{j}2]+2\atop\partial=\partial_\alpha,\partial_\beta}|\Gamma^i\partial \chi(t)|_\infty(1+\ln t)+\frac1t)\}\\&\qquad\qquad\times
(\sum_{|i|\le [\frac{j}2]+2\atop\partial=\partial_\alpha,\partial_\beta}(|\Gamma^i\partial \chi(t)|_\infty+|\Gamma^i\partial\frak v(t)|_\infty)(1+\ln t)+E_{j}^{1/2}(t)\frac1t)
\end{aligned}
\end{equation}

Step 3. We consider the term $\iint G_{j,3}^\phi\cdot\{(\partial_t+ b\cdot \nabla_\bot)\Phi^j \}d\alpha d\beta$ for $\phi=\frak v$.

From \eqref{cubic3}, we know
\begin{equation}\label{357}
G_{j,3}^{\frak v}=(I-\mathcal H)\Gamma^j( \frac{\frak a_t}{\frak a}\circ k^{-1}A \mathcal N\times \nabla\chi+A(u_\beta\chi_\alpha-u_\alpha\chi_\beta)+(\partial_t+ b\cdot \nabla_\bot)G^\chi)
\end{equation}
In this step, we will only consider the estimates of $\|(I-\mathcal H)\Gamma^j( \frac{\frak a_t}{\frak a}\circ k^{-1}A \mathcal N\times \nabla\chi)(t)\|_2$, $\|(I-\mathcal H)\Gamma^j( (\partial_t+ b\cdot \nabla_\bot)G^\chi)(t)\|_2$ 
and $\|(I-\mathcal H)\Gamma^j((A-1)(u_\beta\chi_\alpha-u_\alpha\chi_\beta))(t)\|_2$. We will
leave the estimate of $\|(I-\mathcal H)\Gamma^j(u_\beta\chi_\alpha-u_\alpha\chi_\beta)(t)\|_2$ to Step 5.

First, we have by using \eqref{at}, Lemma~\ref{lemmaih}, and Propositions~\ref{propcauchy1},~\ref{propcauchy2},~\ref{propcauchy3}, ~\ref{propsobolev}, and ~\ref{propl2est},~\ref{liest} that
\begin{equation}\label{358}
\begin{aligned}
\|(I-\mathcal H)\Gamma^j( \frac{\frak a_t}{\frak a}\circ k^{-1}A &\mathcal N\times \nabla\chi)(t)\|_2\lesssim E_j^{1/2}(t)\sum_{|i|\le [\frac{j}2]+2\atop\partial=\partial_\alpha,\partial_\beta}(|\Gamma^i\partial\chi(t)|_\infty+|\Gamma^i\partial\frak v(t)|_\infty)\\&\times
(\sum_{|i|\le [\frac{j}2]+2\atop\partial=\partial_\alpha,\partial_\beta}(|\Gamma^i\partial\chi(t)|_\infty+|\Gamma^i\partial\frak v(t)|_\infty)(1+\ln t)+E_j^{1/2}(t)\frac1t)
\end{aligned}
\end{equation}
And using Propositions~\ref{propl2est},~\ref{liest}, we have
\begin{equation}\label{-364}
\begin{aligned}
&\|(I-\mathcal H)\Gamma^j((A-1)(u_\beta\chi_\alpha-u_\alpha\chi_\beta))(t)\|_2\lesssim \\&\qquad
E_{j+1}^{1/2}(t) \sum_{|i|\le [\frac{j}2]+2\atop\partial=\partial_\alpha,\partial_\beta}(|\Gamma^i\partial\chi(t)|_\infty+|\Gamma^i\partial\frak v(t)|_\infty)^2
\end{aligned}
\end{equation}
We handle the estimate of $\|(I-\mathcal H)\Gamma^j( (\partial_t+ b\cdot \nabla_\bot)G^\chi)(t)\|_2$  similar to Step 2 by rewriting
 the term $(I-\mathcal H)\Gamma^j(\partial_t+ b\cdot \nabla_\bot)I_1$, where $I_1$ is as defined in \eqref{316}, as the following:
\begin{equation}\label{321}
\begin{aligned}
(I-\mathcal H)\Gamma^j(\partial_t+b\cdot\nabla_\bot)I_1&=[\Gamma^j,\mathcal H](\partial_t+b\cdot\nabla_\bot)I_1+\Gamma^j[\partial_t+b\cdot\nabla_\bot, \mathcal H]I_1\\&+\Gamma^j(\partial_t+b\cdot\nabla_\bot)(I-\mathcal H) I_1
\end{aligned}
\end{equation}
and use \eqref{320} to calculate $(I-\mathcal H) I_1$. 
 We get, by using Propositions~\ref{propcauchy1},,~\ref{propcauchy2},~\ref{propcauchy3} with $r=t$, and ~\ref{propl2est},~\ref{liest} that
 \begin{equation}\label{359}
 \begin{aligned}
 &\|(I-\mathcal H)\Gamma^j( (\partial_t+ b\cdot \nabla_\bot)G^\chi)(t)\|_2\lesssim E_{j+1}^{1/2}(t)\\&\qquad\qquad\times(\sum_{|i|\le [\frac{j}2]+2\atop\partial=\partial_\alpha,\partial_\beta}(|\Gamma^i\partial\chi(t)|_\infty+|\Gamma^i\partial\frak v(t)|_\infty)(1+\ln t)+E_j^{1/2}(t)\frac1t)^2
  \end{aligned}
 \end{equation}

Step 4. We consider the term $\iint G_{j,2}^\phi\cdot\{(\partial_t+ b\cdot \nabla_\bot)\Phi^j \}d\alpha d\beta$ for $\phi=\chi$ and $\frak v$.

We know $$G_{j,2}^\phi=(I-\mathcal H)[\mathcal P, \Gamma^j]\phi=\sum_{k=1}^j(I-\mathcal H)\Gamma^{j-k}[\mathcal P, \Gamma]\Gamma^{k-1}\phi.$$
A further expansion of \eqref{commgp} gives that for $\Gamma=\partial_t,\partial_\alpha,\partial_\beta, \varpi$, 
\begin{equation}
\begin{aligned}
& [ \Gamma, \mathcal P]=- \{(\dot\Gamma(A-1)(\zeta_\beta\partial_\alpha
-\zeta_\alpha\partial_\beta)+A(\partial_\beta\dot\Gamma\lambda\partial_\alpha-\partial_\alpha\dot\Gamma\lambda\partial_\beta) \}+\\&
\{\ddot\Gamma(\partial_t+b\cdot\nabla_\bot )b-\ddot\Gamma b\cdot\nabla_\bot b\}\cdot\nabla_\bot+
\ddot\Gamma b\cdot\{(\partial_t+b\cdot\nabla_\bot)\nabla_\bot+ \nabla_\bot(\partial_t+b\cdot\nabla_\bot)\}
\end{aligned}
\end{equation}
where  $\dot\Gamma f=\partial_t f,\partial_\alpha f,\partial_\beta f, \varpi f+\frac 12 f e_3$, $\ddot\Gamma f=\partial_t f,\partial_\alpha f,\partial_\beta f, (\varpi -\frac 12 e_3)f$ respectively.  Also 
\begin{equation}
\begin{aligned}
&[ L_0, \mathcal P]=-\mathcal P- \{L_0(A-1)(\zeta_\beta\partial_\alpha
-\zeta_\alpha\partial_\beta)+ A (\partial_\beta(L_0-I)\lambda\partial_\alpha-\partial_\alpha(L_0-I)\lambda\partial_\beta)   \}\\&
+\{L_0(\partial_t+b\cdot\nabla_\bot)b-  (L_0 b-\frac12 b)\cdot\nabla_\bot b\}\cdot\nabla_\bot
\\&
+(L_0 b-\frac12 b) 
\cdot\{  (\partial_t+b\cdot\nabla_\bot)\nabla_\bot+ \nabla_\bot(\partial_t+b\cdot\nabla_\bot)\}.
\end{aligned}
\end{equation}
Therefore typically there are three types of terms in $[\mathcal P, \Gamma^j]\phi=\sum_{k=1}^j\Gamma^{j-k}[\mathcal P, \Gamma]\Gamma^{k-1}\phi  $. 
The first type $(C)$ are of cubic and higher orders and are consists of the following:
\begin{equation}\label{360}
\begin{aligned}
&\Gamma^{j-k}\{\Gamma^i(A-1)(\zeta_\beta\partial_\alpha
-\zeta_\alpha\partial_\beta)\}\Gamma^{k-1}\phi,\qquad{i=0,1},\  k=1,\dots, j\\&
\Gamma^{j-k}\{\Gamma^i(\partial_t+b\cdot\nabla_\bot)b\}\cdot \nabla_\bot\Gamma^{k-1}\phi,\quad \Gamma^{j-k}\{\Gamma^ib\cdot\nabla_\bot b\}\cdot \nabla_\bot\Gamma^{k-1}\phi\\&
\Gamma^{j-k}\{\Gamma^i b\cdot((\partial_t+b\cdot\nabla_\bot)\nabla_\bot+ \nabla_\bot(\partial_t+b\cdot\nabla_\bot))\}\Gamma^{k-1}\phi,
\end{aligned}
\end{equation}
The second type $(Q)$ are quadratic and are consists of the following:
\begin{equation}\label{361}
\begin{aligned}
&\Gamma^{j-k}\{A(\partial_\beta\Gamma^i\lambda\partial_\alpha\Gamma^{k-1}\phi-\partial_\alpha\Gamma^i\lambda\partial_\beta\Gamma^{k-1}\phi)\}\quad i=0, 1, \ k=1,\dots, j\\&
\Gamma^{j-k}\{A(\partial_\beta\lambda e_3\partial_\alpha\Gamma^{k-1}\phi-\partial_\alpha\lambda e_3\partial_\beta\Gamma^{k-1}\phi)\}
\end{aligned}
\end{equation}
And the third type is  of the form $\Gamma^{j-k}\mathcal P\Gamma^{k-1}\phi$ for some $1\le k \le j$, which can be treated in the same way as in Steps 2--6. We first consider the terms of the first type $(C)$ and let the sum of these terms  be $C(t)$.
We have, by using Propositions~\ref{propAb},~\ref{propcauchy1}, ~\ref{propcauchy2}, ~\ref{propcauchy3} with $r=t$, ~\ref{propl2est}, ~\ref{liest} that
\begin{equation}\label{362}
\begin{aligned}
&\|(I-\mathcal H)C(t)\|_2\lesssim \\& \sum_{|i|\le j}(\|\Gamma^i(A-1)(t)\|_2+\|\Gamma^i(\partial_t+b\cdot\nabla_\bot)b(t)\|_2+\|\Gamma^i b(t)\|_2)\sum_{|i|\le[\frac{j}2]+2\atop\partial=\partial_\alpha,\partial_\beta}|\partial\Gamma^i\phi(t)|_\infty\\&+
 \sum_{|i|\le [\frac{j}2]}(\|\Gamma^i(A-1)(t)\|_2+\|\Gamma^i(\partial_t+b\cdot\nabla_\bot)b(t)\|_2+\|\Gamma^i b(t)\|_2)\sum_{|i|\le j\atop\partial=\partial_\alpha,\partial_\beta}|\partial\Gamma^i\phi(t)|_\infty+
 \\&\sum_{|i|\le [\frac{j}2]+1}\|\Gamma^i b(t)\|_\infty(\sum_{|i|\le j}\|\Gamma^i b(t)\|_2\sum_{|i|\le[\frac{j}2]+2\atop\partial=\partial_\alpha,\partial_\beta}|\partial\Gamma^i\phi(t)|_\infty+
 \sum_{|i|\le [\frac{j}2]}\|\Gamma^i b(t)\|_2\sum_{|i|\le j\atop\partial=\partial_\alpha,\partial_\beta}|\partial\Gamma^i\phi(t)|_\infty)
 \\&
 \lesssim 
 E^{1/2}_j(t)\{(1+\ln t)\sum_{|i|\le [\frac{j}2]+2\atop\partial=\partial_\alpha,\partial_\beta}(|\Gamma^i\partial\chi(t)|_\infty+|\Gamma^i\partial\frak v(t)|_\infty)+E_j^{1/2}(t)\frac 1t\}\\&\times
 \{\sum_{|i|\le j+1\atop\partial=\partial_\alpha,\partial_\beta}|\Gamma^i\partial\chi(t)|_\infty+|\Gamma^i\partial\frak v(t)|_\infty+E_{j+3}^{1/2}(t)(\sum_{|i|\le [\frac{j+1}2]+2\atop\partial=\partial_\alpha,\partial_\beta}|\Gamma^i\partial \chi(t)|_\infty(1+\ln t)+\frac1t)\}
 \end{aligned}
\end{equation}
We also give the estimates of the following two cubic and higher order terms in \eqref{361}. 
First we have  for $\phi=\chi,\ \frak v$, $k=1,\dots, j$, $i=0,1$,
\begin{equation}\label{363}
\begin{aligned}
&\|(I-\mathcal H)\Gamma^{j-k}\{(A-1)(\partial_\beta\Gamma^i\lambda\partial_\alpha\Gamma^{k-1}\phi-\partial_\alpha\Gamma^i\lambda\partial_\beta\Gamma^{k-1}\phi)(t)\}\|_2\lesssim \\&\qquad
E_j^{1/2}(t) \sum_{|i|\le [\frac{j}2]+2\atop\partial=\partial_\alpha,\partial_\beta}(|\Gamma^i\partial\chi(t)|_\infty+|\Gamma^i\partial\frak v(t)|_\infty)^2
\end{aligned}
\end{equation}
Recall definition \eqref{lambda}: $\lambda=\lambda^*-\mathcal K\frak z e_3$. We have for $\phi=\chi,\ \frak v$, $k=1,\dots, j$, $i=0,1$,
\begin{equation}\label{365}
\begin{aligned}
&\|(I-\mathcal H)\Gamma^{j-k}\{\partial_\beta\Gamma^i\mathcal K\frak z e_3\partial_\alpha\Gamma^{k-1}\phi-\partial_\alpha\Gamma^i \mathcal K\frak z e_3    \partial_\beta\Gamma^{k-1}\phi)(t)\}\|_2\lesssim \\&
E_j^{1/2}(t) \sum_{|i|\le [\frac{j}2]+2\atop\partial=\partial_\alpha,\partial_\beta}(|\Gamma^i\partial\chi(t)|_\infty+|\Gamma^i\partial\frak v(t)|_\infty)^2(1+\ln t)\\&+E_j(t)\frac 1t \sum_{|i|\le [\frac{j}2]+2\atop\partial=\partial_\alpha,\partial_\beta}(|\Gamma^i\partial\chi(t)|_\infty+|\Gamma^i\partial\frak v(t)|_\infty)
\end{aligned}
\end{equation}
Therefore the only terms in \eqref{361} that are left to be estimated are the following
\begin{equation}\label{364}
\begin{aligned}
&\Gamma^{j-k}\{\partial_\beta\Gamma^i\lambda^*\partial_\alpha\Gamma^{k-1}\phi-\partial_\alpha\Gamma^i\lambda^*\partial_\beta\Gamma^{k-1}\phi\}\\&
\Gamma^{j-k}\{\partial_\beta\lambda^* e_3\partial_\alpha\Gamma^{k-1}\phi-\partial_\alpha\lambda^* e_3\partial_\beta\Gamma^{k-1}\phi\},
\quad i=0,1,\  k=1,\dots, j.
\end{aligned}
\end{equation}

Step 5. We consider the term $(I-\mathcal H)\Gamma^j(u_\beta\chi_\alpha-u_\alpha\chi_\beta)$ in $G_{j,3}^\frak v$, the term 
$\iint\{ \Phi^j\cdot(u_\beta\Phi^j_\alpha-u_\alpha\Phi^j_\beta)\}\,d\alpha\,d\beta$ for $\phi=\frak v$ in \eqref{340}, and those terms in \eqref{364} for $\phi=\frak v$.
 Without loss of generality, for terms in \eqref{364} with $\phi=\frak v$, we will only write for  $$\Gamma^{j-k}\{\partial_\beta\Gamma\lambda^*\partial_\alpha\Gamma^{k-1}\frak v-\partial_\alpha\Gamma\lambda^*\partial_\beta\Gamma^{k-1}\frak v\}.$$

Using \eqref{rotation}, we rewrite 
\begin{equation}\label{366}
\begin{aligned}
 \partial_{\beta}&u\partial_{\alpha}\chi-\partial_{\alpha} u\partial_{\beta}\chi=\frac2t\{\Upsilon u\,\partial_t (e_2\partial_{\alpha}-e_1\partial_{\beta}) \chi\\&+ \partial_{\beta}u\,\Omega^-_{01}(e_2\partial_{\alpha}-e_1\partial_{\beta})\chi- \partial_{\alpha}u\, \Omega^-_{02}(e_2\partial_{\alpha}-e_1\partial_{\beta})\chi\},
\end{aligned}
\end{equation}
\begin{equation}\label{367}
\begin{aligned}
 \partial_{\beta}&u\partial_{\alpha}\Phi^j-\partial_{\alpha} u\partial_{\beta}\Phi^j=\frac2t\{\Upsilon u\,\partial_t (e_2\partial_{\alpha}-e_1\partial_{\beta}) \Phi^j\\&+ \partial_{\beta}u\,\Omega^-_{01}(e_2\partial_{\alpha}-e_1\partial_{\beta})\Phi^j- \partial_{\alpha}u\, \Omega^-_{02}(e_2\partial_{\alpha}-e_1\partial_{\beta})\Phi^j\},
\end{aligned}
\end{equation}
and
\begin{equation}\label{368}
\begin{aligned}
\partial_{\alpha}&\Gamma\lambda^*\partial_{\beta}\Gamma^{k-1}\frak v- \partial_{\beta}\Gamma\lambda^*\partial_{\alpha}\Gamma^{k-1}\frak v=\frac2t\{-\partial_t (e_2\partial_{\alpha}-e_1\partial_{\beta}) \Gamma\lambda^*\,\Upsilon \Gamma^{k-1}\frak v\\&+\Omega_{01}^+(e_2\partial_{\alpha}-e_1\partial_{\beta})\Gamma\lambda^*\,\partial_{\beta}\Gamma^{k-1}\frak v-  \Omega_{02}^+(e_2\partial_{\alpha}-e_1\partial_{\beta})\Gamma\lambda^*\,\partial_{\alpha}\Gamma^{k-1}\frak v\}
\end{aligned}
\end{equation}
Notice that $[\Omega^\pm_{01}, e_2\partial_\alpha-e_1\partial_\beta]= \mp e_2\partial_t$, $[\Omega^\pm_{02}, e_2\partial_\alpha-e_1\partial_\beta]= \pm e_1\partial_t$.
Using Propositions~\ref{propl2est},~\ref{liest} and \eqref{eqbasicoj},\eqref{comgp}, we get
\begin{equation*}
\begin{aligned}
&\|(I-\mathcal H)\Gamma^j(\partial_{\beta}u\partial_{\alpha}\chi-\partial_{\alpha} u\partial_{\beta}\chi)(t)\|_2\lesssim \\&\frac 1t(E_{j+1}^{1/2}(t)+t\sum_{|i|\le j}\|\Gamma^i\frak P^-\chi(t)\|_2)\sum_{|i|\le [\frac{j}2]+2\atop\partial=\partial_\alpha,\partial_\beta}(|\partial \Gamma^i\chi(t)|_\infty+|\partial \Gamma^i\frak v(t)|_\infty)\\&+\frac 1t(E_{j}^{1/2}(t)+t\sum_{|i|\le [\frac {j}2]}\|\Gamma^i\frak P^-\chi(t)\|_2)
\sum_{|i|\le j\atop\partial=\partial_\alpha,\partial_\beta}(| \Gamma^i\partial_t\partial\chi(t)|_\infty+|\Gamma^i\partial u(t)|_\infty)
\end{aligned}
\end{equation*}
Further applying \eqref{3-10} and \eqref{com0}, we obtain
\begin{equation}\label{369}
\begin{aligned}
&\|(I-\mathcal H)\Gamma^j(\partial_{\beta}u\partial_{\alpha}\chi-\partial_{\alpha} u\partial_{\beta}\chi)(t)\|_2\lesssim \\& 
E^{1/2}_{j+1}(t)(\sum_{|i|\le [\frac{j}2]+2\atop\partial=\partial_\alpha,\partial_\beta}(|\partial \Gamma^i\chi(t)|_\infty+|\partial \Gamma^i\frak v(t)|_\infty)+\frac 1t)\sum_{|i|\le [\frac{j}2]+2\atop\partial=\partial_\alpha,\partial_\beta}(|\partial \Gamma^i\chi(t)|_\infty+|\partial \Gamma^i\frak v(t)|_\infty)
\\&+
E^{1/2}_{j}(t)(\sum_{|i|\le [\frac{j}2]+2\atop\partial=\partial_\alpha,\partial_\beta}|\partial \Gamma^i\chi(t)|_\infty+|\partial \Gamma^i\frak v(t)|_\infty+\frac 1t)\times\\&(
\sum_{|i|\le j+1\atop\partial=\partial_\alpha,\partial_\beta}|\partial \Gamma^i\chi(t)|_\infty+|\partial \Gamma^i\frak v(t)|_\infty+E^{1/2}_{j+3}(t)
(\sum_{|i|\le [\frac{j}2]+2\atop\partial=\partial_\alpha,\partial_\beta}(|\partial \Gamma^i\chi(t)|_\infty+|\partial \Gamma^i\frak v(t)|_\infty)(1+\ln t)+\frac 1t)
\end{aligned}
\end{equation}
Similarly,
\begin{equation*}
\begin{aligned}
\| (\partial_{\beta}u&\partial_{\alpha}\Phi^j-\partial_{\alpha} u\partial_{\beta}\Phi^j)(t)\|_2\lesssim \\&
\sum_{i=1,2\atop\partial=\partial_\alpha,\partial_\beta}\frac 1t(\|\Upsilon u(t)\|_2|\partial_t\partial\Phi^j(t)|_\infty+|\partial u(t)|_\infty\|\Omega_{0i} (e_2\partial_{\alpha}-e_1\partial_{\beta})  \Phi^j(t)\|_2)
\end{aligned}
\end{equation*}
For $\phi=\frak v$, $\partial=\partial_\alpha,\partial_\beta$, we know 
\begin{equation}
\begin{aligned}
&\partial_t\partial\Phi^j=\partial_t\partial(I-\mathcal H)\Gamma^j\frak v=\partial_t(I-\mathcal H)\partial\Gamma^j\frak v-\partial_t[\partial, \mathcal H]\Gamma^j\frak v\\&=(I-\mathcal H)\partial_t\partial\Gamma^j\frak v-[\partial_t,\mathcal H]\partial\Gamma^j\frak v-\partial_t[\partial, \mathcal H]\Gamma^j\frak v
\end{aligned}
\end{equation}
 Therefore using Propositions~\ref{propcauchy3},~\ref{propsobolev},~\ref{propl2est}, we have
 \begin{equation}
|\partial_t\partial\Phi^j(t)|_\infty\lesssim \sum_{i\le j+2\atop\partial=\partial_\alpha,\partial_\beta}(|\Gamma^i\partial v(t)|_\infty(1+\ln t) +E^{1/2}_{j+2}(t)\frac1t)
 \end{equation}
Using further Propositions~\ref{propl2est}, ~\ref{liest},~\ref{propbasicoj}, we obtain
\begin{equation}
 \begin{aligned}
\| (\partial_{\beta}u&\partial_{\alpha}\Phi^j-\partial_{\alpha} u\partial_{\beta}\Phi^j)(t)\|_2\lesssim \\&\frac 1tE_j^{1/2}(t)
\sum_{i\le j+2\atop\partial=\partial_\alpha,\partial_\beta}(|\Gamma^i\partial v(t)|_\infty(1+\ln t) +E^{1/2}_{j+2}(t)\frac1t)\\&+
\sum_{i\le 2\atop\partial=\partial_\alpha,\partial_\beta}(|\Gamma^i\partial \chi(t)|_\infty+|\Gamma^i\partial v(t)|_\infty)(\frac 1t E^{1/2}_{j+2}(t)+\|\frak P^-\Phi^j(t)\|_2)
\end{aligned}
\end{equation}
We know from \eqref{313},\eqref{314}, \eqref{319} the estimate of $\|\mathcal P^-\Phi^j(t)\|_2=\|G_j^\frak v(t)\|_2$. Using further \eqref{371}, Propositions~\ref{propl2est},~\ref{liest}, we have
\begin{equation}
\begin{aligned}
\|\frak P^-\Phi^j(t)\|_2\lesssim &E_j^{1/2}(t)
\sum_{i\le [\frac{ j}2]+2\atop\partial=\partial_\alpha,\partial_\beta}(|\Gamma^i\partial \chi(t)|_\infty+|\Gamma^i\partial v(t)|_\infty)\\&+
\sum_{i\le 4\atop\partial=\partial_\alpha,\partial_\beta}(|\Gamma^i\partial \chi(t)|_\infty+|\Gamma^i\partial v(t)|_\infty) E^{1/2}_{j+2}(t)
\end{aligned}
\end{equation}
Therefore
\begin{equation}\label{373}
\begin{aligned}
&\| (\partial_{\beta}u\partial_{\alpha}\Phi^j-\partial_{\alpha} u\partial_{\beta}\Phi^j)(t)\|_2\lesssim
\\&\frac 1tE_j^{1/2}(t)
\sum_{i\le j+2\atop\partial=\partial_\alpha,\partial_\beta}(|\Gamma^i\partial v(t)|_\infty(1+\ln t) +E^{1/2}_{j+2}(t)\frac1t)\\&+
\sum_{i\le 2\atop\partial=\partial_\alpha,\partial_\beta}(|\Gamma^i\partial \chi(t)|_\infty+|\Gamma^i\partial v(t)|_\infty) E^{1/2}_{j+2}(t)
(\frac 1t+\sum_{i\le  [\frac{ j}2]+2   \atop\partial=\partial_\alpha,\partial_\beta}|\Gamma^i\partial \chi(t)|_\infty+|\Gamma^i\partial v(t)|_\infty) 
\end{aligned}
\end{equation}
The estimate for $\|(I-\mathcal H)\Gamma^{j-k}\{\partial_\beta\Gamma\lambda^*\partial_\alpha\Gamma^{k-1}\frak v-\partial_\alpha\Gamma\lambda^*\partial_\beta\Gamma^{k-1}\frak v\}(t)\|_2$ is similar: we have from \eqref{368} that for $k=1,\dots, j$,
\begin{equation*}
\begin{aligned}
&\|(I-\mathcal H)\Gamma^{j-k}\{\partial_\beta\Gamma\lambda^*\partial_\alpha\Gamma^{k-1}\frak v-\partial_\alpha\Gamma\lambda^*\partial_\beta\Gamma^{k-1}\frak v\}(t)\|_2\lesssim \\&\frac1t E_j^{1/2}(t)\sum_{i\le j+1\atop\partial=\partial_\alpha,\partial_\beta}|\Gamma^i\partial\lambda^*(t)|_\infty+
\frac 1t(E_{j+1}^{1/2}(t)+t\sum_{|i|\le j-1}\|\Gamma^i\frak P^-\Gamma\lambda^*(t)\|_2)\sum_{|i|\le [\frac{j}2]+2\atop\partial=\partial_\alpha,\partial_\beta}|\partial \Gamma^i\frak v(t)|_\infty\\&+\frac 1t(E_{j}^{1/2}(t)+t\sum_{|i|\le [\frac {j}2]-1}\|\Gamma^i\frak P^-\Gamma\lambda^*(t)\|_2)
\sum_{|i|\le j-1\atop\partial=\partial_\alpha,\partial_\beta}|\Gamma^i\partial \frak v(t)|_\infty
\end{aligned}
\end{equation*}
Notice that
$$\frak P^+\Gamma\lambda^*=(\frak P^+-\mathcal P^+)\Gamma\lambda^*+[\mathcal P^+,\Gamma]\lambda^*+\Gamma\mathcal P^+\lambda^*$$
Using \eqref{quasi2}, \eqref{371},\eqref{commgp}, Propositions~\ref{propcauchy1}, ~\ref{propcauchy2}, ~\ref{propcauchy3}, ~\ref{propsobolev}, ~\ref{propl2est}, ~\ref{liest}, we get for $k\le j-1$,
\begin{equation}\label{374}
\|\Gamma^{k}\frak P^+\Gamma\lambda^*(t)\|_2\lesssim E_{k+2}^{1/2}(t)(\sum_{i\le [\frac j2]+2\atop\partial=\partial_\alpha,\partial_\beta}(|\Gamma^i\partial \chi(t)|_\infty+|\Gamma^i\partial v(t)|_\infty)(1+\ln t) + E^{1/2}_{j}(t)\frac1t)
\end{equation}
Therefore
\begin{equation}\label{375}
\begin{aligned}
&\|(I-\mathcal H)\Gamma^{j-k}\{\partial_\beta\Gamma\lambda^*\partial_\alpha\Gamma^{k-1}\frak v-\partial_\alpha\Gamma\lambda^*\partial_\beta\Gamma^{k-1}\frak v\}(t)\|_2\lesssim \\&\frac1t E_j^{1/2}(t)\{\sum_{i\le j+1\atop\partial=\partial_\alpha,\partial_\beta}|\Gamma^i\partial\chi(t)|_\infty+E_{j+3}^{1/2}(t)(\sum_{i\le [\frac {j+1}2]+2\atop\partial=\partial_\alpha,\partial_\beta}|\Gamma^i\partial\chi(t)|_\infty (1+\ln t)+\frac 1t)\}\\&+(\sum_{i\le [\frac j2]+2\atop\partial=\partial_\alpha,\partial_\beta}(|\Gamma^i\partial\chi(t)|_\infty +|\Gamma^i\partial v(t)|_\infty)(1+\ln t)+\frac 1t)\\&\quad\quad\times(E_j^{1/2}\sum_{|i|\le j-1\atop \partial=\partial_\alpha,\partial_\beta}|\Gamma^i\partial\frak v(t)|_\infty+ E_{j+1}^{1/2}\sum_{|i|\le  [\frac j2]+2 \atop \partial=\partial_\alpha,\partial_\beta}|\Gamma^i\partial\frak v(t)|_\infty)
\end{aligned}
\end{equation}

Step 6. Finally, we consider the term 
$\iint\{ \Phi^j\cdot(u_\beta\Phi^j_\alpha-u_\alpha\Phi^j_\beta)\}\,d\alpha\,d\beta$ for $\phi=\chi$ in \eqref{340}, and the remaining terms
\begin{equation}\label{376}
\begin{aligned}
&\Gamma^{j-k}\{\partial_\beta\Gamma^i\lambda^*\partial_\alpha\Gamma^{k-1}\chi-\partial_\alpha\Gamma^i\lambda^*\partial_\beta\Gamma^{k-1}\chi\}\\&
\Gamma^{j-k}\{\partial_\beta\lambda^* e_3\partial_\alpha\Gamma^{k-1}\chi-\partial_\alpha\lambda^* e_3\partial_\beta\Gamma^{k-1}\chi\}
\quad i=0,1,\  k=1,\dots, j
\end{aligned}
\end{equation}
in \eqref{364}. Without loss of generality, among terms in \eqref{376}, we will only write for  $$\Gamma^{j-k}\{\partial_\beta\Gamma\lambda^*\partial_\alpha\Gamma^{k-1}\chi-\partial_\alpha\Gamma\lambda^*\partial_\beta\Gamma^{k-1}\chi\}.$$ Notice that the ideas as that in \eqref{367},\eqref{368} doesn't work here, since we do not have estimates for $\|\Phi^j(t)\|_2$ for $\phi=\chi$, and for  $\|\Upsilon \Gamma^{k-1}\chi(t)\|_2$. We resolve these issue by using commutators. 

We first consider $\iint\{ \Phi^j\cdot(u_\beta\Phi^j_\alpha-u_\alpha\Phi^j_\beta)\}\,d\alpha\,d\beta$ for $\phi=\chi$. Using integration by parts, we have
$$\iint\{ \Phi^j\cdot(u_\beta\Phi^j_\alpha-u_\alpha\Phi^j_\beta)\}\,d\alpha\,d\beta=-\iint\{ \Phi^j_\beta\cdot (u\Phi^j_\alpha)-\Phi^j_\alpha\cdot(u\Phi^j_\beta)\}\,d\alpha\,d\beta$$
Now for $\phi=\chi$, we have, by using Proposition~\ref{prophdual},
\begin{equation}\label{377}
\begin{aligned}
&\iint\{ \partial_\alpha(I-\mathcal H)\Gamma^j\chi\cdot(u\Phi^j_\beta)- \partial_\beta(I-\mathcal H)\Gamma^j\chi  \cdot (u\Phi^j_\alpha)\}\,d\alpha\,d\beta\\&=-
\iint\{ [\partial_\alpha,\mathcal H]\Gamma^j\chi\cdot(u\Phi^j_\beta)-[\partial_\beta,\mathcal H]\Gamma^j\chi
\cdot (u\Phi^j_\alpha)\}\,d\alpha\,d\beta\\&
+
\iint\{  (\mathcal H^*-\mathcal H)\partial_\alpha\Gamma^j\chi\cdot(u\Phi^j_\beta)-  (\mathcal H^*-\mathcal H)  \partial_\beta\Gamma^j\chi
\cdot (u\Phi^j_\alpha)\}\,d\alpha\,d\beta\\&
+
\iint\{ \partial_\alpha\Gamma^j\chi\cdot((I-\mathcal H)(u\Phi^j_\beta))- \partial_\beta\Gamma^j\chi  \cdot ((I-\mathcal H)(u\Phi^j_\alpha))\}\,d\alpha\,d\beta
\end{aligned}
\end{equation}
We further rewrite the term $\iint\{ \partial_\alpha\Gamma^j\chi\cdot((I-\mathcal H)(u\Phi^j_\beta))- \partial_\beta\Gamma^j\chi  \cdot ((I-\mathcal H)(u\Phi^j_\alpha))\}\,d\alpha\,d\beta
$. We know $\mathcal H u= u$. Let $\Phi^j_i$ be the $e_i$ component of $\Phi^j$, for $i=0,\dots, 3$. We have
\begin{equation}\label{378}
\begin{aligned}
& \partial_\alpha\Gamma^j\chi\cdot\{(I-\mathcal H)(u\,\partial_\beta\Phi^j_i e_i)\}- \partial_\beta\Gamma^j\chi  \cdot \{(I-\mathcal H)(u\,\partial_\alpha\Phi^j_i e_i)\}\\&
=\partial_\alpha\Gamma^j\chi\cdot([ \partial_\beta\Phi^j_i ,\mathcal H]u e_i)- \partial_\beta\Gamma^j\chi  \cdot ([\partial_\alpha\Phi^j_i,\mathcal H]u e_i)\\&
=\frac2t \partial_t (e_2\partial_{\alpha}-e_1\partial_{\beta}) \Gamma^j\chi\cdot\{
[\Upsilon \Phi^j_i,\mathcal H]u e_i-[\alpha,\mathcal H]( u     \partial_\beta\Phi^j_i e_i)+ [\beta,\mathcal H] (u     \partial_\alpha\Phi^j_i e_i)\}
\\&+\frac 2t\{ \Omega_{01}^-(e_2\partial_{\alpha}-e_1\partial_{\beta})\Gamma^j\chi\cdot [\partial_\beta\Phi^j_i,\mathcal H]u e_i-
 \Omega_{02}^-(e_2\partial_{\alpha}-e_1\partial_{\beta})\Gamma^j\chi\cdot [\partial_\alpha\Phi^j_i,\mathcal H]u e_i\}
\end{aligned}
\end{equation}
Further applying integration by parts gives us:
\begin{equation}\label{379}
\begin{aligned}
&\iint \partial_t (e_2\partial_{\alpha}-e_1\partial_{\beta}) \Gamma^j\chi\cdot\{
[\Upsilon \Phi^j_i,\mathcal H]u e_i-[\alpha,\mathcal H]( u     \partial_\beta\Phi^j_i e_i)+ [\beta,\mathcal H] (u     \partial_\alpha\Phi^j_i e_i)\}
\,d\alpha\,d\beta\\&=-
\iint  e_2\partial_t\Gamma^j\chi\cdot\partial_{\alpha}\{
[\Upsilon \Phi^j_i,\mathcal H]u e_i-[\alpha,\mathcal H]( u     \partial_\beta\Phi^j_i e_i)+ [\beta,\mathcal H] (u     \partial_\alpha\Phi^j_i e_i)\}
\,d\alpha\,d\beta\\&+
\iint  e_1\partial_t \Gamma^j\chi\cdot \partial_{\beta}\{
[\Upsilon \Phi^j_i,\mathcal H]u e_i-[\alpha,\mathcal H]( u     \partial_\beta\Phi^j_i e_i)+ [\beta,\mathcal H] (u     \partial_\alpha\Phi^j_i e_i)\}
\,d\alpha\,d\beta.
\end{aligned}
\end{equation}
Through \eqref{377}-\eqref{379} we have  put $\iint\{ \Phi^j\cdot(u_\beta\Phi^j_\alpha-u_\alpha\Phi^j_\beta)\}\,d\alpha\,d\beta$ for $\phi=\chi$ in the right form for estimates.  Using Propositions~\ref{propcauchy1}, ~\ref{propl2est}, Lemma~\ref{propbasicoj},  we obtain
\begin{equation*}
\begin{aligned}
&|\iint\{ \Phi^j\cdot(u_\beta\Phi^j_\alpha-u_\alpha\Phi^j_\beta)\}\,d\alpha\,d\beta|\lesssim \sum_{\partial=\partial_\alpha,\partial_\beta}\{(|\partial\lambda(t)|_\infty+|\partial\frak z(t)|_\infty)E_j(t)|\partial\Phi^j(t)|_\infty\\&+\frac 1t E_j^{1/2}(t) E_{j+1}^{1/2}(t)|\partial\Phi^j(t)|_\infty+\|\frak P^-\Gamma^j\chi(t)\|_2E_j^{1/2}(t)|\partial\Phi^j(t)|_\infty+\frac 1t E_j(t) |\partial\Upsilon\Phi^j(t)|_\infty\}
\end{aligned}
\end{equation*}
We know for $\phi=\chi$, $\partial=\partial_\alpha,\partial_\beta$,
$\partial\Phi^j=(I-\mathcal H)\partial\Gamma^j\chi-[\partial,\mathcal H]\Gamma^j\chi$. Using Proposition~\ref{propcauchy3},~\ref{propsobolev},~\ref{propl2est}, we have
$$|\partial\Phi^j(t)|_\infty+|\Upsilon\partial\Phi^j(t)|_\infty\lesssim \sum_{i\le j+2}|\Gamma^i\partial\chi(t)|_\infty(1+\ln t)+E_{j+1}^{1/2}(t)\frac 1t$$
Therefore by further using \eqref{3-10}, we arrive at
\begin{equation}\label{380}
\begin{aligned}
|\iint\{ \Phi^j\cdot(u_\beta\Phi^j_\alpha&-u_\alpha\Phi^j_\beta)\}\,d\alpha\,d\beta|\lesssim  (\sum_{i\le j+2}|\Gamma^i\partial\chi(t)|_\infty(1+\ln t)+E_{j+1}^{1/2}(t)\frac 1t)\\&\qquad\times E_j^{1/2} E_{j+1}^{1/2}
(\sum_{|i|\le [\frac j2]+2\atop\partial=\partial_\alpha,\partial_\beta}(|\Gamma^i\partial\chi(t)|_\infty+|\Gamma^j\partial\frak v(t)|_\infty)+\frac{1}t)
\end{aligned}
\end{equation}

At last we give the estimate of $\|(I-\mathcal H)(\Gamma^{j-k}\{\partial_\beta\Gamma\lambda^*\partial_\alpha\Gamma^{k-1}\chi-\partial_\alpha\Gamma\lambda^*\partial_\beta\Gamma^{k-1}\chi\})(t)\|_2$ for $k=1,\dots,j$.  We first rewrite, 
\begin{equation}\label{381}
\begin{aligned}
(I-&\mathcal H)(\Gamma^{j-k}\{\partial_\beta\Gamma\lambda^*\partial_\alpha\Gamma^{k-1}\chi-\partial_\alpha\Gamma\lambda^*\partial_\beta\Gamma^{k-1}\chi\})\\&=\Gamma^{j-k}(I-\mathcal H)\{\partial_\beta\Gamma\lambda^*\partial_\alpha\Gamma^{k-1}\chi-\partial_\alpha\Gamma\lambda^*\partial_\beta\Gamma^{k-1}\chi\}\\&+[\Gamma^{j-k}, \mathcal H]\{\partial_\beta\Gamma\lambda^*\partial_\alpha\Gamma^{k-1}\chi-\partial_\alpha\Gamma\lambda^*\partial_\beta\Gamma^{k-1}\chi\}
\end{aligned}
\end{equation}
Let $\Gamma^{k-1}\chi_i$ be the $e_i$ component of $\Gamma^{k-1}\chi$. Using the fact $\mathcal H\lambda^*=\lambda^*$, we rewrite further   
\begin{equation}\label{382}
\begin{aligned}
(I-&\mathcal H)\{\partial_\beta\Gamma\lambda^*\partial_\alpha\Gamma^{k-1}\chi_i-\partial_\alpha\Gamma\lambda^*\partial_\beta\Gamma^{k-1}\chi_i\}\\&=\partial_\alpha\Gamma^{k-1}\chi_i[\partial_\beta\Gamma, \mathcal H]\lambda^*-\partial_\beta\Gamma^{k-1}\chi_i[\partial_\alpha\Gamma,\mathcal H]\lambda^*\\&+
[\partial_\alpha\Gamma^{k-1}\chi_i, \mathcal H-\mathcal H^*]\partial_\beta\Gamma\lambda^*-[\partial_\beta\Gamma^{k-1}\chi_i,\mathcal H-\mathcal H^*]\partial_\alpha\Gamma\lambda^*\\&+[\partial_\alpha\Gamma^{k-1}\chi_i, \mathcal H^*]\partial_\beta\Gamma\lambda^*-[\partial_\beta\Gamma^{k-1}\chi_i,\mathcal H^*]\partial_\alpha\Gamma\lambda^*
\end{aligned}
\end{equation}
and in which we rewrite $[\partial_\alpha\Gamma^{k-1}\chi_i, \mathcal H^*]\partial_\beta\Gamma\lambda^*-[\partial_\beta\Gamma^{k-1}\chi_i,\mathcal H^*]\partial_\alpha\Gamma\lambda^*$ using the idea of \eqref{rotation}:
\begin{equation}\label{383}
\begin{aligned}
-\frac t2&\{[\partial_\alpha\Gamma^{k-1}\chi_i, \mathcal H^*]\partial_\beta\Gamma\lambda^*-[\partial_\beta\Gamma^{k-1}\chi_i,\mathcal H^*]\partial_\alpha\Gamma\lambda^*\}\\&
=(-[\Upsilon \Gamma^{k-1}\chi_i,\mathcal H^*]
+\partial_\beta\Gamma^{k-1}\chi_i[\alpha,\mathcal H^*]-\partial_\alpha\Gamma^{k-1}\chi_i[\beta,\mathcal H^*]) \partial_t (e_2\partial_{\alpha}-e_1\partial_{\beta}) \Gamma\lambda^*
\\&+[\partial_{\beta}\Gamma^{k-1}\chi_i,\mathcal H^*]\Omega_{01}^+(e_2\partial_{\alpha}-e_1\partial_{\beta})\Gamma\lambda^*- [\partial_{\alpha}\Gamma^{k-1}\chi_i,\mathcal H^*] \Omega_{02}^+(e_2\partial_{\alpha}-e_1\partial_{\beta})\Gamma\lambda^*\}
\end{aligned}
\end{equation}
Using \eqref{381}-\eqref{383}, and Propositions~\ref{propcauchy1},~\ref{propcauchy2},~\ref{propl2est},~\ref{liest}, we have
\begin{equation*}
\begin{aligned}
&\|(I-\mathcal H)(\Gamma^{j-k}\{\partial_\beta\Gamma\lambda^*\partial_\alpha\Gamma^{k-1}\chi-\partial_\alpha\Gamma\lambda^*\partial_\beta\Gamma^{k-1}\chi\})(t)\|_2\lesssim \\& E_j^{1/2}(t)\sum_{|i|\le [\frac j2]+2\atop\partial=\partial_\alpha,\partial_\beta}|\Gamma^i\partial\chi(t)|_\infty(\sum_{|i|\le [\frac j2]+2\atop\partial=\partial_\alpha,\partial_\beta}|\Gamma^i\partial\chi(t)|_\infty(1+\ln t)+E_j^{1/2}(t)\frac{1}t)\\&
+
\frac1t E_j^{1/2}(t)\sum_{i\le j\atop\partial=\partial_\alpha,\partial_\beta}|\Gamma^i\partial\chi(t)|_\infty+
\frac 1t(E_{j+1}^{1/2}(t)+t\sum_{|i|\le j-1}\|\Gamma^i\frak P^-\Gamma\lambda^*(t)\|_2)\sum_{|i|\le [\frac{j}2]+2\atop\partial=\partial_\alpha,\partial_\beta}|\partial \Gamma^i\chi(t)|_\infty\\&+\frac 1t(E_{j}^{1/2}(t)+t\sum_{|i|\le [\frac {j}2]-1}\|\Gamma^i\frak P^-\Gamma\lambda^*(t)\|_2)
\sum_{|i|\le j-1\atop\partial=\partial_\alpha,\partial_\beta}|\Gamma^i\partial \chi(t)|_\infty
\end{aligned}
\end{equation*}
Further using \eqref{374}, we arrive at
\begin{equation}\label{384}
\begin{aligned}
\|(I-&\mathcal H)(\Gamma^{j-k}\{\partial_\beta\Gamma\lambda^*\partial_\alpha\Gamma^{k-1}\chi-\partial_\alpha\Gamma\lambda^*\partial_\beta\Gamma^{k-1}\chi\})(t)\|_2\lesssim \\&
(\sum_{i\le [\frac j2]+2\atop\partial=\partial_\alpha,\partial_\beta}(|\Gamma^i\partial\chi(t)|_\infty+|\Gamma^i\partial v(t)|_\infty) (1+\ln t)+\frac 1t)\\&\times\quad(E_j^{1/2}\sum_{|i|\le j\atop \partial=\partial_\alpha,\partial_\beta}|\Gamma^i\partial\chi(t)|_\infty+ E_{j+1}^{1/2}\sum_{|i|\le  [\frac j2]+2 \atop \partial=\partial_\alpha,\partial_\beta}|\Gamma^i\partial\chi(t)|_\infty)
\end{aligned}
\end{equation}

Combine ~\eqref{340},
~\eqref{341},~\eqref{345},~\eqref{346},~\eqref{348},
~\eqref{355},~\eqref{356},~\eqref{358},~\eqref{-364},
~\eqref{359}, 
~\eqref{362}, 
\newline
~\eqref{363}, ~\eqref{365}, ~\eqref{373},~\eqref{375},~\eqref{380},~\eqref{384}, notice that for $l\ge 15$, $q\ge l+9$,
$[\frac l2]+3+5\le l+2$ and $l+4\le \min\{2l-11, q-5\}$. Applying \eqref{li}, we obtain \eqref{partenergyineq}.
\end{proof}

\subsection{A conclusive estimate}
We now sum up the results in Lemma~\ref{lemmaliest}, Propositions~\ref{propfullenergy},~\ref{proppartenergy}. 

Let $\epsilon, \, L >0$,  $M_0$ be the constant such that Lemmas~\ref{lemmaliest},~\ref{lemmaef1}, Propositions~\ref{propfullenergy},~\ref{proppartenergy} hold.
Assume that $\frak F_{l+2}(0)\le \epsilon^2, \mathcal F_{l+3}(0)\le \epsilon^2$, and $\mathcal F_{l+9}(0)\le L^2$. 
 \begin{theorem}\label{mainestimate}
 Let $l\ge 17$, $q\ge l+9$, $l\le n\le l+9$. There exists $\varepsilon_0>0$, depending on $M_0$, $L$, such that for  $\epsilon\le \varepsilon_0$, we have 1.
 \begin{equation}\label{polynormialgrowth}
 \mathcal F_n(t)\le \mathcal F_n(0) (1+t)^{1/2},\qquad\text{for }t\in [0, T];
 \end{equation}
 2.  \begin{equation}\label{global}
 \frak F^{1/2}_{l+2}(t)\le (c\, L+1)\epsilon,\qquad\text{for }t\in [0, T],
 \end{equation}
 where $c$ is a constant depending on $M_0$.
\end{theorem}
\begin{proof}
We know for $l\ge 17$, $l\le n\le l+9\le q$, $5\le [\frac n2]+2\le \min\{2l -11, q-5\}$ and $[\frac n2]+7\le l+2$. From Lemmas~\ref{lemmaliest},~\ref{lemmaef1}, Proposition~\ref{propfullenergy}, we get for $t\in [0, T]$,
$$\frac {d\mathcal F_n(t)}{dt}\le c_0(M_0)\frac{E_{l+2}^{1/2}(t)}{1+t}\mathcal F_n(t)\le c_1(M_0)\frac{\frak F_{l+2}^{1/2}(t)}{1+t}\mathcal F_n(t),$$
where $c_0(M_0), c_1(M_0)$ are constants depending on $M_0$. Therefore
\begin{equation}\label{3101}
\mathcal F_n(t)\le \mathcal F_n(0) (1+t)^{M_1(\tau)} \qquad\text{for }t\in [0, \tau],\ \tau\le T,
\end{equation}
where $M_1(\tau)=c_1(M_0)\sup_{[0, \tau]}\frak F_{l+2}^{1/2}(t)$. Applying \eqref{3101}, Lemma~\ref{lemmaef1} to Proposition~\ref{proppartenergy}, we obtain,
\begin{equation}\label{3102}
\frac{d\frak F_{l+2}(t)}{dt}\le c_2(M_0)\,\epsilon L\,\frak F_{l+2}^{1/2}(t) (1+t)^{M_1(\tau)}(\frac {1+\ln (1+t)}{t+1})^2\qquad\text{for $t\in [ 0, \tau]$}
\end{equation}
where $c_2(M_0)$ is a constant depending on $M_0$. 
Let $\varepsilon_0=\min\{\frac1{4c_1(M_0)}, \frac 1{2c_1(M_0)(Lc+1)}\}$, where $c=c_2(M_0)\int_0^\infty (1+t)^{-3/2}(1+\ln (1+t))^2\,dt$, and $\epsilon\le \varepsilon_0$. Therefore $c_1(M_0)\frak F_{l+2}^{1/2}(0)\le \frac 14$. 
 Let $0< T_1\le T$ be the largest such that $M_1(T_1)\le \frac 12$. From \eqref{3102} we get
$$ \frak F_{l+2}^{1/2}(t)\le \frac 12 \epsilon L c+\epsilon\qquad\text{for }t\in [0, T_1]$$
This implies $M_1(T_1)\le c_1(M_0) (\frac 12  L c+1)\varepsilon_0<\frac 12$. So $T_1=T$ or otherwise $T_1$ is not the largest. Therefore \eqref{polynormialgrowth},\eqref{global} holds for $t\in [0, T]$.
\end{proof}

\section{Global wellposedness of the 3D full water wave equation}

In this section we prove that the 3D full water wave equation \eqref{euler}, or equivalently \eqref{ww1}-\eqref{ww2} is uniquely solvable globally in time for small data.  
This is achieved by combining a local wellposedness result for the quasilinear system \eqref{quasi1}-\eqref{at}-\eqref{b} and  Theorem~\ref{mainestimate}.

In what follows all the constants $c(p)$, $c_i(p)$ etc. satisfy $c(p)\le c(p_0)$, $c_i(p)\le c_i(p_0)$  for some $p_0>0$ and all $0\le p\le p_0$.

We first present two lemmas. The first shows that for interface that is a graph small in its steepness (and two more derivatives), the change of coordinate $k$ defined in \eqref{k} is a diffeomorphism. The second gives the regularity relation on quantities
before and after  change of coordinates.

\begin{lemma}\label{diffeo} Let $\xi=(\alpha,\beta, z(\alpha, \beta))$, $k=\xi-(I+\frak H)ze_3+\frak K ze_3$ be defined as in \eqref{k}. 

1. Assume that $N=\sum_{|i|\le 2\atop\partial=\partial_\alpha,\partial_\beta}\|\partial^i \partial z\|_2<\infty$. Then for $\partial=\partial_\alpha,\,\partial_\beta$, 
\begin{equation}\label{-401}
|\partial(k-P)|_\infty\le c(N)N
\end{equation}
 for some constant $c(N)$ depending on $N$. In particular, there exists a $N_0>0$, such that for $N\le N_0$,  we have $|\partial(k-P)|_\infty\le \frac 14$, $\frac 14\le J(k)\le 2$, $k:\mathbb R^2\to \mathbb R^2$ is a diffeomorphism and
$$\frac 14(|\alpha-\alpha'|+|\beta-\beta'|)\le |k(\alpha,\beta)-k(\alpha',\beta')|\le 2 (|\alpha-\alpha'|+|\beta-\beta'|).$$

2. Let $q\ge 5$ be an integer, and $\Gamma=\{\partial_\alpha,\partial_\beta, L_0,\varpi\}$. Assume $\sum_{|j|\le q-1\atop\partial=\partial_\alpha,\partial_\beta}\|\Gamma^j\partial z\|_{H^{1/2}}=L<\infty$. Then 
\begin{equation}\label{-402}
\sum_{|j|\le q-1\atop\partial=\partial_\alpha,\partial_\beta}\|\Gamma^j\partial(k-P)\|_{H^{1/2}}\le c(L)L
\end{equation}
 for some constant $c(L)$ depending on L.
\end{lemma}
\begin{proof}
Notice that for $P=(\alpha,\beta)$, $k-P=-\frak H z e_3+\frak K z e_3$, and for $\partial=\partial_\alpha,\,\partial_\beta$, 
$$\partial(k-P)=-[\partial, \frak H]ze_3-\frak H\partial z e_3+[\partial, \frak K]ze_3+\frak K\partial z e_3.$$
\eqref{-401} follows directly from applying Lemma~\ref{lemma 1.2}, \eqref{comjp}, \eqref{comgk}, Propositions~\ref{propcauchy1},~\ref{propcauchy2},~\ref{propsobolev};  the inequality \eqref{-402} follows with a further application of 
Lemma 6.2  of \cite{wu2} and interpolation. The rest of the statements in Lemma~\ref{diffeo} part 1 follows straightforwardly from \eqref{-401}.
\end{proof}

\begin{lemma}\label{coordinate} Let $q\ge 5$ be an integer, $0<T<\infty$. Assume that for each $t\in [0, T]$,
 $k(\cdot,t):\mathbb R^2\to\mathbb R^2$ is a diffeomorphism, and there are constants $0<c_1,\, c_2,\ \mu_1,\, \mu_2<\infty$, such that $\mu_1\le J(k(t))\le \mu_2$ and $c_1|(\alpha,\beta)-(\alpha',\beta')|\le |k(\alpha,\beta,t)-
k(\alpha',\beta',t)|\le c_2|(\alpha,\beta)-(\alpha',\beta')|$ for $(\alpha,\beta),\, (\alpha',\beta')\in \mathbb R^2$, $t\in [0, T]$. Let $s=0$ or $\frac12$.

1. Let  $\Gamma=\{\partial_\alpha,\,\partial_\beta,\,L_0,\,\varpi\}$, and assume $\sum_{|j|\le q-1\atop\partial=\partial_\alpha,\partial_\beta}\|\Gamma^j\partial (k-P)(0)\|_{H^{s}}\le L<\infty$. Then 
\begin{equation*}
\begin{aligned}
&\sum_{|j|\le q}\|\Gamma^j(f\circ k)(0)\|_{H^{s}}\le c(L)\sum_{|j|\le q}\|\Gamma^j(f)(0)\|_{H^{s}},\\&\sum_{|j|\le q}\|\Gamma^j(f\circ k^{-1})(0)\|_{H^{s}}\le c(L)\sum_{|j|\le q}\|\Gamma^j(f)(0)\|_{H^{s}}.
\end{aligned}
\end{equation*}

2. Let $\Gamma=\{\partial_t,\, \partial_\alpha,\,\partial_\beta,\,L_0,\,\varpi\}$. Assume for $t\in [0, T]$, 
$\sum_{|j|\le q-1\atop\partial=\partial_\alpha,\partial_\beta}\|\Gamma^j\partial (k-P)(t)\|_{H^{s}}\le L$, $\sum_{|j|\le q-1}\|\Gamma^j k_t(t)\|_{H^{s}}\le L$, $L<\infty$. Then for $t\in [0, T]$,
\begin{equation*}
\begin{aligned}
&\sum_{|j|\le q}\|\Gamma^j(f\circ k)(t)\|_{H^{s}}\le c(L)\sum_{|j|\le q}\|\Gamma^j(f)(t)\|_{H^{s}},\\&\sum_{|j|\le q}\|\Gamma^j(f\circ k^{-1})(t)\|_{H^{s}}\le c(L)\sum_{|j|\le q}\|\Gamma^j(f)(t)\|_{H^{s}}.
\end{aligned}
\end{equation*}
Here $c(L)$ is a constant depending on $L$ and $c_1,c_2,\mu_1,\mu_2$, and need not be the same in different contexts.
\end{lemma}
\begin{proof}
The proof of Lemma~\ref{coordinate} is similar to that of Lemma 5.4 in \cite{wu3}. The main difference is here we use
the relations
$$\alpha\partial_\beta f-\beta\partial_\alpha f=\Upsilon f,\qquad \alpha\partial_\alpha f+\beta\partial_\beta f= L_0 f-\frac 12 t\partial_t f$$
to derive that
$$\nabla_\bot f=\frac{(-\beta,\alpha)}{\alpha^2+\beta^2}\Upsilon f+\frac{(\alpha,\beta)}{\alpha^2+\beta^2}(L_0f-\frac 12 t\partial_t f)$$
and  for $\Gamma=\varpi,\, L_0$, we use the identities
\begin{equation*}
\begin{aligned}
\Upsilon (f\circ k^{-1})&=\partial_\beta k^{-1}\cdot(\alpha\,\nabla f\circ k^{-1})-\partial_\alpha k^{-1}\cdot(\beta\,\nabla f\circ k^{-1}), \\ L_0 (f\circ k^{-1})&=\partial_\alpha k^{-1} \cdot(\alpha\,\nabla f\circ k^{-1})+\partial_\beta k^{-1}\cdot(\beta\,\nabla f\circ k^{-1})+\frac 12 t\partial_t(f\circ k^{-1})
\end{aligned}
\end{equation*}
The proof follows an inductive argument similar to that of Lemma 5.4 of \cite{wu3}, and in the case  $s=\frac 12$, the proof further uses Lemma 6.2  of \cite{wu2} and interpolation. We omit the details.

\end{proof}

We now present a local well-posedness result. Similar to (5.21)-(5.22) in \cite{wu2}, we first rewrite the quasilinear system \eqref{quasi1}-\eqref{at}-\eqref{b} in a format for which local wellposedness can be proved by using  energy estimates and iterative scheme. Let ${\bold n}=\frac{\mathcal N}{|\mathcal N|}$. From \eqref{ww1} we know ${\bold n}=\tilde{\bold n}=\frac{w+e_3}{|w+e_3|}$.
From \eqref{ww1}-\eqref{ww2}, and the fact that 
 $A>0$ for nonself-intersecting interface (i.e. the Taylor sign condition holds, see \cite{wu2}), we know $A|\mathcal N|=|w+e_3|$ and $(A\mathcal N\times \nabla) u= |w+e_3|{\bold n}\cdot \nabla_\xi^+ u$. Let $f=(I-\mathcal H)(U_k^{-1}(\frak a_t N))$. From the fact that $\frak a_t$ is real valued,  we know $U_k^{-1}(\frak a_t N)=-\tilde{\bold n}(I+ \tilde{\mathcal K^*})^{-1}(\Re \{\tilde{\bold n} f\})$, where  $ \tilde{\mathcal K^*} =\Re  \tilde{\bold n}\mathcal H \tilde{\bold n}$.   Therefore \eqref{quasi1}-\eqref{at}-\eqref{b} can be rewritten as the following:
\begin{equation}\label{401}
(\partial_t+b\cdot \nabla_\bot)^2 u+a {\bold n}\cdot \nabla_\xi^+ u=-\tilde {\bold n}(I+\tilde {\mathcal K^*})^{-1}(\Re \{\tilde{\bold n} f\})
\end{equation}
where
\begin{equation}\label{402}
\begin{aligned}
&a=|w+e_3|\qquad \tilde {\bold n}=\frac{w+e_3}{|w+e_3|},\qquad w=(\partial_t+b\cdot\nabla_\bot)u,\\
&b= \frac12(\mathcal H-\overline{\mathcal H})\overline u-[\partial_t+b\cdot\nabla_\bot,\mathcal H]\frak z e_3+ [\partial_t+b\cdot\nabla_\bot,\mathcal K]\frak z e_3+ \mathcal K u_3 e_3\\
&(\partial_t+b\cdot \nabla_\bot)\zeta=u,\qquad \frak U=\frac12( u+\mathcal H u)\\
&f=2\iint
K(\zeta'-\zeta)(w-w')\times (
\zeta'_{\beta'}\partial_{ \alpha'}-
\zeta'_{\alpha'}\partial_{ \beta'})\frak U'\,d\alpha'd\beta'\\&+
\iint K(\zeta'-\zeta)\,\{((u-u')\times
u'_{\beta'})\frak U'_{ \alpha'}-((u-u')\times
u'_{\alpha'})\frak U'_{ \beta'}\}\,d\alpha'd\beta'\\&+
2\iint K(\zeta'-\zeta)\,(u-u')\times(
\zeta'_{\beta'}\partial_{ \alpha'}-
\zeta'_{\alpha'}\partial_{ \beta'})(\partial_t'+b'\cdot\nabla'_\bot)\frak U'\,d\alpha'd\beta'\\&+
\iint
((u'-u)\cdot\nabla)K(\zeta'-\zeta) (u-u')\times(\zeta'_{\beta'}\partial_\alpha'-\zeta'_{\alpha'}\partial_\beta')\frak U'
\,d\alpha'd\beta'.
\end{aligned}
\end{equation}
 \eqref{401}-\eqref{402} is a well-defined quasilinear system. We give in the following  the initial data for \eqref{401}-\eqref{402}. 
As we know the initial data describing the water wave motion should satisfy the compatibility conditions given on pages 464-465 of \cite{wu2}.

Assume that the initial interface $\Sigma(0)$ separates $\mathbb R^3$ into two simply connected, unbounded $C^2$ domains, 
$\Sigma(0)$ approaches the $xy$-plane at infinity, and assume that the water occupies the lower region $\Omega(0)$. 
Take a parameterization for $\Sigma(0): \xi^0=\xi^0(\alpha,\beta), (\alpha,\beta)\in \mathbb R^2$, such that $N_0=\xi_{\alpha}^0\times\xi_\beta^0$ is an outward normal of $\Omega(0)$, $|\xi^0_\alpha\times\xi^0_\beta|\ge \mu$, and $|\xi^0(\alpha,\beta)-\xi^0(\alpha',\beta')|\ge C_0|(\alpha,\beta)-(\alpha',\beta')|$ for $(\alpha,\beta), (\alpha',\beta')\in \mathbb R^2$ and some constants $\mu,\ C_0>0$.  Let 
\begin{equation}\label{406}
\xi(\alpha,\beta, 0)=(x^0, y^0, z^0)=\xi^0(\alpha,\beta), \quad \xi_t(\alpha,\beta,0)=\frak u^0(\alpha,\beta),\quad \xi_{tt}(\alpha,\beta, 0)=\frak w^0(\alpha,\beta)
\end{equation}
Assume that the data in \eqref{406} satisfy the compatibility conditions (5.29)-(5.30) of \cite{wu2}, that is $\frak u^0=\frak H_{\Sigma(0)}\frak u^0$, and
\begin{equation}\label{403}
\frak w^0=-e_3+(\bold n_0\cdot e_3)\bold n_0 - \bold n_0 (I+\mathcal
K_0^*)^{-1}(\Re\{\bold 
n_0\,[\partial_t, 
\frak H_{\Sigma(0)}]\frak u_0+ \frak H^*_{\Sigma(0)}(\bold n_0\times 
e_3)\})
\end{equation}
where $\bold n_0=\frac{ N_0}{|N_0|}$, $\frak H^*_{\Sigma(0)}=\bold n_0\frak H_{\Sigma(0)}\bold n_0$ and $\mathcal K_0^*=\Re \frak H^*_{\Sigma(0)}$.
 Assume that $k(0)=k_0=\xi^0-(I+\frak H_{\Sigma(0)})z^0 e_3+\frak K^0 z^0e_3$, where $\frak K^0=\Re \frak H_{\Sigma(0)}$,  as defined in \eqref{k} is a diffeomorphism with its Jacobian $\nu_1\le J(k_0)\le \nu_2$, and 
$c_1|(\alpha,\beta)-(\alpha',\beta')|\le |k_0(\alpha,\beta)-k_0(\alpha',\beta')|\le c_2|(\alpha,\beta)-(\alpha',\beta')|$ for $(\alpha,\beta), (\alpha',\beta')\in \mathbb R^2$ and some constants $0<\nu_1,\,\nu_2,\, c_1, \, c_2<\infty$. Let
\begin{equation}\label{404}
\zeta(\cdot, 0)=\zeta^0(\cdot)=P+\lambda^0(\cdot),\qquad u(\cdot,0)=u^0(\cdot), \qquad (\partial_t+b\cdot\nabla_\bot)u(\cdot, 0)=w^0(\cdot)
\end{equation}
where
\begin{equation}\label{405}
\lambda^0(\cdot)=\xi^0\circ k_0^{-1}(\cdot)-P,\quad u^0(\cdot)=\frak u^0\circ k_0^{-1}(\cdot),\quad w^0(\cdot)=\frak w^0\circ k_0^{-1}(\cdot)
\end{equation}
 Let $s\ge 5$ be an integer. Assume that for $\Gamma=\partial_\alpha,\,\partial_\beta, \,L_0,\, \varpi$, 
\begin{equation}\label{4066}
\sum_{|j|\le s-1\atop\partial=\partial_\alpha,\partial_\beta}\|\Gamma^j\partial\lambda^0\|_{H^{1/2}}+\|\Gamma^j u^0\|_{H^{1/2}}+\|\Gamma^j\partial u^0\|_{H^{1/2}}+ \|\Gamma^j w^0\|_{L^2}+\|\Gamma^j\partial w^0\|_{L^2} <\infty
\end{equation}
We have the following local well-posedness result for the initial value problem \eqref{401}-\eqref{402}-\eqref{404} with a non-blow-up criteria.
\begin{theorem}[Local existence]\label{local}
1. There exists $T>0$, depending on the norm of the initial data, so that the initial value problem \eqref{401}-\eqref{402}-\eqref{404} has a unique solution $(u, \zeta)=(u(\alpha,\beta, t), \zeta(\alpha,\beta, t))$ for $t\in [0, T]$, satisfying for  $|j|\le s-1$, $\Gamma=\partial_\alpha,\partial_\beta, L_0, \varpi$, $\partial=\partial_\alpha,\partial_\beta$,
\begin{equation}\label{regularity}
\Gamma^j\partial (\zeta-P), \Gamma^j u, \Gamma^j\partial u \in C([0, T], H^{1/2}(\mathbb R^2)),\quad \Gamma^j w, \Gamma^j\partial w \in C([0, T], L^2(\mathbb R^2)),
\end{equation}
and $|\zeta_\alpha\times\zeta_\beta|\ge \nu$, $|\zeta(\alpha,\beta, t)-\zeta(\alpha', \beta' ,t)|\ge C_1 |(\alpha,\beta)-(\alpha',\beta')|$ for all $(\alpha,\beta), (\alpha',\beta')\in \mathbb R^2$ and $t\in [0, T]$, for some constants $C_1,\, \nu > 0$. 

Moreover, if $T^*$ is the supremum over all such times $T$, then either $T^*=\infty$ or
\begin{equation}\label{criteria}
\begin{aligned}
\sum_{|j|\le [\frac s2]+3}&\|\Gamma^j w(t)\|_{L^2}+\|\Gamma^j u(t)\|_{L^2}\\&
+\sup_{(\alpha,\beta)\ne (\alpha',\beta')}\frac {|(\alpha,\beta)-(\alpha',\beta')|}{|\zeta(\alpha,\beta,t)-\zeta(\alpha',\beta',t)|}+\big|\frac 1{|(\zeta_\alpha\times\zeta_\beta)(t)|}\big|_{L^\infty}\notin L^\infty[0, T^*)
\end{aligned}
\end{equation}
2. Let 
$P:\mathbb R^2\to \mathbb R^2$ be the identity map: $P(\alpha,\beta)=(\alpha,\beta)$ for $(\alpha,\beta)\in \mathbb R^2$,
\begin{equation}\label{h}
h_t(\cdot, t)=b(h(\cdot, t),t),\qquad h(\cdot, 0)=P(\cdot),
\end{equation}
 and $T<T^*$. Then for $t\in [0, T]$,  $h(\cdot, t): \mathbb R^2\to \mathbb R^2$ exists and is a diffeomorphism, with its Jacobian $c_1(T)\le J(h(t))\le c_2(T)$ for some constants $c_1(T), c_2(T)>0$; and $\xi(\cdot, t)=\zeta(h(k_0(\cdot),t),t)$ is the solution of the water wave system \eqref{ww1}-\eqref{ww2}, satisfying the initial condition \eqref{406}.
 Furthermore, $h(k_0(\cdot),t)=k(\cdot,t)$ for $t\in [0, T^*)$, where $k(\cdot,t)$ is as defined in \eqref{k}, and $\zeta\circ k=\xi$, $u\circ k=\xi_t,\, w\circ k=\xi_{tt}$.
\end{theorem}
\begin{proof} The proof of part 1 is very much the same as that in \cite{wu2}. The main modification is to use the vector fields 
$\Gamma=\partial_\alpha,\partial_\beta, L_0, \varpi$ instead of using only $\partial_\alpha,\partial_\beta$ as in \cite{wu2}, and
use $\partial_t+b\cdot\nabla_\bot$ instead of $\partial_t$. We omit the details.

Let $T<T^*$. Notice that for the solution obtained in part 1, $b=b(\cdot, t)$ is defined for $t\in[0, T]$. Furthermore by applying Lemma~\ref{lemma 1.2}, Proposition~\ref{propcomgh}, \eqref{comjp}, \eqref{comgk}, Propositions~\ref{propcauchy1},~\ref{propcauchy2},~\ref{propsobolev},  Lemma 6.2  of \cite{wu2} and interpolation, and \eqref{regularity}, we know for $\partial=\partial_\alpha,\partial_\beta$,  and $|j|\le s-1$,  $\Gamma^j b,\, \Gamma^j\partial b\in C([0, T], H^{1/2}(\mathbb R^2))$. Therefore for $|j|\le 3$, 
$ \partial^j b\in C([0, T], C(\mathbb R^2)\cap L^\infty(\mathbb R^2))$.  Thus from the classical ODE theory we know \eqref{h}
has a unique solution $h(\cdot, t)$ on $[0, T]$,  $h(\cdot, t): \mathbb R^2\to \mathbb R^2$ is a diffeomorphism with $c_1\le J(h(t))\le c_2$, $c_3|(\alpha,\beta)-(\alpha',\beta')|\le |h(\alpha,\beta,t)-h(\alpha',\beta',t)|\le c_4|(\alpha,\beta)-(\alpha',\beta')|$  for $(\alpha,\beta),\,(\alpha',\beta')\in\mathbb R^2$, 
$t\in [0, T]$ and some constants $0<c_i<\infty$, $i=1,\dots,4$; and $\partial^j (h-P)\in C([0, T], H^{1/2}(\mathbb R^2))$, for $|j|\le s$. 
Let  $\frak u=u\circ h\circ k_0$, $\xi=\zeta\circ h\circ k_0$. From the chain rule we know $\frak u=\xi_t$, and for $t\in [0, T^*)$, $(\frak u, \xi)$ satisfies the quasilinear system (5.21)-(5.22) in \cite{wu2}. Therefore as was proved in \cite{wu2}, $\xi$ solves the water wave system \eqref{ww1}-\eqref{ww2} with initial data satisfying \eqref{406}. Furthermore, for $k$ as defined in \eqref{k}, we know $k_t=(h\circ k_0)_t$ (see \eqref{-140}), and $k(0)=(h\circ k_0)(0)$. Therefore $k(\cdot,t)=h(k_0(\cdot), t)$ for $t\in [0, T^*)$, so $k(t): \mathbb R^2\to \mathbb R^2$ is a diffeomorphism and $J(k(t))>0$ for each $t\in [0, T^*)$. $w\circ k=\xi_{tt}$ follows straightforwardly from the chain rule.

\end{proof}

\begin{remark}\label{remark41}
 Let  $\xi$ be the solution obtained in Theorem~\ref{local}. As a consequence of Theorem~\ref{local} part 2,  we know for  $t\in [0, T^*)$, the mapping $k=k(\cdot,t)$  defined in \eqref{k} is a diffeomorphism and 
 the solution $(u, \zeta)$ for \eqref{401}-\eqref{402}-\eqref{404}   
 coincides with those defined in 
 \eqref{zeta}. 
Let $\lambda=\zeta-P$. Notice that $\partial_t\lambda=u-b-b\cdot\nabla_\bot\lambda$, $\partial_t u= w-b\cdot\nabla_\bot u$. By taking successive derivatives to $t$  to \eqref{quasi1}(or equivalently \eqref{401}), we know that in fact  for  $|j|\le s-1$, and $\Gamma=\partial_t, \,\partial_\alpha,\,\partial_\beta,\, L_0, \,\varpi$, $\partial=\partial_\alpha,\,\partial_\beta$, 
\begin{equation}\label{407}
\begin{aligned}
&\Gamma^j \partial_t\lambda,\Gamma^j\partial\lambda, \Gamma^j u,\Gamma^j\partial_t u,\Gamma^j\partial u\in C([0, T^*), H^{1/2}(\mathbb R^2)),\\&
\Gamma^j w,\Gamma^j\partial_t w,\Gamma^j\partial w\in C([0, T^*), L^2(\mathbb R^2)).
\end{aligned}
\end{equation}
\end{remark}
\begin{remark}\label{remark42}
 Notice that $\eta=\xi(k_0^{-1}(\cdot), t)=\zeta\circ h(\cdot, t)$ is a solution of the water wave equation \eqref{ww1}-\eqref{ww2} with data $\eta(\cdot, 0)=\xi^0\circ k_0^{-1}(\cdot)$, $\eta_t(\cdot,0)=\frak u^0\circ k_0^{-1}(\cdot)$. Let  $|j|\le s-1$, and $\Gamma=\partial_t, \,\partial_\alpha,\,\partial_\beta, \,L_0, \,\varpi$, $\partial=\partial_\alpha,\,\partial_\beta$. Using  \eqref{407}, Lemma~\ref{lemma 1.2}, Proposition~\ref{propcomgh}, \eqref{comjp}, \eqref{comgk}, Propositions~\ref{propcauchy1},~\ref{propcauchy2},~\ref{propsobolev},  and  
Lemma 6.2  of \cite{wu2} and interpolation, we know that the function $b$ defined in \eqref{402} satisfies $\Gamma^j b$, $\Gamma^j\partial b\in C([0, T^*), H^{1/2}(\mathbb R^2))$. Therefore we have for $h$ the solution of \eqref{h}, 
$\Gamma^j (h-P), \Gamma^j\partial_t (h-P), \Gamma^j\partial(h-P)\in C([0, T^*), H^{1/2}(\mathbb R^2))$. This implies the solution $\eta$  satisfies 
\begin{equation}\label{-407}
\begin{aligned}
&\Gamma^j \partial_t\eta, \Gamma^j \partial (\eta-P), \Gamma^j\partial_t\eta_t, \Gamma^j\partial\eta_t\in C([0, T^*), H^{1/2}(\mathbb R^2)),\\&
 \Gamma^j\partial\eta_{tt}, \Gamma^j\partial_t\eta_{tt}\in C([0, T^*), L^2(\mathbb R^2)).
\end{aligned}
\end{equation}

\end{remark}
From Proposition~\ref{propsobolev}, we know there is $N_1>0$ small enough, such that whenever $\sum_{|i|\le 2\atop\partial=\partial_\alpha,\partial_\beta}\|\partial^i\partial\lambda(t)\|_2\le N_1$, 
$|\partial_\alpha\lambda(t)|_\infty+|\partial_\beta\lambda(t)|_\infty\le\frac 14$; this in turn implies that 
$$|\zeta(\alpha,\beta, t)-\zeta(\alpha',\beta',t)|\ge \frac 14(|\alpha-\alpha'|+|\beta-\beta'|),\qquad |\zeta_\alpha\times\zeta_\beta|\ge \frac 14,$$
and $\Sigma(t): \zeta=\zeta(\alpha,\beta,t), \ (\alpha,\beta)\in \mathbb R^2$ is a graph. 

We now present a global in time well-posedness result. Let $s\ge 27$, $\max\{[\frac s2]+1, 17\}\le l\le s-10$, and the initial interface $\Sigma(0)$ be a graph given by $\xi^0=(\alpha,\beta, z^0(\alpha,\beta))$, satisfying $N=\sum_{|i|\le 2\atop\partial=\partial_\alpha,\partial_\beta}\|\partial^i\partial z^0\|_2\le N_0$, where $N_0$ is the constant in Lemma~\ref{diffeo}, part 1.   Therefore the corresponding mapping $k(0)=k_0$ defined in \eqref{k} is a diffeomorphism with its Jacobian $1/4\le J(k_0)\le 2$ and $\frac14(|\alpha-\alpha'|+|\beta-\beta'|)\le |k_0(\alpha,\beta)-k_0(\alpha',\beta')|\le 2(|\alpha-\alpha'|+|\beta-\beta'|)$. Assume that the initial data satisfies \eqref{406}-\eqref{4066}, and for $\Gamma=\partial_\alpha,\partial_\beta, L_0,\varpi$,
\begin{equation}\label{420}
L=\sum_{|j|\le l+9\atop\partial=\partial_\alpha,\partial_\beta}\|\Gamma^j |D|^{1/2} \frak z^0\|_{2}+\|\Gamma^j\partial \lambda^0\|_{2}+\|\Gamma^j  u^0\|_{H^{1/2}}+\|\Gamma^j w^0\|_{2}<\infty.
\end{equation}
here $\frak z^0=z^0\circ k_0^{-1}$. Let
\begin{equation}\label{421}
\epsilon=\sum_{|j|\le l+3\atop\partial=\partial_\alpha,\partial_\beta}\|\Gamma^j |D|^{1/2} \frak z^0\|_{2}+\|\Gamma^j\partial \lambda^0\|_{2}+\|\Gamma^j  u^0\|_{H^{1/2}}+\|\Gamma^j w^0\|_{2}.
\end{equation}
and assume $\epsilon\le N_1$. An argument as that in Remark~\ref{remark41} and an application of Lemma~\ref{lemma 1.2}, Proposition~\ref{propcomgh}, \eqref{comgk}, \eqref{comjp},  Propositions~\ref{propcauchy1},~\ref{propcauchy2},~\ref{propsobolev} gives that for $\Gamma=\partial_t,\,\partial_\alpha,\,\partial_\beta,\,L_0,\,\varpi$, 
\begin{equation}\label{411}
\mathcal M_0=\sum_{|j|\le l+2\atop
\partial=\partial_\alpha,\partial_\beta}(\|\Gamma^j\partial \lambda^0\|_2+\|\Gamma^j\partial \frak z^0\|_2
+\|\Gamma^j\frak v(0)\|_2+\|\Gamma^j(\partial_t+b\cdot\nabla_\bot)\frak v(0)\|_2)\le c_1(\epsilon)\epsilon<\infty
\end{equation}
and a further application of Lemma 6.2  of \cite{wu2} and interpolation gives that (for $\epsilon>0$ small enough such that $c_1(\epsilon)\epsilon\le M_0$)
\begin{equation}\label{412}
\frak F_{l+2}(0)\le c_2(\epsilon)\epsilon^2,\qquad \mathcal F_{l+3}(0)\le c_3(\epsilon)\epsilon^2,\qquad\mathcal F_{l+9}(0)= c_4(L)<\infty
\end{equation}
Here $c_i(p)$, $i=1,2,3,4$ are constants depending on $p$. 

Take $M_0$ such that $0<M_0\le N_1$ and all the estimates derived in section~\ref{energy} holds. 
\begin{theorem}[Global well-posedness]\label{globalexistence}
There exists $\epsilon_0>0$, depending on $M_0$, $L$,  where $L$ is as in \eqref{420}, such that for $0\le \epsilon\le \epsilon_0$, the initial value problem \eqref{ww1}-\eqref{ww2}-\eqref{406} has a unique classical solution globally in time. For $0\le t<\infty$, the solution satisfies \eqref{407}, \eqref{-407}, the interface is a graph, and 
\begin{equation}\label{413}(1+t)\sum_{|j|\le l-3\atop\partial=\partial_\alpha,\partial_\beta}(|\partial\Gamma^j \chi(t)|_\infty+|\partial\Gamma^j\frak v(t)|_\infty)\lesssim \frak F^{1/2}_{l+2}(t)\le C( M_0, L)\epsilon.\end{equation}
Here $C(M_0, L)$ is a constant depending on $M_0, L$.
\end{theorem}
\begin{proof}
From Theorem~\ref{local}, Remarks~\ref{remark41}, ~\ref{remark42}, we know there exists a unique solution $\xi=\xi(\cdot, t)$ for $ t\in [0, T^*)$ of \eqref{ww1}-\eqref{ww2}-\eqref{406}, with $k(\cdot, t):\mathbb R^2\to \mathbb R^2$ as defined in \eqref{k} being a diffeomorphism, $\lambda$, $u$, $w$ as defined in \eqref{zeta}, \eqref{lambda} satisfying \eqref{407} for $t\in [0, T^*)$, and $\eta=\xi\circ k_0^{-1}$ satisfying \eqref{-407}.  Applying Lemma~\ref{lemma 1.2}, Proposition~\ref{propcomgh}, \eqref{comjp}, \eqref{comgk}, Propositions~\ref{propcauchy1},~\ref{propcauchy2},~\ref{propsobolev},
Lemma 6.2  of \cite{wu2} and interpolation, and the fact that 
$\frak z(\cdot,t)=\frak z^0(\cdot)+\int_0^t (u_3- b\cdot\nabla_\bot \frak z)(\cdot, s)\,ds$,
here $u_3$ is the $e_3$ component of $u$,  we have $\mathcal F_n(t),\ \frak F_n(t)\in C^1[0, T^*)$ for $n\le l+9$. 
Let $0<\epsilon_1\le N_1$ be small enough such that for $\epsilon\le \epsilon_1$, $\mathcal M_0\le c_1(\epsilon)\epsilon\le \frac {M_0}2$. 
Let $T_1\le T_*$ be the largest such that for $t\in [0, T_1)$, \eqref{assumm1} holds. From Theorem~\ref{polynormialgrowth},  Lemma~\ref{lemmaef1}, we know there is a $0<\epsilon_2\le \epsilon_1$, such that when $0<\epsilon\le \epsilon_2$,  $\sup_{[0, T_1)}E_{l+2}(t)\lesssim \frak F_{l+2}(t)\le c(M_0, L)^2\epsilon^2$ 
for some constant $c(M_0, L)$ depending on $M_0, L$.  On the other hand from Proposition~\ref{propl2est} we have that for $t\in [0, T_1)$, 
$$\sum_{|j|\le l+2\atop
\partial=\partial_\alpha,\partial_\beta}(\|\Gamma^j\partial \lambda(t)\|_2+\|\Gamma^j\partial \frak z(t)\|_2
+\|\Gamma^j\frak v(t)\|_2+\|\Gamma^j(\partial_t+b\cdot\nabla_\bot)\frak v(t)\|_2)\le C(M_0)E_{l+2}(t)^{1/2}$$
where  $C(M_0)$ is a constant depending on $M_0$. Taking $\epsilon_0\le\epsilon_2$, such that $C(M_0)c(M_0, L)\epsilon_0\le \frac{3M_0}4$. Therefore when $\epsilon\le \epsilon_0$, we have for $t\in [0, T_1)$,
$$\sum_{|j|\le l+2\atop
\partial=\partial_\alpha,\partial_\beta}(\|\Gamma^j\partial \lambda(t)\|_2+\|\Gamma^j\partial \frak z(t)\|_2
+\|\Gamma^j\frak v(t)\|_2+\|\Gamma^j(\partial_t+b\cdot\nabla_\bot)\frak v(t)\|_2)\le   \frac{3M_0}4 $$
This implies that $T_1=T^*$ or otherwise it contradicts with the assumption that $T_1$ is the largest. Applying Proposition~\ref{propl2est} again we deduce that
\begin{equation}
\sum_{|j|\le l+2}\|\Gamma^j w(t)\|_{L^2}+\|\Gamma^j u(t)\|_{L^2}\in L^\infty[0, T^*).
\end{equation}
Furthermore from $M_0\le N_1$ we have
\begin{equation}
\sup_{(\alpha,\beta)\ne (\alpha',\beta')}\frac {|(\alpha,\beta)-(\alpha',\beta')|}{|\zeta(\alpha,\beta,t)-\zeta(\alpha',\beta',t)|}+\big|\frac 1{|\zeta_\alpha\times\zeta_\beta(t)|}\big|_{L^\infty}\in L^\infty[0, T^*).
\end{equation}
and $\Sigma(t): \zeta=\zeta(\cdot,t)$ defines a graph for $t\in [0, T^*)$. 
Now from our assumption  we know $[\frac s2]+3\le l+2$. Applying \eqref{criteria}, we obtain  $T^*=\infty$.  \eqref{413} is a consequence of Lemma~\ref{lemmaliest}. 

\end{proof}
\begin{remark}
As a consequence of \eqref{413} and Proposition~\ref{liest}, the steepness, the acceleration of the interface and the derivative of the velocity on the interface decay at the rate $\frac 1t$.
\end{remark}
\begin{remark}
Instead of \eqref{4066},\eqref{420},\eqref{421}, we may assume
 for $|j|\le s-1$, and $\Gamma=\partial_\alpha,\,\partial_\beta,\, L_0,\, \varpi$, 
\begin{equation}\label{-409-}
\sum_{|j|\le s-1\atop\partial=\partial_\alpha,\partial_\beta}\|\Gamma^j\partial z^0\|_{H^{1/2}}+\|\Gamma^j \frak u^0\|_{H^{3/2}}+\|\Gamma^j\frak w^0\|_{H^1}<\infty;
\end{equation}
\begin{equation}\label{409}
L=\sum_{|j|\le l+9\atop\partial=\partial_\alpha,\partial_\beta}\|\Gamma^j |D|^{1/2} z^0\|_{2}+\|\Gamma^j\partial z^0\|_{2}+\|\Gamma^j \frak u^0\|_{H^{1/2}}+\|\Gamma^j\frak w^0\|_{2}<\infty; 
\end{equation}
and let 
\begin{equation}\label{-409}
\epsilon=\sum_{|j|\le l+3\atop\partial=\partial_\alpha,\partial_\beta}\|\Gamma^j |D|^{1/2} z^0\|_{2}+\|\Gamma^j\partial z^0\|_{2}+\|\Gamma^j \frak u^0\|_{H^{1/2}}+\|\Gamma^j\frak w^0\|_{2}.
\end{equation}
We know from Lemma~\ref{diffeo} and 
Lemma~\ref{coordinate} that \eqref{-409-},\eqref{409},\eqref{-409} implies
\begin{equation}\label{-410}
\sum_{|j|\le s-1\atop\partial=\partial_\alpha,\partial_\beta}\|\Gamma^j\partial \lambda^0\|_{H^{1/2}}+\|\Gamma^j  u^0\|_{H^{3/2}}+\|\Gamma^j w^0\|_{H^1}<\infty\qquad\text{and}
\end{equation}
\begin{equation}\label{410}
\sum_{|j|\le l+9\atop\partial=\partial_\alpha,\partial_\beta}\|\Gamma^j |D|^{1/2} \frak z^0\|_{2}+\|\Gamma^j\partial \lambda^0\|_{2}+\|\Gamma^j  u^0\|_{H^{1/2}}+\|\Gamma^j w^0\|_{2}\le c_5(L)L<\infty
\end{equation}
\begin{equation}\label{410-}
\sum_{|j|\le l+3\atop\partial=\partial_\alpha,\partial_\beta}\|\Gamma^j |D|^{1/2} \frak z^0\|_{2}+\|\Gamma^j\partial \lambda^0\|_{2}+\|\Gamma^j  u^0\|_{H^{1/2}}+\|\Gamma^j w^0\|_{2}\le c_5(\epsilon)\epsilon<\infty
\end{equation}
for some constants $c_5(L)$, $c_6(\epsilon)$  depending on $L$, $\epsilon$ respectively. Therefore the same conclusions of
Theorem~\ref{globalexistence} hold, and furthermore by using Lemmas~\ref{diffeo},~\ref{coordinate}, we have for $\xi=\eta\circ k_0$ the solution of the initial value problem of the water wave equations \eqref{ww1}-\eqref{ww2}-\eqref{406}, and $|j|\le s-1$, $\Gamma=\partial_t,\partial_\alpha,\partial_\beta,L_0,\varpi$, (notice that $k_0=k_0(\alpha,\beta)$ is independent of $t$.)
\begin{equation*}
\begin{aligned}
&\Gamma^j \partial_t\xi, \Gamma^j \partial (\xi-P), \Gamma^j\partial_t\xi_t, \Gamma^j\partial\xi_t\in C([0, T^*), H^{1/2}(\mathbb R^2)),\\&
 \Gamma^j\partial\xi_{tt}, \Gamma^j\partial_t\xi_{tt}\in C([0, T^*), L^2(\mathbb R^2)).
\end{aligned}
\end{equation*}

\end{remark}

\end{document}